\documentclass[11pt,reqno]{amsart}

\usepackage[T1]{fontenc}
\usepackage[utf8]{inputenc}
\usepackage{amsmath,amssymb,amsthm}
\usepackage{xcolor}
\usepackage{tikz}
\usepackage{enumitem}
\usepackage{booktabs}
\usepackage[colorlinks=true,linkcolor=blue!50!black,citecolor=blue!50!black,urlcolor=blue!50!black]{hyperref}
\usepackage{cleveref}
\usepackage{doi}

\crefname{section}{Section}{Sections}       \Crefname{section}{Section}{Sections}
\crefname{subsection}{Section}{Sections}    \Crefname{subsection}{Section}{Sections}
\crefname{appendix}{Appendix}{Appendices}   \Crefname{appendix}{Appendix}{Appendices}
\crefname{figure}{Figure}{Figures}          \Crefname{figure}{Figure}{Figures}
\crefname{table}{Table}{Tables}             \Crefname{table}{Table}{Tables}
\crefname{theorem}{Theorem}{Theorems}       \Crefname{theorem}{Theorem}{Theorems}
\crefname{proposition}{Proposition}{Propositions}
\Crefname{proposition}{Proposition}{Propositions}
\crefname{lemma}{Lemma}{Lemmas}             \Crefname{lemma}{Lemma}{Lemmas}
\crefname{corollary}{Corollary}{Corollaries}
\Crefname{corollary}{Corollary}{Corollaries}
\crefname{definition}{Definition}{Definitions}
\Crefname{definition}{Definition}{Definitions}
\crefname{remark}{Remark}{Remarks}          \Crefname{remark}{Remark}{Remarks}
\crefname{example}{Example}{Examples}       \Crefname{example}{Example}{Examples}

\newcounter{eqdest}
\makeatletter
\AtBeginDocument{%
  \let\sc@theHequation\theHequation
  \renewcommand*{\theHequation}{\sc@theHequation.\arabic{eqdest}}}
\makeatother
\newcommand{\dtag}[1]{\tag{#1}\stepcounter{eqdest}}

\usepackage[noesvect,overrightharpoonup]{overarrows}   

\newcommand{\lcb}{\left\lbrace}       
\newcommand{\rcb}{\right\rbrace}      
\newcommand{\cb}[1]{\lcb #1 \rcb}     
\newcommand{\lab}{\left[}             
\newcommand{\rab}{\right]}            
\newcommand{\ab}[1]{\lab #1 \rab}     
\newcommand{\lb}{\left(}              
\newcommand{\rb}{\right)}             
\newcommand{\br}[1]{\lb #1 \rb}       
\newcommand{\brOf}[1]{\!\br{#1}}      
\newcommand{\abs}[1]{\left| #1 \right|}

\newcommand*{\mc}[1]{\mathcal{#1}}
\newcommand*{\ms}[1]{\mathsf{#1}}

\newcommand{\R}{\mathbb{R}}
\newcommand{\Rp}{[0, \infty)}
\newcommand{\Rpp}{(0, \infty)}

\DeclareMathOperator{\diam}{\mathsf{diam}}

\newcommand{\innerProduct}[2]{\left\langle#1\,,\, #2\right\rangle}
\newcommand{\ip}[2]{\innerProduct{#1}{#2}}

\newcommand{\ind}{\mathbf{1}}         
\newcommand{\dl}{\mathrm{d}}          

\newcommand{\pr}{^\prime}
\newcommand{\prr}{^{\prime\prime}}

\newcommand{\equationFullstop}{\, .}
\newcommand{\eqfs}{\equationFullstop}
\newcommand{\equationComma}{\, ,}
\newcommand{\eqcm}{\equationComma}

\newcommand{\ol}[2]{\overline{#1#2}}            
\newcommand{\oa}[2]{\overrightharpoonup{#1#2}}  
\newcommand{\olt}[2]{\ol{#1}{#2}^2}             
\newcommand{\cato}{\texorpdfstring{$\mathrm{CAT}(0)$}{CAT(0)}}
\newcommand{\catk}{\texorpdfstring{$\mathrm{CAT}(\kappa)$}{CAT(kappa)}}

\newcommand{\tran}{\tau}
\newcommand{\dtran}{\tran\pr}
\newcommand{\ddtran}{\tran\prr}
\newcommand{\ddrtran}{\partial_+\dtran}         
\newcommand{\oltr}[2]{\tran(\ol{#1}{#2})}

\newcommand{\ctransymbol}{L}
\newcommand{\ctranobest}[1]{\ctransymbol_{#1}}  

\newcommand{\setcc}{\mc S}
\newcommand{\setcco}{\mc S_0}

\newcommand{\sqmap}{x^2}

\newcommand{\setC}{\mc C}

\newcommand{\Dk}{D_{\kappa}}              
\newcommand{\Mk}{M^{2}_{\kappa}}          
\newcommand{\chmap}{\varsigma_{\kappa}}   
\newcommand{\ch}[1]{\widehat{#1}}         
\newcommand{\och}[2]{\widehat{#1#2}}      
\newcommand{\Th}{\Theta_{\kappa}}         
\newcommand{\Kap}{\mathrm{K}_{\kappa}}    
\newcommand{\Pidia}{\Pi_{\ms{dia}}}       
\newcommand{\Pileg}{\Pi_{\ms{leg}}}       
\newcommand{\Log}{\operatorname{Log}}
\newcommand{\Tcone}[1]{T_{#1}}            
\newcommand{\Scone}[1]{\Sigma_{#1}}       
\newcommand{\ang}[1]{\angle_{#1}}

\newcommand{\Quad}{\mathsf{Q}}
\newcommand{\Quadt}{\Quad_{\tran}}
\newcommand{\QuadOf}[4]{\Quadt(#1,#2;#3,#4)}

\newcommand{\by}{\bar y}
\newcommand{\bz}{\bar z}
\newcommand{\bq}{\bar q}
\newcommand{\bp}{\bar p}

\newcommand{\Am}{t}          
\newcommand{\Kk}{K}          
\newcommand{\dz}{\delta_0}   
\newcommand{\tent}[2]{T_{#1,#2}}
\newcommand{\Phifun}{\Phi}
\newcommand{\Gfun}{G}
\newcommand{\posp}[1]{(#1)_+}

\newcommand{\seg}[2]{\ab{#1,#2}}

\newcommand{\sleg}{s_{\ms{leg}}}
\newcommand{\sdia}{s_{\ms{dia}}}
\newcommand{\smin}{s_{\ms{min}}}

\newcommand{\leanchecked}{\textnormal{Lean}\,\checkmark}  
\newcommand{\yes}{\ensuremath{\checkmark}}                
\newcommand{\no}{\ensuremath{\times}}

\definecolor{colQuadDiag}{HTML}{D81B60}
\definecolor{colQuadSidePos}{HTML}{1E88E5}
\definecolor{colQuadSideNeg}{HTML}{004D40}

\theoremstyle{plain}
\newtheorem{theorem}{Theorem}[section]
\newtheorem{proposition}[theorem]{Proposition}
\newtheorem{lemma}[theorem]{Lemma}
\newtheorem{corollary}[theorem]{Corollary}
\theoremstyle{definition}
\newtheorem{definition}[theorem]{Definition}
\newtheorem{example}[theorem]{Example}
\theoremstyle{remark}
\newtheorem{remark}[theorem]{Remark}

\makeatletter
\newcommand{\orcidname}{{\itshape ORCID}}
\newcommand{\orcidentry}[2]{\begingroup
    \@ifnotempty{#2}{\nobreak\indent\orcidname
        \@ifnotempty{#1}{, \ignorespaces#1\unskip}\/:\space
        \href{https://orcid.org/#2}{\ttfamily https://orcid.org/#2}\par}%
    \endgroup}
\newcommand{\orcid}[2][]{\g@addto@macro\addresses{\orcidentry{#1}{#2}}}
\makeatother

\begin{document}
    
\title[The trapezoid comparison inequality]
{The trapezoid comparison inequality in metric spaces with curvature bounded above}

\author{Christof Sch\"otz}
\address{Munich Climate Center and Earth System Modelling Group, Department of Aerospace and Geodesy, TUM School of Engineering and Design, Technical University of Munich, Munich, Germany}
\address[Also at]{Research Department 4 - Complexity Science, Potsdam Institute for Climate Impact Research, Potsdam, Germany}
\email{christof.schoetz@tum.de}
\orcid{0000-0003-3528-4544}


\begin{abstract}
    Among four points of a CAT(0) space, the planar symmetric trapezoids are configurations on which Ptolemy's and Reshetnyak's inequalities are both equalities.
    We show that the symmetric trapezoids remain extremal for the entire family of inequalities interpolating between the two, indexed by the nondecreasing convex functions with concave derivative, and that this function class is characterized by
    this property.
    The result is qualitatively stronger than the known quadruple inequalities for this class of functions, and recovers them with their optimal constants, which were previously known only for power functions.
    Moreover, we extend the analysis to CAT($\kappa$) spaces with $\kappa>0$.
    We derive variants of Reshetnyak's quadrilateral comparison and of Ptolemy's inequality under positive upper curvature bounds, each with the optimal constant.
    These two inequalities yield the trapezoid comparison inequality in CAT($\kappa$) spaces, where the product of the bases carries an additional constant factor compared to the $\kappa=0$ case.
    Again the constant is optimal.
\end{abstract}

\maketitle

\section{Introduction}\label{sec:intro}

\subsection{Cauchy--Schwarz, Ptolemy, and symmetric trapezoids} \label{ssec:intro:intro}

Consider four points $y,z,q,p$ of the Euclidean plane, read as a quadrilateral traversed in that order, and write $\oa yz$ for the vector from $y$ to $z$ and $\ol yz$ for the distance between them.
Then $\oa yz$ and $\oa pq$ are one pair of opposite sides, which we call the \emph{bases}; $\oa yp$ and $\oa zq$ are the other pair, the \emph{legs}; and $\oa yq,\oa zp$ are the \emph{diagonals}.
This terminology is borrowed from trapezoids, even though $y,z,q,p$ need not form one.
Two classical inequalities constrain the lengths of these six vectors: The Cauchy--Schwarz inequality applied to the bases is
\begin{equation}\label{eq:intro-cs}
    \ip{\oa yz}{\oa pq} \leq\, \ol yz\,\ol pq
    \eqfs
\end{equation}
It is an equality exactly when $\oa yz$ and $\oa pq$ are parallel and point the same way.
Ptolemy's inequality
\begin{equation}\dtag{P}\label{eq:P}
    \ol yq\,\ol zp \,\leq\, \ol yp\,\ol zq+\ol yz\,\ol pq
    \eqcm
\end{equation}
in which the product of the diagonals is bounded by the sum of the products of the legs and of the bases, is an equality exactly when the four points, in the given cyclic order, lie on a circle or on a line.
Each inequality is therefore tight on its own family---trapezoids for \eqref{eq:intro-cs}, cyclic quadrilaterals for \eqref{eq:P}---and a trapezoid is cyclic exactly when it is \emph{symmetric} (isosceles), so, leaving aside collinear configurations, the two families meet in the symmetric trapezoids.
It stands to reason that symmetric trapezoids are then the extremal configurations for any inequality lying, in some sense, between \eqref{eq:intro-cs} and \eqref{eq:P}.
They are, as our main result shows.

\subsection{Main result}\label{ssec:intro:result}

For any transformation function $\tran \colon \Rp \to\R$ and four points in a metric space, define the \emph{quadruple functional} $\QuadOf yzqp := \oltr yq+\oltr zp-\oltr yp-\oltr zq$ (\cref{def:Q}).

\begin{figure}[tb]
    \centering
    \begin{tikzpicture}[scale=0.44,thick,
        qy/.style={colQuadDiag}, ql/.style={colQuadSideNeg}, qd/.style={colQuadSidePos},
        pt/.style={fill,circle,inner sep=1.0pt}]
        \foreach \px/\py/\qx/\qy [count=\k from 0] in {%
            -0.850/1.810/0.850/1.810,
            1.150/0.000/2.850/0.000,
            0.653/1.108/2.347/1.241,
            -0.617/2.233/0.617/1.063,
            0.178/1.279/0.022/-0.414,
            -0.489/-0.000/-0.511/1.700
        }{
            \pgfmathsetmacro{\cx}{mod(\k,3)*7.0}
            \pgfmathsetmacro{\cy}{-floor(\k/3)*5.2}
            \begin{scope}[shift={(\cx,\cy)}]
                \coordinate (y) at (-1.7,0);   \coordinate (z) at (1.7,0);
                \coordinate (p) at (\px,\py);  \coordinate (q) at (\qx,\qy);
                \draw[qd] (y)--(q); \draw[qd] (z)--(p);
                \draw[ql] (y)--(p); \draw[ql] (z)--(q);
                \draw[qy] (y)--(z); \draw[qy] (p)--(q);
                \node[pt] at (y){}; \node[pt] at (z){};
                \node[pt] at (p){}; \node[pt] at (q){};
                \ifnum\k=0
                \node[below left =-3pt,font=\tiny] at (y) {$\by$};
                \node[below right=-3pt,font=\tiny] at (z) {$\bz$};
                \node[above left =-3pt,font=\tiny] at (p) {$\bp$};
                \node[above right=-3pt,font=\tiny] at (q) {$\bq$};
                \else
                \node[below left =-3pt,font=\tiny] at (y) {$y$};
                \node[below right=-3pt,font=\tiny] at (z) {$z$};
                \node[above left =-3pt,font=\tiny] at (p) {$p$};
                \node[above right=-3pt,font=\tiny] at (q) {$q$};
                \fi
            \end{scope}
        }
        \begin{scope}[shift={(7.5,-8.2)},font=\scriptsize]
            \draw[qy] (-4.6,0)--(-4.0,0) node[right=1pt,black]{bases};
            \draw[ql] (-1.4,0)--(-0.8,0) node[right=1pt,black]{legs};
            \draw[qd] ( 1.4,0)--( 2.0,0) node[right=1pt,black]{diagonals};
        \end{scope}
    \end{tikzpicture}
    \caption{Quadrilaterals with the same two base lengths and the same average leg length $s=\tfrac12(\ol yp+\ol zq)$; the first is the comparison trapezoid, whose two diagonals both have length $\Am=\sqrt{s^2+\ol yz\,\ol pq}$.
        The pictures are planar only for legibility: the main theorem applies to any four points of any \cato{} space.}
    \label{fig:trapezoid}
\end{figure}
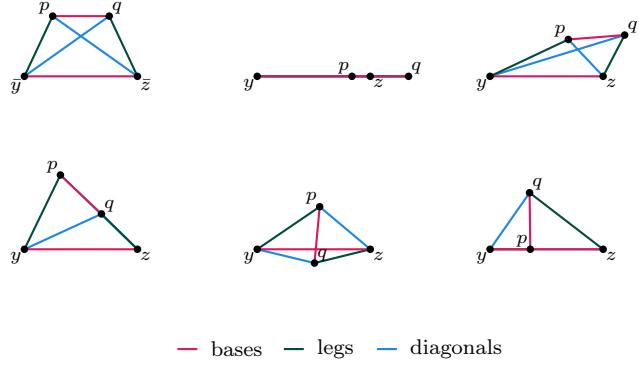
Intuitively, if $\tran$ is nondecreasing, $\QuadOf yzqp$ is maximized by making the leg lengths $\ol yp$, $\ol zq$ small and the diagonal lengths $\ol yq$, $\ol zp$ large.
If $\tran$ is convex and $\ol yp + \ol zq$ is fixed, we should choose equal leg lengths to maximize $\QuadOf yzqp$.
If we additionally fix the base lengths, the maximization problem is constrained enough to yield a finite result: $\QuadOf yzqp$ is maximized at the symmetric trapezoid.
While short legs and long diagonals are not independent of one another---equalizing the legs on its own can lower $\QuadOf yzqp$---the following theorem asserts that the symmetric trapezoid is maximizing regardless.

\begin{theorem}[Trapezoid comparison inequality]\label{thm:intro:main}
    For every quadruple $y$, $z$, $q$, $p$ of a \cato{} space and every nondecreasing convex $\tran\colon\Rp\to\R$ with concave derivative,
    \begin{equation}\dtag{T}\label{eq:T}
        \QuadOf yzqp \leq \QuadOf{\by}{\bz}{\bq}{\bp}
        \eqcm
    \end{equation}
    where $\by,\bz,\bq,\bp$ are the vertices of the (possibly degenerate) symmetric trapezoid of the Euclidean plane whose bases $\oa\by\bz, \oa\bp\bq$ have length $\ol\by\bz=\ol yz$ and $\ol\bq\bp=\ol pq$ and whose two legs $\oa\by\bp, \oa\bz\bq$ both have length equal to the average $\tfrac12(\ol yp+\ol zq)$ of the two original leg lengths \textup{(\cref{fig:trapezoid})}.
    Such a trapezoid exists, and it is the only symmetric trapezoid with these bases and legs, up to isometries of the Euclidean plane.
\end{theorem}

\Cref{thm:intro:main} is proved in \cref{ssec:master:trapezoid}.
Let us explain how it generalizes the observations from \cref{ssec:intro:intro}.

A \cato{} space is a geodesic metric space of globally nonpositive curvature in the sense of Alexandrov; see \cite{bridson99,burago01,bacak14b} for the standard theory.
Model cases are Hilbert spaces, complete simply connected Riemannian manifolds of nonpositive sectional curvature, metric trees, and the space of positive definite matrices with the affine-invariant metric \cite{bacak14b}.
A \cato{} space generalizes Hilbert spaces, allowing curvature to be negative while keeping the two model inequalities available: The role of Cauchy--Schwarz is taken by Reshetnyak's quadrilateral comparison \cite{reshetnyak68},
\begin{equation}\dtag{R}\label{eq:R}
    \olt yq+\olt zp-\olt yp-\olt zq = \Quad_{x^2}(y,z;q,p) \leq 2\,\ol yz\,\ol pq
    \eqcm
\end{equation}
which, in an inner product space, is \eqref{eq:intro-cs} doubled, since $\Quad_{x^2}(y,z;q,p) = 2\ip{\oa yz}{\oa pq}$ in this case.
By Berg and Nikolaev \cite{berg08}, a geodesic metric space is \cato{} \emph{if and only if} \eqref{eq:R} holds throughout it.
Furthermore, every \cato{} space is Ptolemaic \cite{foertsch07}, so \eqref{eq:P} continues to hold as well.

Write $\setcc$ for the class of transformations \cref{thm:intro:main} admits: the nondecreasing convex $\tran\colon\Rp\to\R$ whose derivative $\dtran$ is concave.
These conditions place $\tran$ between the identity and the square:
Since $\dtran$ is concave and nondecreasing it is at most linear, so no member of $\setcc$ grows faster than $x^2$; and since $\dtran$ is nondecreasing, a $\tran$ that is not constant has $\dtran\geq c>0$ from some point on, so it grows at least linearly.
Examples are $x^{\alpha}$ with $\alpha\in[1,2]$, the Huber loss $x^2-\posp{x-\theta}^2$ \cite{huber64}, the pseudo-Huber loss $\sqrt{1+x^2}-1$ \cite{charbonnier94}, and $\log(\cosh(x))$ \cite{green90}.

In \cref{ssec:intro:intro}, the symmetric trapezoid appeared as a configuration extremal for Cauchy--Schwarz and for Ptolemy's inequality at the same time; the theorem says that it stays extremal for everything lying between them, where ``between'' now has a precise meaning.
At $\tran=\sqmap$, \eqref{eq:T} is Reshetnyak's inequality \eqref{eq:R}, which becomes Cauchy--Schwarz in inner product spaces.
Squared and with $\tran=\mathrm{id}$, \eqref{eq:T} is exactly Reshetnyak's inequality \eqref{eq:R} plus twice Ptolemy's inequality \eqref{eq:P}.
And \eqref{eq:T} is also true for every $\tran$ in between the square and the identity, that is, for $\tran\in\setcc$, with the same maximizing quadrilateral, the symmetric trapezoid, for all $\tran$. 

\Cref{thm:intro:main} has several interesting corollaries, discussed in \cref{ssec:intro:constants}.
First, though, we record how sharp the statement is, and where it sits in the literature.

\subsection{Scope and sharpness}\label{ssec:intro:scope}

The proof of \cref{thm:intro:main} requires only that the quadruple at hand satisfies \eqref{eq:R} and \eqref{eq:P}, for the one labeling in which it is applied, so the theorem holds verbatim on any quadruple of any metric space that does, which is a strictly larger family than the quadruples of \cato{} spaces (\cref{ssec:master:fourpoint}).
\Cref{thm:intro:main} is derived from a more general, purely algebraic theorem, \cref{thm:algebraic}, which is a statement about six nonnegative real numbers that satisfy Reshetnyak's and (a version of) Ptolemy's inequalities when read as distances; no metric space enters its proof, not even the triangle inequality.
Furthermore, the comparison is sharp in three senses.
The class $\setcc$ cannot be enlarged in \cref{thm:intro:main}, not even if the \cato{} space is restricted to the Euclidean plane (\cref{thm:converse}); neither of the two hypotheses \eqref{eq:R} and \eqref{eq:P} can be dropped (\cref{prop:necessity}); and, for $\ol yz\,\ol pq>0$, equality holds in \eqref{eq:T}---for all $\tran\in\setcc$ alike, or just for a single $\tran$ whose derivative is affine on no interval, such as $x^{3/2}$---exactly when the four points can be embedded in the Euclidean plane as their own symmetric comparison trapezoid (\cref{prop:equality}).

\subsection{Background}\label{ssec:intro:background}
Let $y,z,q,p$ be four points of a \cato{} space.
In \cite{quadruple}, it was shown that
\begin{equation}\label{eq:intro-prev-class}
    \QuadOf yzqp \leq 2 \,\ol pq \,\dtran(\ol yz)
    \eqcm
\end{equation}
for all $\tran\in\setcc$.
Earlier, \cite{schoetz19} had obtained the sharp constant for the powers $\tran(x)=x^{\alpha}$ with $\alpha\in[1,2]$:
\begin{equation}\label{eq:intro-prev-power}
    \Quad_{x^\alpha}(y,z;q,p) \leq \alpha 2^{2-\alpha}  \,\ol pq \, \ol yz^{\alpha-1}
    \eqfs
\end{equation}
Inequalities of this kind, relating the six distances of four points, are called \emph{quadruple inequalities}.
Such inequalities arise for the Fr\'echet mean \cite{frechet48,sturm03} in statistics, the minimizer of $q\mapsto\int\tran(\ol yq)\dl\mu(y)$, which for $\tran=\sqmap$ in Euclidean spaces becomes the expectation and for $\tran = \mathrm{id}$ in $\R$ the median.
Taking $\tran$ of slower growth than $x^2$ buys robustness, in the sense of far-away observations exerting less influence on the Fr\'echet mean, and existence under weaker moment conditions. In a space with no linear structure, the quadruple inequality is the central tool to prove such robustness properties \cite{schoetz19,transformed}.

The class of functions $\mc S$ has been studied by Pinelis in the context of von Bahr--Esseen inequalities \cite{Pinelis2015}, where its Choquet structure is noted. This property is central for our proof and makes the arguments substantially more elegant and concise compared to previous proofs of quadruple inequalities \cite{schoetz19,quadruple}. 
Furthermore, Niculescu \cite[Corollary~1]{niculescu24} observes that $\setcc$ consists precisely of the nondecreasing convex functions that are $3$-concave in the sense of Hopf and Popoviciu, that is, whose divided differences of order three are nonpositive; see \cite{marinescu23} for a survey of this notion.

\subsection{Extension to positive curvature bounds}\label{ssec:intro:positive}

If the four points $y,z,q,p$ live in a \catk{} space with $\kappa>0$, allowing for positive curvature, and are not too far apart, the trapezoid comparison of \cref{thm:intro:main} holds up to a constant factor on the bases.
Any two points $p,q$ of such a space with $\ol pq<\pi/\sqrt\kappa$ are joined by a unique geodesic segment \cite[Prop.~II.1.4]{bridson99}, which we denote by $\seg pq$.
For the constant, define
\begin{equation}\label{eq:CRdef}
    C_R := \begin{cases}
        1 & \text{for } R=0\eqcm\\
        \frac{R}{\sin(R)} & \text{for } 0<R<\pi\eqfs
    \end{cases}
\end{equation}

\begin{theorem}\label{thm:intro:positive}
    Let $y,z,q,p$ be elements of a \catk{} space with $\kappa>0$ such that
    \begin{equation}\label{eq:RdefAndRange}
        R := \sqrt{\kappa}\, \diam\br{\seg yz\cup\seg pq} < \pi
        \eqfs
    \end{equation}
    Then, for every $\tran\in\setcc$,
    \begin{equation}\dtag{T$_+$}\label{eq:TC}
        \QuadOf yzqp \leq \QuadOf{\by}{\bz}{\bq}{\bp}
        \eqcm
    \end{equation}
    where $\by,\bz,\bq,\bp$ are the vertices of a (possibly degenerate) symmetric trapezoid of the Euclidean plane whose bases $\oa\by\bz, \oa\bp\bq$ satisfy $\ol\by\bz\,\ol\bq\bp = C_R \, \ol yz \, \ol pq$ and whose two legs $\oa\by\bp, \oa\bz\bq$ both have length equal to the average $\tfrac12(\ol yp+\ol zq)$ of the two original leg lengths.
\end{theorem}

\begin{remark}\label{rem:intro:R}
    \begin{enumerate}[label=(\roman*)]
        \item 
        The comparison trapezoids of \cref{thm:intro:positive} exist, for instance the rectangle with both bases equal to $\sqrt{C_R\,\ol yz\,\ol pq}$.
        All such trapezoids have the same diagonals $\ol\by\bq=\ol\bz\bp=\sqrt{s^2+C_R\,\ol yz\,\ol pq}$, where $s=\tfrac12(\ol yp+\ol zq)$ is the common length of the legs, so that the right-hand side of \eqref{eq:TC} does not depend on the choice.
        \item
        The two segments in \eqref{eq:RdefAndRange} exist as soon as $\ol yz$ and $\ol pq$ are smaller than $\pi/\sqrt\kappa$, which the hypothesis $R<\pi$ enforces.
        \item 
        The range $R<\pi$ is optimal: without that bound, \eqref{eq:TC} fails on large circles \textup{(\cref{rem:range})}.
        \item\label{it:R:mono}
        As $R/\sin(R)$ is increasing, \cref{thm:intro:positive} may be read with any upper bound for $R$ in place of $R$ itself.
        \begin{enumerate}[label=(\alph*)]
            \item\label{it:R:convex}
            If $y,z,q,p$ lie in a convex subset $\setC$ with $\diam(\setC)<\pi/\sqrt\kappa$, then  \eqref{eq:RdefAndRange} is fulfilled.
            \item\label{it:R:hull} 
            If $\sqrt\kappa\diam\cb{y,z,q,p}<\pi/2$, then $R=\sqrt\kappa\diam\cb{y,z,q,p}$, as closed balls of radius less than $\pi/(2\sqrt\kappa)$ are convex \cite[Prop.~II.1.4]{bridson99}. Beyond that radius the two diameters genuinely differ \textup{(\cref{rem:segdiam})}.
        \end{enumerate}
    \end{enumerate}
\end{remark}

\Cref{thm:intro:positive} rests on two new inequalities of independent interest, \eqref{eq:RC} and \eqref{eq:PC}.
They are the $\kappa>0$ counterparts of the two \cato{} inputs: Reshetnyak's quadrilateral comparison \eqref{eq:R} and Ptolemy's inequality \eqref{eq:P}.
\begin{theorem}\label{thm:intro:RC}
    Let $y,z,q,p$ be elements of a \catk{} space with $\kappa>0$ and assume \eqref{eq:RdefAndRange}.
    Then
    \begin{equation}\dtag{R$_+$}\label{eq:RC}
        \olt yq+\olt zp-\olt yp-\olt zq \leq 2 C_R \;\ol yz\,\ol pq
        \eqfs
    \end{equation}
\end{theorem}
\begin{theorem}\label{thm:intro:PC}
    Let $y,z,q,p$ be elements of a \catk{} space with $\kappa>0$ and assume \eqref{eq:RdefAndRange}.
    Then
    \begin{equation}\dtag{P$_+$}\label{eq:PC}
        \ol yq\,\ol zp \leq \ol yp\,\ol zq+ C_R \,\ol yz\,\ol pq
        \eqfs
    \end{equation}
\end{theorem}
While \eqref{eq:RC} is \eqref{eq:TC} evaluated at $\tran=\sqmap$, \eqref{eq:PC} is not implied by \eqref{eq:TC} for any $\tran\in\setcc$ (\cref{ssec:master:fourpoint}).
The two theorems are proved in \cref{app:catk}, where the constant $C_R$ is also shown to be optimal in each of them, which implies that the constant $C_R$ in \cref{thm:intro:positive} also cannot be lowered.

\subsection{Consequences}\label{ssec:intro:constants}

Write $\setcco:=\cb{\tran\in\setcc:\tran(0)=0}$ for the members of $\setcc$ that vanish at the origin.
Let $y,z,q,p$ be four points of a \catk{} space with $\kappa\geq0$ satisfying \eqref{eq:RdefAndRange}; for $\kappa=0$ this gives $R=0$ and hence $C_R=1$ with $C_R$ as in \eqref{eq:CRdef}.
Writing out the right-hand side of \eqref{eq:TC} with \cref{thm:intro:main,thm:intro:positive}, we obtain the explicit bound
\begin{equation}\label{eq:intro-explicit}
    \QuadOf yzqp \leq 2\tran\brOf{\sqrt{\frac14\br{\ol yp+\ol zq}^2+C_R\,\ol yz\,\ol pq}} - 2\tran\brOf{\frac12\br{\ol yp+\ol zq}}
    \eqcm
\end{equation}
for $\tran\in\setcc$.
We state the most important consequences here and refer to \cref{sec:geometry} for more.

\emph{The sharp constant for quadruple inequalities} (\cref{prop:Wtau}).
For nonconstant $\tran\in\setcc$, write $\ctranobest\tran$ for the smallest constant $L$ such that $\QuadOf yzqp\leq L\,\ol pq\,\dtran(\ol yz)$ holds for every quadruple of every \cato{} space.
Then $\ctranobest\tran$ is given by
\begin{equation}\label{eq:intro-Wtau}
    \ctranobest\tran = \sup_{x>0}\frac{2\dtran(x)}{\dtran(2x)}
    \in [1,2]
    \eqfs
\end{equation}
The value of $\ctranobest\tran$ was previously known only to lie in $[1,2]$, and explicitly only for the powers, where $\ctranobest{x^{\alpha}}=2^{2-\alpha}$ for $\alpha\in[1,2]$ \cite{schoetz19,quadruple}.

\emph{Quadruple inequality under positive curvature} (\cref{cor:twosided-catk}, \eqref{eq:ts:W}).
The quadruple inequality of \cite{quadruple} extends to positive curvature: for every quadruple of a \catk{} space with $\kappa>0$ satisfying \eqref{eq:RdefAndRange}, and every nonconstant $\tran\in\setcc$,
\begin{equation}\label{eq:intro-catk-quadruple}
    \oltr yq - \oltr yp - \oltr zq + \oltr zp \leq \ctranobest{\tran}\, C_R\, \ol pq \,\dtran(\ol yz)
    \eqfs
\end{equation}

\emph{The six transformed distances} (\cref{cor:twosided}, \eqref{eq:ts:sum}).
For every quadruple of a \cato{} space and every $\tran\in\setcco$, we have $\QuadOf yzqp \leq\tran(\ol yz)+\tran(\ol pq)$.
This was previously shown for inner product spaces \cite[Theorem~3]{quadruple}, and in general \cato{} spaces with an unnecessary factor $2$ on the right-hand side \cite[Corollary~1]{quadruple}.
Rearranging the terms of this result yields
\begin{equation}\label{eq:intro-J}
    \oltr yq+\oltr zp \leq \oltr yp+\oltr zq+\oltr yz+\oltr pq
    \eqfs
\end{equation}
That is, the sum of the transformed lengths of the two diagonals of a quadrilateral is bounded by the sum of the transformed lengths of its four sides.
At $\tran=\mathrm{id}$ this is the triangle inequality, averaged over the two paths from $y$ to $q$ and from $z$ to $p$ along the sides; at $\tran=\sqmap$ in the Euclidean plane it is Euler's classical quadrilateral inequality, the sum of the squared diagonals being at most the sum of the squared sides, with equality exactly for parallelograms.
Niculescu \cite{niculescu24} starts from the inner-product version and asks how it degrades in a Banach space; the present result moves in the orthogonal direction and extends the inequality to \cato{} spaces.

\emph{Bound by diameter} (\cref{cor:twosided}, \eqref{eq:ts:diam}).
For every quadruple of a \catk{} space with $\kappa\geq0$ satisfying \eqref{eq:RdefAndRange}, and every $\tran\in\setcc$, writing $\diam$ for the diameter of $\cb{y,z,q,p}$ and assuming $\diam>0$, we have
\begin{equation}\label{eq:intro-diam}
    \abs{\QuadOf yzqp}
    \leq 2 C_R \,\ol yz\,\ol pq\,\frac{\dtran\brOf{\tfrac12\diam}}{\diam}
    \eqfs
\end{equation}
Since $x\mapsto\dtran(x)/x$ is nonincreasing, a larger diameter improves this bound.
It is therefore particularly useful when the bases, $\ol yz$ and $\ol pq$, are two short segments far apart.

\subsection{Organization}\label{ssec:intro:org}

\Cref{sec:setting} fixes the notation and the function class $\setcc$.
\Cref{sec:master} states the main theorem in its most general and algebraic form, together with its equality case, its geometric reading as a comparison with the planar symmetric trapezoid, its optimality, and the place of its two hypotheses among the other conditions a quadruple can satisfy.
\Cref{sec:proof} proves the main theorem by reduction to an elementary inequality, \cref{lem:reduced}, which is verified in Lean~4 against Mathlib, see \cref{app:reduced,app:lean}.
\Cref{sec:geometry} derives corollaries of the main theorem.
Four appendices follow: \cref{app:catk} proves the \catk{} inputs \eqref{eq:RC} and \eqref{eq:PC}, and the remaining three treat the deferred proofs, the function class, and the formal verification.

\section{Setting}\label{sec:setting}

\subsection{Notation}\label{ssec:setting:notation}

For real numbers $x,y$ we write $\posp{x}:=\max(x,0)$, $x\wedge y:=\min(x,y)$ and $x\vee y:=\max(x,y)$.

In a metric space $(\mc Q,d)$ we abbreviate $\ol qp := d(q,p)$ for $q,p\in\mc Q$, and $\diam(A):=\sup_{q,p\in A}\ol qp$ for $A\subseteq\mc Q$.
When a geodesic segment joining $q$ to $p$ exists and is unique, we denote it by $\seg qp$; in a \catk{} space this is the case as soon as $\ol qp<\pi/\sqrt\kappa$ \cite[Prop.~II.1.4]{bridson99}, and in a \cato{} space for all $q,p$.

Four points always appearing in the order $y,z,q,p$ are read as a quadrilateral traversed in that order.
For its sides we keep the terminology of \cref{ssec:intro:intro}: $\ol yq$ and $\ol zp$ are the \emph{diagonals}, $\ol yp$ and $\ol zq$ the \emph{legs}, and $\ol yz$ and $\ol pq$ the \emph{bases}.

\begin{definition}[Quadruple functional]\label{def:Q}
    For $\tran\colon\Rp\to\R$ and four points $y,z,q,p$ of a metric space put
    \begin{equation}\label{eq:Qdef}
        \QuadOf yzqp := \oltr yq + \oltr zp - \oltr yp - \oltr zq
        \eqfs
    \end{equation}
\end{definition}

The following lemma records the symmetries of $\Quad_\tran$. Its proof is immediate from \cref{def:Q}.

\begin{lemma}[Symmetries]\label{lem:symmetry}
    For all $\tran$ and all $y,z,q,p$,
    \begin{equation}\label{eq:symmetries}
        \begin{gathered}
            \QuadOf zyqp = -\QuadOf yzqp,\qquad
            \QuadOf yzpq = -\QuadOf yzqp,\\
            \QuadOf qpyz = \QuadOf yzqp
            \eqfs
        \end{gathered}
    \end{equation}
    The first two exchange the diagonals $\cb{\ol yq,\ol zp}$ with the legs $\cb{\ol yp,\ol zq}$; the third exchanges the two bases $\ol yz$ and $\ol pq$.
\end{lemma}

\subsection{The function class}\label{ssec:setting:class}

\begin{definition}[The classes $\setcc$ and $\setcco$]\label{def:class}
Let $\setcc$ denote the set of nondecreasing convex functions $\tran\colon\Rp\to\R$ that are differentiable on $\Rpp$ with concave derivative $\dtran$, extended to $\Rp$ by $\dtran(0):=\lim_{x\searrow0}\dtran(x)$, and let
\begin{equation}\label{eq:classdef}
  \setcco := \cb{\tran\in\setcc \colon \tran(0)=0}
  \eqfs
\end{equation}
\end{definition}

The function class $\setcco$ is a \emph{convex cone} \cite{Pinelis2015}, since each of the three defining properties is preserved by nonnegative multiples and by sums. 
As both sides of \eqref{eq:T} are linear in $\tran$, it is enough to verify it on a set of generators.
Those generators are the Huber functions $\psi_\theta(x)=x^2-\posp{x-\theta}^2$ for $\theta\in\Rpp$, together with $\mathrm{id}$ and $x^2$.
Every $\tran\in\setcco$ is a mixture of them \textup{(\cref{prop:choquet-structure})}.
See \cref{app:class} for a list of example functions in $\setcco$ and further properties.

\section{The main result}\label{sec:master}

\subsection{Algebraic form}\label{ssec:master:algebraic}

\Cref{thm:intro:main,thm:intro:positive} are stated for \catk{} spaces.
In order to prove them, we show a more general result about six nonnegative numbers (proof in \cref{sec:proof}).

\begin{theorem}[Algebraic form of the trapezoid comparison inequality]\label{thm:algebraic}
    Let $a,e,b,c,u,v \geq 0$, and set $s := \tfrac12(b+c)$ and $\Am := \sqrt{s^2+uv}$.
    Then the inequality
    \begin{equation}\dtag{t}\label{eq:t}
        \tran(a) + \tran(e) - \tran(b) - \tran(c)
        \leq 2 \br{\tran(\Am) - \tran(s)}
    \end{equation}
    holds for every $\tran \in \setcc$ if and only if
    \begin{align}
        a^2 + e^2 - b^2 - c^2 &\leq 2uv \eqcm \dtag{r}\label{eq:r}\\
        a + e &\leq 2\Am \eqfs \dtag{p$^{*}$}\label{eq:pstar}
    \end{align}
\end{theorem}

Condition \eqref{eq:r} is \eqref{eq:t} at $\tran=x^2$ and \eqref{eq:pstar} is \eqref{eq:t} at $\tran=\mathrm{id}$.
We could summarize the variables $u$ and $v$ as a single symbol $w :=uv$ without losing power in \cref{thm:algebraic}. The current form is convenient as it makes identification with the six distances of four points simpler.

\begin{remark}\label{rem:six}
    We can read the six nonnegative numbers in \cref{thm:algebraic} in terms of the six distances of four points $y,z,q,p$:
    \begin{equation}\label{eq:six}
        a=\ol yq,\quad e=\ol zp,\quad b=\ol yp,\quad c=\ol zq,\quad
        u = \ol yz,\quad v =  \ol pq
        \eqcm
    \end{equation}
    so that $a,e$ are the \emph{diagonals}, $b,c$ the \emph{legs} and $u,v$ the \emph{bases}.
    Then \eqref{eq:t} and \eqref{eq:r} correspond to \eqref{eq:T} and \eqref{eq:R}, respectively. Furthermore,
    Ptolemy's inequality \eqref{eq:P} is 
    \begin{equation}\dtag{p}\label{eq:p}
        ae \leq bc + uv
        \eqfs
    \end{equation}
    Adding \eqref{eq:r} to twice \eqref{eq:p} gives the square of \eqref{eq:pstar}. Moreover \eqref{eq:pstar} is \eqref{eq:p} applied to $(m,m,s,s,u,v)$ instead of $(a,e,b,c,u,v)$, where $m := \tfrac12(a+e)$ is the mean diagonal.
    Thus, geometrically, \eqref{eq:pstar} is Ptolemy's inequality for a quadruple with averaged legs and averaged diagonals; it states that the mean diagonal of the original quadrilateral is at most the diagonal of the comparison trapezoid, that is $m \leq \Am$.
\end{remark}

\begin{proposition}[Equality]\label{prop:equality}
Let $a,e,b,c,u,v\geq0$ with $uv>0$ satisfy \eqref{eq:r} and \eqref{eq:pstar}, and set $s := \tfrac12(b+c)$ and $\Am := \sqrt{s^2+uv}$.
Then the following are equivalent:
\begin{enumerate}[label=\textup{(\roman*)}]
  \item\label{it:eq:all} \eqref{eq:t} is an equality for every $\tran\in\setcc$;
  \item\label{it:eq:one} \eqref{eq:t} is an equality for some $\tran\in\setcc$ whose derivative $\dtran$ is affine on no nondegenerate interval, for instance for $\tran=x^{3/2}$;
  \item\label{it:eq:trap}
    \begin{equation}\label{eq:eq-trapezoid}
      b = c \qquad\text{and}\qquad a = e = \Am
      \eqfs
    \end{equation}
\end{enumerate}
\end{proposition}

In the reading \eqref{eq:six}, condition \ref{it:eq:trap} says that the quadruple has equal legs and that both diagonals have the length of the comparison diagonal.
The proof is given in \cref{app:tight}.

\subsection{Geometric form}\label{ssec:master:trapezoid}

Let $y,z,q,p$ be four points of a \cato{} space.
By Reshetnyak's inequality \eqref{eq:R} \cite{reshetnyak68} and Ptolemy's inequality \eqref{eq:P} \cite{foertsch07} their six distances satisfy \eqref{eq:r} and \eqref{eq:p}, and hence \eqref{eq:pstar}, when read as in \eqref{eq:six}; so \cref{thm:algebraic} applies, with \begin{equation}\label{eq:sA}
    s := \tfrac12\br{\ol yp+\ol zq}
    \eqcm\qquad
    \Am := \sqrt{s^2+\ol yz\,\ol pq}
    \eqfs
\end{equation}
The next lemmas show that, in metric spaces, we can always build a planar symmetric trapezoid with bases $\ol yz$ and $\ol pq$, legs of length $s$ and diagonals of length $\Am$.

\begin{lemma}[Base difference bound]\label{lem:cross-dominates}
For any four points $y,z,q,p$ of a metric space,
\begin{equation}\label{eq:cross-dominates}
    \abs{\ol yz-\ol pq} \leq \min\brOf{\ol yq+\ol zp, \ol yp+\ol zq}
    \eqfs
\end{equation}
\end{lemma}

\begin{proof}
Applying the triangle inequality multiple times we obtain 
\begin{align*}
  \ol yz &\leq \ol yq+\ol pq+\ol zp\eqcm & \ol pq &\leq \ol yq+\ol yz+\ol zp\eqcm\\
  \ol yz &\leq \ol yp+\ol pq+\ol zq\eqcm & \ol pq &\leq \ol yp+\ol yz+\ol zq
  \eqcm
\end{align*}
which implies \eqref{eq:cross-dominates}.
\end{proof}

\begin{lemma}[The comparison quadruple exists]\label{lem:trapezoid-exists}
Let $y,z,q,p$ be four points of a metric space and put $s=\tfrac12(\ol yp+\ol zq)$.
Then there is a quadruple $\by,\bz,\bq,\bp$ of the Euclidean plane with
\begin{equation}\label{eq:trapezoid-data}
\begin{gathered}
  \ol\by\bz = \ol yz,\qquad \ol\bq\bp = \ol pq,\qquad
  \ol\by\bp = \ol\bz\bq = s,\\
  \text{and}\qquad
  \ol\by\bq = \ol\bz\bp = \sqrt{s^2+\ol yz\,\ol pq}\eqcm
\end{gathered}
\end{equation}
and it is a (possibly degenerate) symmetric trapezoid with parallel sides $\by\bz$ and $\bp\bq$, unique among symmetric trapezoids up to isometries of the Euclidean plane.
\end{lemma}

\begin{proof}
Write $u=\ol yz$, $v=\ol pq$ and
\begin{equation}\label{eq:spm}
  s_0:=\tfrac12\abs{u-v},\qquad t_0:=\tfrac12(u+v),\qquad\text{so that}\qquad
  t_0^2-s_0^2 = uv
  \eqfs
\end{equation}
Put $\by=(-\tfrac u2,0)$, $\bz=(\tfrac u2,0)$, $\bp=(-\tfrac v2,\eta)$, $\bq=(\tfrac v2,\eta)$ with $\eta\geq0$.
Then $\ol\by\bz=u$, $\ol\bq\bp=v$, $\ol\by\bp=\ol\bz\bq=\sqrt{\eta^2+s_0^2}$ and $\ol\by\bq=\ol\bz\bp=\sqrt{\eta^2+t_0^2}$.
Choosing $\eta$ with $\eta^2=s^2-s_0^2$, which is possible as $s\geq s_0$ by \cref{lem:cross-dominates}, gives \eqref{eq:trapezoid-data}, with the last equality by \eqref{eq:spm}.
The two horizontal sides are parallel by construction, and the configuration is symmetric with respect to the vertical axis; it degenerates to four collinear points when $\eta=0$.
For uniqueness, a symmetric trapezoid with bases $u,v$ and both legs $s$ is of the form just constructed, with $\eta\geq0$ determined by $\eta^2=s^2-s_0^2$, so its six distances are those of \eqref{eq:trapezoid-data}; and four points of the Euclidean plane are determined up to isometry by their six mutual distances, because the distances determine the Gram matrix of the three vectors $\bz-\by$, $\bq-\by$, $\bp-\by$, and the Gram matrix determines these vectors up to an orthogonal transformation.
\end{proof}

\begin{proof}[Proof of \cref{thm:intro:main}]
    We use the identification \eqref{eq:six}.
    In a \cato{} space \eqref{eq:R} and \eqref{eq:P} hold, hence so do \eqref{eq:r} and \eqref{eq:p}, and therefore \eqref{eq:pstar} (\cref{rem:six}).
    Thus, \cref{thm:algebraic} applies and yields \eqref{eq:T}.
    Existence and uniqueness of the comparison trapezoid are shown in \cref{lem:trapezoid-exists}.
\end{proof}

\begin{proof}[Proof of \cref{thm:intro:positive}]
    For $a,e,b,c$, we use the identification \eqref{eq:six}, and require for the bases $u v = C_R\, \ol yz\, \ol pq$. 
    Then \eqref{eq:RC} (\cref{thm:intro:RC} shown in \cref{app:catk}) is \eqref{eq:r}, and \eqref{eq:PC} (\cref{thm:intro:PC} shown in \cref{app:catk}) is \eqref{eq:p}. As stated in \cref{rem:six}, this also yields \eqref{eq:pstar}. Thus \cref{thm:algebraic} applies and yields \eqref{eq:TC}.
\end{proof}

\subsection{Optimality}\label{ssec:master:optimality}

The following proposition shows that neither \eqref{eq:R} nor \eqref{eq:P} can simply be dropped. It is proven by explicit examples in \cref{app:witnesses}.

\begin{proposition}\label{prop:necessity}
\leavevmode
\begin{enumerate}[label=\textup{(\roman*)}]
  \item\label{it:nec:R}
    There are four points $y,z,q,p$ forming a metric space in which \eqref{eq:P} holds for every labeling such that \eqref{eq:T} is false for every nonconstant $\tran\in\setcc$.
  \item\label{it:nec:P}
    There are four points $y,z,q,p$ forming a metric space in which \eqref{eq:R} holds for every labeling such that \eqref{eq:T} is false for every $\tran\in\setcc$ except those of the form $\tran(x)=\alpha x^2+\beta$ with $\alpha\geq0$ on $\ab{0,\diam\cb{y,z,q,p}}$.
\end{enumerate}
\end{proposition}

Next, we show that the set $\setcc$ cannot be enlarged, not even if \eqref{eq:T} is only required in the Euclidean plane.
In particular, \cref{thm:converse} below together with \cref{thm:intro:main} shows that, for measurable $\tran\colon\Rp\to\R$, the following three statements are equivalent:
\begin{itemize}[leftmargin=2em,itemsep=0pt]
    \item \eqref{eq:T} holds for all quadruples of all \cato{} spaces;
    \item \eqref{eq:T} holds for all quadruples of the Euclidean plane;
    \item $\tran\in\setcc$.
\end{itemize}

\begin{theorem}\label{thm:converse}
Let $\tran\colon\Rp\to\R$ be measurable and suppose that \eqref{eq:T} holds for every quadruple of the Euclidean plane.
Then $\tran\in\setcc$.
\end{theorem}

See \cref{app:proof:converse} for the proof.


\subsection{The four-point conditions}\label{ssec:master:fourpoint}

We place the two hypotheses, \eqref{eq:R} and \eqref{eq:P}, which we use as input to obtain the trapezoid comparison \eqref{eq:T}, among the conditions that a quadruple of a metric space can satisfy.
Throughout this subsection each condition is required for \emph{every} labeling ($4! = 24$ permutations) of the four points.
For a quadruple $\mc F=\cb{y,z,q,p}$ with six distances read as in \eqref{eq:six}, consider the following conditions:
\begin{itemize}[leftmargin=3.6em,itemsep=0pt]
    \item[\normalfont(M)] $\mc F$ is a metric space;
    \item[\normalfont(H)] $\mc F$ is a subset of a \cato{} space;
    \item[\normalfont(R)] \eqref{eq:r} holds for every labeling;
    \item[\normalfont(P)] \eqref{eq:p} holds for every labeling;
    \item[\normalfont(P$^{*}$)] \eqref{eq:pstar} holds for every labeling;
    \item[\normalfont(T)] \eqref{eq:t} holds for every $\tran\in\setcc$ and every labeling.
\end{itemize}

\begin{proposition}\label{prop:fourpoint}
\begin{enumerate}[label=\textup{(\roman*)}]
  \item\label{it:fp:H} $\textnormal{(H)} \Rightarrow \textnormal{(M)} \wedge \textnormal{(R)} \wedge \textnormal{(P)}$.
  \item\label{it:fp:RP} $\textnormal{(R)}\wedge\textnormal{(P)}\Rightarrow\textnormal{(P}^{*}\textnormal{)}$.
  \item\label{it:fp:T} $\textnormal{(T)}\iff\textnormal{(R)}\wedge\textnormal{(P}^{*}\textnormal{)}$.
  \item\label{it:fp:complete} These are all: if $\mc X$ is a set of the six conditions and $\textnormal{(Y)}$ is one of them that does not follow from $\mc X$ by \ref{it:fp:H}--\ref{it:fp:T}, then some row of \cref{tab:fourpoint} satisfies every condition of $\mc X$ and fails $\textnormal{(Y)}$.
\end{enumerate}
\end{proposition}

In particular, neither \cato{} nor metric structure are required for the trapezoid comparison inequality to hold. 

\begin{table}[ht]
\centering
\small
\begin{tabular}{@{}cl@{\quad}cccccc@{}}
\toprule
 & $(a,e,b,c,u,v)$ & \textnormal{M} & \textnormal{H} & \textnormal{R}
 & \textnormal{P} & $\textnormal{P}^{*}$ & \textnormal{T} \\
\midrule
\ref{it:w:simplex} & $(1,1,1,1,1,1)$ & \yes & \yes & \yes & \yes & \yes & \yes \\
\ref{it:w:notH}    & $(1,2,1,2,2,2)$ & \yes & \no  & \yes & \yes & \yes & \yes \\
\ref{it:w:notM}    & $(1,2,1,3,1,3)$ & \no  & \no  & \yes & \yes & \yes & \yes \\
\ref{it:w:notP}    & $(1,4,2,5,5,3)$ & \yes & \no  & \yes & \no  & \yes & \yes \\
\ref{it:w:notR}    & $(1,2,2,1,1,3)$ & \yes & \no  & \no  & \yes & \yes & \no  \\
\ref{it:w:notPs}   & $(3,4,2,4,3,1)$ & \yes & \no  & \yes & \no  & \no  & \no  \\
\ref{it:w:notRPs}  & $(1,1,1,1,1,2)$ & \yes & \no  & \no  & \yes & \no  & \no  \\
\bottomrule
\end{tabular}
\caption{Seven quadruples separating the conditions of \cref{prop:fourpoint}; each is verified in \cref{app:fourpoint}.
The symbol \no{} means that the condition fails for at least one labeling.}
\label{tab:fourpoint}
\end{table}

\section{\texorpdfstring{Proof of \cref{thm:algebraic}}{Proof of the algebraic theorem}}\label{sec:proof}

The proof reduces \eqref{eq:t}, for all $\tran$ at once, to a single elementary inequality between three explicit piecewise-quadratic functions of one variable, \cref{lem:reduced}, whose proof is an interlocking case analysis deferred to \cref{app:reduced} and machine-checked (\cref{app:lean}).

Throughout this section $a,e,b,c,u,v\geq0$ are fixed and $s,\Am$ are as in \cref{thm:algebraic}; since $uv\geq0$ we have $s\leq\Am$.

For the reduction, we first linearize in $\tran$.
Both sides of \eqref{eq:t} are linear in $\tran$, and $\setcco$ is a convex cone, so it suffices to verify \eqref{eq:t} on a set of generators.
The following representation identifies them.
It is the classical Choquet-type integral representation of the cone of nonnegative nondecreasing concave functions, in the form in which we use it; \cref{prop:choquet-structure} records the corresponding statement, which is due to Pinelis \cite{Pinelis2015}.
We include the short proof, and will apply the lemma to $f=\dtran$.

\begin{lemma}[Representation]\label{lem:representation}
    Let $f\colon\Rpp\to\R$ be nonnegative, nondecreasing and concave.
    Then there are constants $\eta_0,\eta_1 \geq 0$ and a nonnegative Borel measure $\nu$ on $\Rpp$ such that
    \begin{equation}\label{eq:representation}
        f(\sigma) = \eta_0 + \eta_1 \sigma + \int_{\Rpp} (\sigma\wedge \theta)\nu(\dl \theta)
        \qquad\text{for all } \sigma>0
        \eqcm
    \end{equation}
    namely, with $f\pr$ the right derivative of $f$, $\eta_0 = f(0+)$, $\eta_1 = \lim_{\sigma\to\infty} f\pr(\sigma)$, and $\nu = -\dl f\pr$ the nonnegative Lebesgue--Stieltjes measure determined by $\nu((\sigma,\theta]) = f\pr(\sigma)-f\pr(\theta)$.
\end{lemma}

\begin{proof}
    Being concave on $\Rpp$, $f$ has a right derivative $f\pr$ there, and $f\pr$ is nonincreasing and right-continuous; being nondecreasing, $f$ has $f\pr\geq0$.
    Hence $\eta_1 := \lim_{\sigma\to\infty}f\pr(\sigma) = \inf_{\sigma>0} f\pr(\sigma)$ exists and is nonnegative, and $f\pr-\eta_1$ is a nonnegative, nonincreasing function vanishing at $\infty$; let $\nu$ be the nonnegative measure on $\Rpp$ determined by $\nu((\sigma,\theta]):=f\pr(\sigma)-f\pr(\theta)$, so that
    \begin{equation}\label{eq:fprime}
        f\pr(\sigma) = \eta_1 + \nu\brOf{(\sigma,\infty)}
        \qquad(\sigma>0)
        \eqfs
    \end{equation}
    Since $f$ is concave and nondecreasing it is continuous on $\Rpp$ and $\eta_0 := f(0+) = \inf_{\sigma>0}f(\sigma) \geq 0$ exists, with $f(\sigma) = \eta_0+\int_0^\sigma f\pr$.
    Inserting \eqref{eq:fprime} and applying Tonelli to the nonnegative integrand,
    \begin{equation*}
        \int_0^\sigma \nu\brOf{(r,\infty)}\dl r
        = \int_{\Rpp}\abs{\cb{r\in(0,\sigma)\colon r<\theta}}\nu(\dl \theta)
        = \int_{\Rpp}(\sigma\wedge \theta)\nu(\dl \theta)
        \eqcm
    \end{equation*}
    which is \eqref{eq:representation}.
\end{proof}

\begin{lemma}[Linearization]\label{lem:choquet}
Let $\tran\in\setcco$, with $a,e,b,c,u,v,s,\Am$ as above.
Put
\begin{equation}\label{eq:rho}
  \rho(\sigma) := \ind_{\sigma<a}+\ind_{\sigma<e}-\ind_{\sigma<b}-\ind_{\sigma<c}
             -2\,\ind_{s\leq \sigma<\Am}
  \eqcm
\end{equation}
\begin{equation}\label{eq:GPhi}
  \Gfun(\sigma) := \int_\sigma^\infty \rho,
  \qquad
  \Phifun(\theta) := \int_0^\theta \Gfun(\sigma)\dl\sigma
  \eqfs
\end{equation}
Then, with $\eta_0,\eta_1,\nu$ as in \cref{lem:representation} applied to $f = \dtran$,
\begin{equation}\label{eq:decomposition}
\begin{gathered}
  \tran(a)+\tran(e)-\tran(b)-\tran(c) - 2\br{\tran(\Am)-\tran(s)}\\
  = \eta_0\Gfun(0) + \eta_1\Phifun(\infty)
        + \int_{\Rpp}\Phifun(\theta)\nu(\dl \theta)
  \eqcm
\end{gathered}
\end{equation}
where
\begin{equation}\label{eq:G0Phiinf}
  \Gfun(0) = a+e-2\Am,
  \qquad
  \Phifun(\infty) = \tfrac12\br{a^2+e^2-b^2-c^2} - uv
  \eqfs
\end{equation}
Consequently \eqref{eq:t} holds for every $\tran\in\setcco$ as soon as $\Gfun(0)\leq0$, $\Phifun(\infty)\leq0$ and $\Phifun(\theta)\leq0$ for all $\theta>0$.
\end{lemma}

\begin{proof}
A convex $\tran$ with $\tran(0)=0$ satisfies $\tran(\lambda x)\leq\lambda\tran(x)$ for $\lambda\in[0,1]$, hence is continuous at $0$, and therefore $\tran(x)=\int_0^x\dtran$ for every $x\geq0$; equivalently $\tran(x)=\int_0^\infty \dtran(\sigma)\ind_{\sigma<x}\dl \sigma$.
As $s\leq\Am$ we have $\ind_{\sigma<\Am}-\ind_{\sigma<s}=\ind_{s\leq \sigma<\Am}$, so the left-hand side of \eqref{eq:decomposition} equals $\int_0^\infty\dtran(\sigma)\rho(\sigma)\dl \sigma$, which converges absolutely because $\dtran$ is nondecreasing, hence bounded on compacts, and $\rho$ is bounded with compact support.

Apply \cref{lem:representation} to $f=\dtran$, which is nonnegative, nondecreasing and concave, and insert \eqref{eq:representation}.
The resulting interchange of $\int\dl \sigma$ with $\int\nu(\dl \theta)$ is legitimate even though $\nu$ may have infinite mass, because
\begin{equation*}
  \int_{\Rpp}\!\int_0^\infty (\sigma\wedge \theta)\abs{\rho(\sigma)}\dl \sigma\nu(\dl \theta)
  = \int_0^\infty \abs{\rho(\sigma)}\br{\dtran(\sigma)-\eta_0-\eta_1 \sigma}\dl \sigma
  <\infty
\end{equation*}
by \eqref{eq:representation} and the same boundedness.
This gives \eqref{eq:decomposition} with $\int_0^\infty\rho = \Gfun(0)$ and $\int_0^\infty \sigma\rho(\sigma)\dl \sigma = \Phifun(\infty)$ and, by Fubini,
\begin{equation*}
  \int_0^\infty (\sigma\wedge \theta)\rho(\sigma)\dl \sigma
  = \int_0^\theta\!\int_r^\infty\rho(\sigma)\dl\sigma\dl r
  = \Phifun(\theta)
  \eqfs
\end{equation*}

For \eqref{eq:G0Phiinf}, $\int_0^\infty\ind_{\sigma<x}\dl \sigma = x$ gives $\Gfun(0)=a+e-b-c-2(\Am-s)=a+e-2\Am$ because $b+c=2s$; and $\int_0^\infty \sigma\ind_{\sigma<x}\dl \sigma = \tfrac12 x^2$ gives $\Phifun(\infty)=\tfrac12(a^2+e^2-b^2-c^2)-(\Am^2-s^2)$, and $\Am^2-s^2=uv$.
The final assertion is immediate from $\eta_0,\eta_1\geq0$ and $\nu\geq0$.
\end{proof}

\begin{remark}\label{rem:endpoints}
$\Gfun(0)\leq0$ is \eqref{eq:pstar}, that is \eqref{eq:t} for $\tran=\mathrm{id}$; and $\Phifun(\infty)\leq0$ is \eqref{eq:r}, that is \eqref{eq:t} for $\tran=x^2$.
Both are hypotheses of \cref{thm:algebraic}.
What has to be proved is what happens in between.
\end{remark}

To make $\Gfun$ and $\Phifun$ explicit we use tents: for $\mu\geq r\geq0$ let
\begin{equation}\label{eq:tent}
  \tent{\mu}{r}(\sigma)
  := \min\brOf{\posp{\sigma-\mu+r},\posp{\mu+r-\sigma}}
\end{equation}
be the \emph{tent} supported on $[\mu-r,\mu+r]$, of height $r$ and area $r^2$.

\begin{lemma}[Tent form]\label{lem:tent}
Let $a,e,b,c,u,v$ satisfy \eqref{eq:r} and \eqref{eq:pstar}, and put
\begin{equation}\label{eq:notation}
  m := \tfrac12(a+e),\qquad h := \tfrac12\abs{a-e},\qquad \delta := \tfrac12\abs{b-c}
  \eqcm
\end{equation}
so that $\cb{a,e} = \cb{m-h,m+h}$ and $\cb{b,c} = \cb{s-\delta,s+\delta}$.
Then
\begin{equation}\label{eq:tent-hyp}
  0\leq h\leq m\leq\Am,\qquad 0\leq\delta\leq s\leq\Am,\qquad
  h^2-\delta^2\leq\Am^2-m^2
  \eqcm
\end{equation}
and $\Gfun$ of \eqref{eq:GPhi} is
\begin{equation}\label{eq:tent-G}
  \Gfun(\sigma)
  = \tent{m}{h}(\sigma) - \tent{s}{\delta}(\sigma)
        - 2\posp{\Am-\sigma\vee m}
  \eqfs
\end{equation}
\end{lemma}

\begin{proof}
The bounds $h\leq m$ and $\delta\leq s$ hold because $a,e,b,c\geq0$; $m\leq\Am$ is \eqref{eq:pstar} and $s\leq\Am$ holds because $uv\geq0$.
Since $4\Am^2=(b+c)^2+4uv$ and $(a-e)^2+(a+e)^2=2(a^2+e^2)$, $(b-c)^2+(b+c)^2=2(b^2+c^2)$, hypothesis \eqref{eq:r} is the same as $(a-e)^2-(b-c)^2\leq4\Am^2-(a+e)^2$, that is $h^2-\delta^2\leq\Am^2-m^2$.

For the formula for $\Gfun$ we use that two nonnegative numbers $x_1,x_2$ with mean $\mu$ and half-spread $r$ satisfy
\begin{equation}\label{eq:tent-identity}
  \posp{x_1-\sigma}+\posp{x_2-\sigma}-2\posp{\mu-\sigma}
  = \tent{\mu}{r}(\sigma)
  \qquad(\sigma\in\R)
  \eqfs
\end{equation}
To see it, write $x_1,x_2=\mu\mp r$ and check the four regimes.
For $\sigma\leq\mu-r$ the left side is $(x_1-\sigma)+(x_2-\sigma)-2(\mu-\sigma)=0$ and so is the right side.
For $\mu-r\leq\sigma\leq\mu$ it is $(\mu+r-\sigma)-2(\mu-\sigma)=\sigma-\mu+r$, the smaller of the two arguments of the minimum.
For $\mu\leq\sigma\leq\mu+r$ it is $\mu+r-\sigma$, again the smaller one.
For $\sigma\geq\mu+r$ both sides vanish.

Apply \eqref{eq:tent-identity} to the pair $(a,e)$, with mean $m$ and half-spread $h$, and to $(b,c)$, with mean $s$ and half-spread $\delta$:
\begin{equation*}
  \Gfun(\sigma) = \tent mh(\sigma)+2\posp{m-\sigma}
    -\tent s\delta(\sigma)-2\posp{s-\sigma}
    -2\posp{\Am-\sigma}+2\posp{s-\sigma}
  \eqcm
\end{equation*}
where the two $\posp{s-\sigma}$ terms cancel.
Finally $m\leq\Am$ turns $\posp{\Am-\sigma}-\posp{m-\sigma}$ into $\posp{\Am-\sigma\vee m}$---both expressions equal $\Am-m$ for $\sigma\leq m$ and equal $\posp{\Am-\sigma}$ for $\sigma\geq m$---which gives \eqref{eq:tent-G}.
\end{proof}

Writing
\begin{equation}\label{eq:PND}
  P(\theta):=\int_0^\theta\tent{m}{h},\qquad
  N(\theta):=\int_0^\theta\tent{s}{\delta},\qquad
  D(\theta):=2\int_0^\theta\posp{\Am-\sigma\vee m}\dl\sigma
  \eqcm
\end{equation}
so that $\Phifun = P-N-D$, one computes
\begin{equation}\label{eq:closed-forms}
\begin{gathered}
  P(\theta)=\begin{cases}
    \tfrac12\posp{\theta-m+h}^2, & \theta\leq m,\\[2pt]
    h^2-\tfrac12\posp{m+h-\theta}^2, & \theta\geq m,
  \end{cases}
  \\
  N(\theta)=\begin{cases}
    \tfrac12\posp{\theta-s+\delta}^2, & \theta\leq s,\\[2pt]
    \delta^2-\tfrac12\posp{s+\delta-\theta}^2, & \theta\geq s,
  \end{cases}
\end{gathered}
\end{equation}
by integrating the tent \eqref{eq:tent}, of total area $h^2$ and $\delta^2$ respectively, and
\begin{equation}\label{eq:closed-forms-D}
  D(\theta)=\begin{cases}
    2(\Am-m)\theta, & \theta\leq m,\\[2pt]
    \Am^2-m^2-\posp{\Am-\theta}^2, & \theta\geq m,
  \end{cases}
\end{equation}
because the integrand of $D$ equals the constant $2(\Am-m)$ on $\ab{0,m}$, while $2\int_m^\theta\posp{\Am-\sigma}\dl\sigma=(\Am-m)^2-\posp{\Am-\theta}^2$ for $\theta\geq m$, and $2m(\Am-m)+(\Am-m)^2=\Am^2-m^2$.
The total masses are $P(\infty)=h^2$, $N(\infty)=\delta^2$ and $D(\infty)=\Am^2-m^2$, so that $\Phifun(\infty)=h^2-\delta^2-\br{\Am^2-m^2}$, which is nonpositive exactly under \eqref{eq:r}, consistently with \cref{rem:endpoints}.

By \cref{lem:choquet} and \cref{rem:endpoints}, \cref{thm:algebraic} therefore follows from the following elementary statement, into which every trace of the hypotheses has been absorbed through \eqref{eq:tent-hyp}.
It is a statement about the five numbers $h,m,\delta,s,\Am$ alone.

\begin{lemma}[Reduced problem; \leanchecked]\label{lem:reduced}
Let $h,m,\delta,s,\Am \geq 0$. Assume
\begin{equation}\label{eq:reduced-hyp}
  0\leq h\leq m\leq \Am,\qquad 0\leq\delta\leq s\leq \Am,\qquad
  h^2-\delta^2\leq \Am^2-m^2
  \eqfs
\end{equation}
Let $P,N,D$ be given by \eqref{eq:closed-forms} and \eqref{eq:closed-forms-D}.
Then $P(\theta)\leq N(\theta)+D(\theta)$ for all $\theta\geq0$.
\end{lemma}

\Cref{lem:reduced} is proved in \cref{app:reduced} by an interlocking case analysis and machine checked (\cref{app:lean}).

\begin{proof}[Proof of \cref{thm:algebraic}]
Necessity: $\mathrm{id}$ and $x^2$ lie in $\setcc$, and \eqref{eq:t} at those two functions is \eqref{eq:pstar} and \eqref{eq:r}, respectively.

For sufficiency, both sides of \eqref{eq:t} are unchanged when a constant is added to $\tran$, so we may and do assume $\tran(0)=0$, that is $\tran\in\setcco$.
By \cref{lem:tent} the numbers $h,m,\delta,s,\Am$ satisfy \eqref{eq:tent-hyp}, which is \eqref{eq:reduced-hyp}, so \cref{lem:reduced} gives $\Phifun(\theta)=P(\theta)-N(\theta)-D(\theta)\leq0$ for every $\theta\geq0$.
Together with $\Gfun(0)\leq0$ and $\Phifun(\infty)\leq0$, which are \eqref{eq:pstar} and \eqref{eq:r} by \cref{rem:endpoints}, and with $\eta_0,\eta_1\geq0$ and $\nu\geq0$, the decomposition \eqref{eq:decomposition} gives \eqref{eq:t}.
\end{proof}

\section{Consequences}\label{sec:geometry}

The proofs for this section are given in \cref{app:geometry}.

\Cref{thm:intro:main} bounds $\Quadt$ from above using the average leg length.
By the symmetry properties of the quadruple functional (\cref{lem:symmetry}), $-\Quadt$ is bounded similarly by exchanging legs and diagonals.
Together, we obtain the following two-sided bounds.

\begin{corollary}[Two-sided bounds]\label{cor:twosided}
    Let $y,z,q,p$ be elements of a \cato{} space.
    Define
    \begin{equation}\label{eq:s-defs}
        \smin := \tfrac12 \min\brOf{\ol yp+\ol zq,\ol yq+\ol zp}
    \end{equation}
    and
    \begin{equation}\label{eq:diam}
        \diam := \diam\cb{y,z,q,p} = \max\brOf{\ol yq, \ol zp, \ol yp, \ol zq, \ol yz, \ol pq}
        \eqfs
    \end{equation}
    In \eqref{eq:ts:diam} below assume $\diam>0$; if $\diam=0$, the four points coincide and $\QuadOf yzqp=0$.
    Then the value $\abs{\QuadOf yzqp}$ is at most each of
\begin{align}
  & 2\tran\brOf{\sqrt{\smin^2+\ol yz\,\ol pq}}-2\tran(\smin)
  \eqcm & \text{for }\tran&\in\setcc\eqcm \label{eq:ts:trapezoid}\\
  & 2\,\ol yz\,\ol pq\,\frac{\dtran\brOf{\tfrac12\diam}}{\diam}
  \eqcm & \text{for }\tran&\in\setcc\eqcm \label{eq:ts:diam}\\
  & 2\tran\brOf{\tfrac12\br{\ol yz+\ol pq}}-2\tran\brOf{\tfrac12\abs{\ol yz-\ol pq}}
  \eqcm & \text{for }\tran&\in\setcc\eqcm \label{eq:ts:bases}\\
  & 2\,\ol pq\,\dtran\brOf{\tfrac12\ol yz} 
  \quad\text{and}\quad 
  2\,\ol yz\,\dtran\brOf{\tfrac12\ol pq}
  \eqcm & \text{for }\tran&\in\setcc\eqcm \label{eq:ts:hh}\\
  & 2\min(\ol pq, \ol yz)\dtran\brOf{\tfrac12\max(\ol pq, \ol yz)} 
  \eqcm & \text{for }\tran&\in\setcc\eqcm \label{eq:ts:hhminmax}\\
  & \ctranobest{\tran}\,\ol pq\,\dtran\brOf{\ol yz} 
  \quad\text{and}\quad
  \ctranobest{\tran}\, \ol yz \,\dtran\brOf{\ol pq} 
  \eqcm & \text{for }\tran&\in\setcc\eqcm \label{eq:ts:W}\\
  & \ctranobest{\tran} \min(\ol pq, \ol yz) \dtran\brOf{\max(\ol pq, \ol yz)} 
  \eqcm & \text{for }\tran&\in\setcc\eqcm \label{eq:ts:Wminmax}\\
  & 2\tran\brOf{\sqrt{\ol yz\,\ol pq}}
  \eqcm & \text{for }\tran&\in\setcco\eqcm \label{eq:ts:sqrt}\\
  & \tran(\ol yz)+\tran(\ol pq)
  \eqcm & \text{for }\tran&\in\setcco\eqcm \label{eq:ts:sum}
\end{align}
where, for nonconstant $\tran$,
\begin{equation}\label{eq:Wtau}
    \ctranobest{\tran}
    := \sup_{x>0}\frac{2\dtran(x)}{\dtran(2x)}
    = \sup_{x_1+x_2>0}
    \frac{\dtran(x_1)+\dtran(x_2)}{\dtran(x_1+x_2)}
    \in[1,2]
    \eqcm
\end{equation}
the suprema being taken over $x>0$ and over $x_1,x_2\geq0$ with $x_1+x_2>0$; here $\dtran>0$ on $\Rpp$ because a nonnegative nondecreasing concave function that vanishes at some $x_0>0$ vanishes identically.
For constant $\tran$ both sides of \eqref{eq:ts:W} and \eqref{eq:ts:Wminmax} vanish, whatever the value of $\ctranobest{\tran}$.
\end{corollary}

\begin{remark}[How the bounds compare]\label{rem:bound-order}
    The bounds are ordered as follows:
    \begin{equation}\label{eq:bound-order}
        \begin{gathered}
            \eqref{eq:ts:trapezoid}\leq\eqref{eq:ts:diam}\leq\eqref{eq:ts:hhminmax}\leq\eqref{eq:ts:Wminmax}
            \eqcm\\[2pt]
            \eqref{eq:ts:trapezoid}\leq\eqref{eq:ts:bases}\leq\eqref{eq:ts:hhminmax}
            \eqcm\qquad
            \eqref{eq:ts:bases}\leq\eqref{eq:ts:sqrt}\leq\eqref{eq:ts:sum}
            \eqcm
        \end{gathered}
    \end{equation}
    and \eqref{eq:ts:hhminmax}, \eqref{eq:ts:Wminmax} are the smaller of the two bounds in \eqref{eq:ts:hh}, \eqref{eq:ts:W} respectively.
    No further relation holds.
\end{remark}

\begin{remark}\label{rem:far-field}
    Only \eqref{eq:ts:trapezoid} and \eqref{eq:ts:diam} improve when the two bases move apart with their lengths held fixed; the others see the two bases alone.
\end{remark}

The bounds \eqref{eq:ts:W}, \eqref{eq:ts:sqrt} and \eqref{eq:ts:sum} improve the constants of the corresponding bounds in \cite[Corollary~1]{quadruple}, and \cref{prop:Wtau} shows that the improved constants are optimal.

\begin{proposition}[Optimal constants]\label{prop:Wtau}
The constants $2$ in \eqref{eq:ts:sqrt} and $1$ in \eqref{eq:ts:sum} cannot be lowered, not even for a single nonconstant $\tran\in\setcco$: both bounds are equalities whenever $y=p$ and $z=q$.
Moreover, for nonconstant $\tran\in\setcc$ the constant $\ctranobest{\tran}$ of \eqref{eq:Wtau} cannot be lowered either: it is the smallest constant $L$ such that
\begin{equation*}
  \QuadOf yzqp \leq L\,\ol pq\,\dtran\brOf{\ol yz}
\end{equation*}
holds for every quadruple of every \cato{} space.
\end{proposition}

Similar bounds as in \cref{cor:twosided} hold for \catk{} space with $\kappa>0$ when the product of the bases is multiplied by $C_R$.
For a simpler form, $C_R$ may be moved out of the comparison diagonal of \eqref{eq:TC} and into an outer factor by \cref{lem:gap-scaling}.

\begin{corollary}[Two-sided bounds under positive curvature]\label{cor:twosided-catk}
    Let $y,z,q,p$ be elements of a \catk{} space with $\kappa>0$ satisfying \eqref{eq:RdefAndRange}.
    Then each of the bounds of \cref{cor:twosided} holds after multiplication of its right-hand side by $C_R$.
    For the trapezoid bound \eqref{eq:ts:trapezoid} the sharper, unscaled form
    \begin{equation}\label{eq:ts:trapezoid-catk}
        \abs{\QuadOf yzqp}
        \leq 2\tran\brOf{\sqrt{\smin^2+C_R\,\ol yz\,\ol pq}}-2\tran(\smin)
        \eqcm\qquad\tran\in\setcc
        \eqcm
    \end{equation}
    holds as well.
\end{corollary}

\appendix

\section{The inputs in \catk{} spaces}\label{app:catk}

This appendix proves \cref{thm:intro:RC,thm:intro:PC} and settles the optimality of their constant.

Throughout, $\kappa>0$ and $\mc Q$ is a \catk{} space.
We write
\begin{equation}\label{eq:Dk}
    \Dk := \frac{\pi}{\sqrt\kappa}
\end{equation}
and $\Mk$ for the model plane of curvature $\kappa$: the sphere of radius $\kappa^{-1/2}$ in $\R^3$ with its intrinsic metric, whose diameter is $\Dk$.
Thus $\mc Q$ is a metric space in which any two points at distance less than $\Dk$ are joined by a geodesic and every geodesic triangle of perimeter less than $2\Dk$ is at least as thin as its comparison triangle in $\Mk$; see \cite[Ch.~II.1]{bridson99}.
Two points at distance less than $\Dk$ are joined by a \emph{unique} geodesic \cite[Prop.~II.1.4]{bridson99}.
Every \cato{} space is \catk{} for every $\kappa>0$, and $\Mk$ itself is \catk{}.

The four points $y,z,q,p\in\mc Q$ are always ones for which
\begin{equation}\label{eq:rdiam}
    r := \diam\br{\seg yz\cup\seg pq} < \Dk
    \eqcm\qquad
    R := \sqrt\kappa\,r < \pi
    \eqcm
\end{equation}
which is the hypothesis \eqref{eq:RdefAndRange} of \cref{thm:intro:positive}.
Since $y,z\in\seg yz$ and $p,q\in\seg pq$, all six distances of the quadruple are at most $r$; in particular each of them is less than $\Dk$, so all the geodesics occurring below are unique.
We keep the abbreviations \eqref{eq:six}, so that $a,e$ are the diagonals, $b,c$ the legs and $u,v$ the bases.
It is convenient to have the constant \eqref{eq:CRdef} available as a function of an unscaled distance: for $0\leq x<\Dk$ put
\begin{equation}\label{eq:Theta}
    \Th(x) := \frac{\sqrt\kappa\,x}{\sin\br{\sqrt\kappa\,x}}
    \eqcm\qquad
    \Th(0) := 1
    \eqcm
\end{equation}
a continuous increasing function on $\ab{0,\Dk}$ with $\Th(0)=1$ and $\Th(x)\to\infty$ as $x\uparrow\Dk$, which satisfies
\begin{equation}\label{eq:Theta-CR}
    \Th(r) = C_R
    \eqcm\qquad
    \Th\brOf{\tfrac r2} = C_{R/2}
    \eqfs
\end{equation}
We write $\Th$ rather than $C$ in this appendix because the constant is needed here at arguments other than $r$.

\Cref{app:catk:chordal} records the chordal transform and the two \emph{exact} inequalities its distances satisfy.
\Cref{app:catk:firstvar} then proves \eqref{eq:RC} by a first-variation argument.
\Cref{app:catk:ptolemy} proves \eqref{eq:PC} by an independent argument, and \cref{app:catk:inputs} assembles the results and settles the optimality of the constants.

A constant is unavoidable: in this generality both \eqref{eq:R} and \eqref{eq:P} fail, already on $\Mk$ and already for small quadruples.

\begin{example}[A spherical rectangle]\label{ex:sphere}
    On the unit sphere $M^2_1$, written in latitude and longitude, take
    \begin{equation*}
        y=(0,0),\qquad z=(0,\omega),\qquad p=(\eta,0),\qquad q=(\eta,\omega),
        \qquad \eta=\omega=\tfrac{\pi}{3}
        \eqfs
    \end{equation*}
    Then the legs and one base are $\ol yp=\ol zq=\ol yz=\tfrac\pi3$, the diagonals are $\ol yq=\ol zp=\arccos\tfrac14$ and the other base is $\ol pq=\arccos\tfrac78$, so that
    \begin{equation*}
    \begin{gathered}
        \diam\cb{y,z,q,p} = \arccos\tfrac14 = 1.31812 < \tfrac\pi2
        \eqcm\\
        \frac{\Quad_{\sqmap}(y,z;q,p)}{2\,\ol yz\,\ol pq} = \frac{1.28161}{1.05842} = 1.21087
        \eqfs
    \end{gathered}
    \end{equation*}
    So \eqref{eq:R} fails, by $0.22319$, and with it \eqref{eq:T} at $\tran=\sqmap$; the same quadruple has $\ol yq\,\ol zp-\ol yp\,\ol zq = 1.21087\,\ol yz\,\ol pq$, so \eqref{eq:P} fails as well.
    By \cref{rem:intro:R}\ref{it:R:mono}\ref{it:R:hull} the quantity $R$ of \eqref{eq:rdiam} equals $\arccos\tfrac14$ here, well inside the admissible range.
\end{example}

\begin{remark}[Why the segments and not the four points]\label{rem:segdiam}
    The quantity $r$ of \eqref{eq:rdiam} is the diameter of the two segments, not of the four points, and the two genuinely differ.
    Consider the theta-graph formed by three arcs of length $\pi$ with common endpoints, which is $\mathrm{CAT}(1)$ because each of its three embedded cycles has length $2\pi$.
    Take $y,z$ straddling the midpoint of one arc and $p,q$ straddling the midpoint of another, by four offsets that are pairwise distinct: equal offsets of $y$ and $q$, or of $z$ and $p$, would put that pair at distance exactly $\pi$.
    Then all six distances of the quadruple are less than $\pi$, while $\seg yz$ and $\seg pq$ contain a pair of antipodal points, namely the two arc midpoints, so that $\diam\br{\seg yz\cup\seg pq}=\pi$.
    The segment diameter is what the proofs below use.
    In \cref{prop:firstvar} the derivative of $\lambda\mapsto\olt y{\gamma(\lambda)}-\olt z{\gamma(\lambda)}$ is controlled at a point $\gamma(\lambda)$ of $\seg pq$ by comparing the directions there towards \emph{every} point of $\seg yz$, through \cref{lem:chain}; what that step can afford is $\Th$ of the largest distance from $\gamma(\lambda)$ to $\seg yz$, and the supremum over $\lambda$ of those distances is at most $r$.
    The proof of \eqref{eq:PC} likewise passes through \cref{prop:Lambda-bound}, an estimate along the two segments.
    By \cref{rem:intro:R}\ref{it:R:mono}\ref{it:R:hull} the distinction is invisible below the radius $\tfrac12\Dk$.

    In fact every appeal to $r$ below is at a distance between a point of $\seg yz$ and a point of $\seg pq$, the two bases $\ol yz$ and $\ol pq$ entering only through their product.
    The arguments therefore give \eqref{eq:RC} and \eqref{eq:PC} with $\Th(r_\times)$ in place of $\Th(r)$, where $r_\times:=\sup\cb{\ol xw\colon x\in\seg yz,\ w\in\seg pq}\leq r$, provided $\ol yz,\ol pq,r_\times<\Dk$.
    We state them with $r$ because \eqref{eq:RdefAndRange} is then a single condition, and because $r_\times=r$ in the extremal family of \cref{rem:theta-sharp}, so that no sharpness is lost.
\end{remark}

\subsection{The chordal transform}\label{app:catk:chordal}

The distances of a \catk{} space do not satisfy \eqref{eq:R} or \eqref{eq:P}, but their \emph{chords} satisfy both, exactly and for every quadruple of the space, however large.
This is the substitute on which \cref{app:catk:ptolemy} is built, and it is also what makes the constant of \cref{lem:logcomparison} below sharp.

\begin{definition}[Chordal transform]\label{def:chord}
    For $x\in\ab{0,\Dk}$ put
    \begin{equation}\label{eq:chord}
        \chmap(x) := \frac{2}{\sqrt\kappa}\,\sin\brOf{\frac{\sqrt\kappa\,x}{2}}
        \eqcm
    \end{equation}
    and write $\och yz:=\chmap\br{\ol yz\wedge\Dk}$ for the \emph{chordal distance} of two points $y,z$.
\end{definition}

The truncation leaves $\och yz$ defined for every pair, and is inactive exactly when $\ol yz\leq\Dk$, which is the case for every quadruple considered after this subsection.
The map $\chmap$ is increasing and concave on $\ab{0,\Dk}$ with $\chmap(0)=0$ and $\chmap\leq\mathrm{id}$, and $\chmap\to\mathrm{id}$ locally uniformly as $\kappa\searrow0$.
It is the chord function of the model plane: two points of $\Mk\subset\R^3$ at intrinsic distance $x$ are at Euclidean distance $\chmap(x)$.
The reason the same transform is the right one in an arbitrary \catk{} space is Berestovskii's theorem.

\begin{lemma}[Chordal distances are \cato{} distances]\label{lem:cone}
    Let $\mc Q$ be a \catk{} space with $\kappa>0$, let $Y:=(\mc Q,\sqrt\kappa\,d)$, a $\mathrm{CAT}(1)$ space, and let $\mathrm{Cone}(Y)$ be the Euclidean cone over $Y$ \cite[Def.~I.5.6]{bridson99}: the set $\Rp\times Y$ with all points $(0,y)$ identified to a single apex, metrized by
    \begin{equation}\label{eq:conemetric}
        d_{\mathrm{Cone}(Y)}\br{(\lambda,y),(\lambda\pr,y\pr)}^2
        = \lambda^2+(\lambda\pr)^2-2\lambda\lambda\pr\cos\brOf{\pi\wedge\sqrt\kappa\,d(y,y\pr)}
        \eqfs
    \end{equation}
    Then $\mathrm{Cone}(Y)$ is a \cato{} space, and $\iota\colon\mc Q\to \mathrm{Cone}(Y)$, $\iota(y):=\br{\kappa^{-1/2},y}$, satisfies
    \begin{equation}\label{eq:coneiso}
        d_{\mathrm{Cone}(Y)}\br{\iota(y),\iota(y\pr)} = \och y{y\pr}
        \qquad\text{for all }y,y\pr\in\mc Q
        \eqfs
    \end{equation}
\end{lemma}

\begin{proof}
    The first assertion is Berestovskii's theorem \cite[Thm.~II.3.14]{bridson99}: the Euclidean cone over a metric space is \cato{} if and only if that space is $\mathrm{CAT}(1)$, and $Y$ is $\mathrm{CAT}(1)$ because $\mc Q$ is \catk{}.
    For \eqref{eq:coneiso} put $\lambda=\lambda\pr=\kappa^{-1/2}$ in \eqref{eq:conemetric}: with $x:=d(y,y\pr)\wedge\Dk$ the right-hand side is $\tfrac2\kappa\br{1-\cos\sqrt\kappa x}=\tfrac4\kappa\sin^2\br{\sqrt\kappa x/2}=\chmap(x)^2=\och y{y\pr}^2$; the truncation in \cref{def:chord} is exactly the one already present in \eqref{eq:conemetric}.
\end{proof}

By \eqref{eq:coneiso} the chordal distance is the pullback under $\iota$ of the metric of $\mathrm{Cone}(Y)$; since $\chmap$ vanishes only at $0$, it is itself a metric on $\mc Q$, with values in $\ab{0,2\kappa^{-1/2}}$, and it is bounded even when $\mc Q$ is not.

\begin{proposition}[Chordal Reshetnyak and chordal Ptolemy]\label{prop:chordal}
    Let $y,z,q,p$ be any four points of a \catk{} space with $\kappa>0$.
    Then
    \begin{align}
        \och yq^2+\och zp^2-\och yp^2-\och zq^2 &\leq 2\,\och yz\,\och pq \eqcm \dtag{$\ch{\mathrm R}$}\label{eq:chordR}\\
        \och yq\,\och zp &\leq \och yp\,\och zq+\och yz\,\och pq \eqfs \dtag{$\ch{\mathrm P}$}\label{eq:chordP}
    \end{align}
    \textup{(}The hat marks the chordal version of an inequality or of a distance.\textup{)}
\end{proposition}

\begin{proof}
    By \eqref{eq:coneiso} the six chordal numbers $\och yq,\dots,\och pq$ are the six mutual distances of the four points $\iota(y),\iota(z),\iota(q),\iota(p)$ of the \cato{} space $\mathrm{Cone}(Y)$, whatever the four points are.
    Inequalities \eqref{eq:R} and \eqref{eq:P} apply to them, with no size restriction of any kind.
\end{proof}

\begin{remark}\label{rem:chordal}
    Inequality \eqref{eq:chordP} is Valentine's spherical analogue of Ptolemy's theorem \cite{valentine70b} in the form given for \catk{} spaces by G\'omez and M\'emoli \cite[Thm.~3.4]{gomez24}, there under an additional perimeter condition, which \cref{lem:cone} shows to be unnecessary.
    We have not found \eqref{eq:chordR} stated explicitly; it follows immediately from \cref{lem:cone}, that is, from Berestovskii's theorem together with \eqref{eq:R}, and degenerates back as $\kappa\searrow0$ into \eqref{eq:R}, the inequality that \emph{characterizes} \cato{} spaces \cite{berg08}.
    A different $\kappa\neq0$ counterpart of \eqref{eq:R} is due to Berg and Nikolaev themselves \cite{berg18}, through a $\kappa$-quadrilateral cosine.
    Theirs characterizes the curvature bound, but only for spaces of diameter at most $\pi/(2\sqrt\kappa)$, a hypothesis they show to be sharp; \eqref{eq:chordR} asserts an inequality inside a space already known to be \catk{}, and carries no size restriction at all: by \cref{def:chord} the six chordal distances are defined for every quadruple, and the inequality holds for every one of them.
    Both \eqref{eq:chordR} and \eqref{eq:chordP} are attained: on the spherical rectangle of \cref{ex:sphere} they hold with equality.
\end{remark}

Since \eqref{eq:chordR} and \eqref{eq:chordP} give the hypotheses of \cref{thm:algebraic}, read for the chordal distances, the trapezoid comparison transfers to them with no constant and no size restriction.

\begin{corollary}[Chordal trapezoid comparison]\label{cor:chordal-trapezoid}
    Let $y,z,q,p$ be any four points of a \catk{} space with $\kappa>0$, and put
    \begin{equation}\label{eq:chordal-sA}
        \ch s := \tfrac12\br{\och yp+\och zq}
        \eqcm\qquad
        \ch\Am := \sqrt{\ch s^2+\och yz\,\och pq}
        \eqfs
    \end{equation}
    Then, for every $\tran\in\setcc$,
    \begin{equation}\dtag{$\ch{\mathrm T}$}\label{eq:chordT}
        \tran\brOf{\och yq}+\tran\brOf{\och zp}-\tran\brOf{\och yp}-\tran\brOf{\och zq}
        \leq 2\br{\tran\brOf{\ch\Am}-\tran\brOf{\ch s}}
        \eqfs
    \end{equation}
\end{corollary}

\begin{proof}
    By \cref{prop:chordal} the six chordal numbers satisfy \eqref{eq:r} and \eqref{eq:p} in the reading \eqref{eq:six}, hence also \eqref{eq:pstar} by \cref{rem:six}; now apply \cref{thm:algebraic}.
\end{proof}

\begin{remark}\label{rem:chordal-trapezoid}
    Unlike \cref{thm:intro:positive}, \eqref{eq:chordT} carries no constant and no restriction on the size of the quadruple. It degenerates to \eqref{eq:T} as $\kappa\searrow0$, since $\chmap\to\mathrm{id}$ locally uniformly.
    It is sharp, by \cref{prop:equality} applied to the chordal numbers.
    What it bounds, however, is a quadruple functional built from the chordal distances rather than from the intrinsic ones, and the two cannot be converted into one another.
    On the one hand, $\tran\circ\chmap^{-1}$ is defined only on $\ab{0,\chmap(\Dk)}$ and its derivative is unbounded at the right endpoint, hence not concave, so $\tran\circ\chmap^{-1}\notin\setcc$ for every nonconstant $\tran$.
    On the other hand, even if it were, the right-hand side of \eqref{eq:chordT} is built from $\ch s$ and $\ch\Am$, which are not $\chmap$ of the intrinsic leg average and comparison diagonal.
\end{remark}

\subsection{First variation on the tangent cone}\label{app:catk:firstvar}

We use the standard notions of angle, space of directions and tangent cone \cite[Ch.~I.1, I.5, II.1, II.3]{bridson99}.
For a base point $x\in\mc Q$, $\ang{x}(w,w\pr)$ is the Alexandrov angle at $x$, $\Scone x$ is the space of directions at $x$ metrized by the angle, and $\Tcone x$ is the Euclidean cone over $\Scone x$, with elements $(\varrho,\xi)$, apex $o_x$ and metric
\begin{equation}\label{eq:conelaw}
    d_{\Tcone x}\br{(\varrho,\xi),(\varrho\pr,\zeta)}^2
    = \varrho^2+(\varrho\pr)^2-2\varrho\varrho\pr\cos\angle(\xi,\zeta)
    \eqcm
\end{equation}
and $\Log_x(w):=\br{\ol wx,\xi_w}\in\Tcone x$ for $\ol wx<\Dk$, where $\xi_w$ is the direction at $x$ of the geodesic $\seg xw$, with $\Log_x(x):=o_x$.

\begin{lemma}[Log comparison, one triangle]\label{lem:logcomparison}
    Let $x,y,y\pr\in\mc Q$ with $\ol yx,\ol{y\pr}x\leq r$ and $\ol yx+\ol{y\pr}x+\ol y{y\pr}<2\Dk$.
    Then
    \begin{equation}\label{eq:logcomp}
        d_{\Tcone x}\br{\Log_x(y),\Log_x(y\pr)}
        \ \leq\ \Th(r)\,\chmap\br{\ol y{y\pr}}
        \ \leq\ \Th(r)\,\ol y{y\pr}
        \eqfs
    \end{equation}
\end{lemma}

\begin{proof}
    Write $k:=\sqrt\kappa$, $d_1:=\ol yx$, $d_2:=\ol{y\pr}x$, $\theta:=\ang{x}(y,y\pr)$ and $\bar d:=\ol y{y\pr}$; if $y=x$ or $y\pr=x$ read $\theta$ as arbitrary, the argument being insensitive to it.
    By the perimeter hypothesis a comparison triangle for $(x,y,y\pr)$ exists in $\Mk$ \cite[Lem.~I.2.14]{bridson99}; let $\bar\theta$ be its angle at the vertex corresponding to $x$.
    Angle comparison \cite[Prop.~II.1.7(4)]{bridson99} gives $\theta\leq\bar\theta$, so by the cone law \eqref{eq:conelaw}
    \begin{equation}\label{eq:logcomp1}
    \begin{gathered}
        d_{\Tcone x}\br{\Log_x(y),\Log_x(y\pr)}^2
        \ \leq\ d_1^2+d_2^2-2d_1d_2\cos\bar\theta\\
        =\ (d_1-d_2)^2+4d_1d_2\sin^2\brOf{\tfrac{\bar\theta}{2}}
        \eqfs
    \end{gathered}
    \end{equation}
    The spherical cosine rule in $\Mk$ reads
    \begin{equation*}
        \cos\br{k\bar d}=\cos(kd_1)\cos(kd_2)+\sin(kd_1)\sin(kd_2)\cos\bar\theta
        \eqfs
    \end{equation*}
    Subtracting it from $1$ and halving gives its half-angle form
    \begin{equation*}
        \sin^2\brOf{\tfrac{k\bar d}{2}}=\sin^2\brOf{\tfrac{k(d_1-d_2)}{2}}+\sin(kd_1)\sin(kd_2)\sin^2\brOf{\tfrac{\bar\theta}{2}}
        \eqcm
    \end{equation*}
    that is, after multiplication by $4/k^2$,
    \begin{equation}\label{eq:logcomp2}
        \chmap\br{\bar d}^2
        \ =\ \chmap\br{\abs{d_1-d_2}}^2
        +\frac{4}{k^2}\,\sin(kd_1)\sin(kd_2)\,\sin^2\brOf{\tfrac{\bar\theta}{2}}
        \eqfs
    \end{equation}
    Compare \eqref{eq:logcomp1} with \eqref{eq:logcomp2} term by term.
    Straight from the definitions \eqref{eq:chord} and \eqref{eq:Theta},
    \begin{equation*}
    \begin{gathered}
        (d_1-d_2)^2=\Th\brOf{\tfrac{\abs{d_1-d_2}}{2}}^2\chmap\br{\abs{d_1-d_2}}^2
        \eqcm\\
        4d_1d_2=\Th(d_1)\,\Th(d_2)\cdot\frac{4}{k^2}\sin(kd_1)\sin(kd_2)
        \eqfs
    \end{gathered}
    \end{equation*}
    Both sines are nonnegative because $kd_1,kd_2<\pi$, and $\abs{d_1-d_2}/2,d_1,d_2\leq r$, so both coefficients are at most $\Th(r)^2$ by monotonicity of $\Th$.
    Hence $d_{\Tcone x}^2\leq\Th(r)^2\chmap(\bar d)^2$.
    The second inequality of \eqref{eq:logcomp} is $\chmap\leq\mathrm{id}$.
\end{proof}

\begin{remark}[The constant of \cref{lem:logcomparison} is sharp]\label{rem:logsharp}
    No convexity of balls is used; the only restriction is the perimeter hypothesis, that is, the existence of the comparison triangle, and \cref{lem:chain} removes it.
    The constant is sharp: on $\Mk$ with $d_1=d_2=r$ the tangent cone is Euclidean and $\theta=\bar\theta$, so \eqref{eq:logcomp1} and \eqref{eq:logcomp2} are equalities and give $d_{\Tcone x}=2r\sin\br{\bar\theta/2}$ and $\chmap(\bar d)=\tfrac{2}{k}\sin(kr)\sin\br{\bar\theta/2}$, whose ratio is $\Th(r)$ for every $\bar\theta$ admitted by the perimeter hypothesis.
\end{remark}

\begin{lemma}[Log comparison along a geodesic]\label{lem:chain}
    Let $x\in\mc Q$ and let $\sigma\colon\ab{0,1}\to\mc Q$ be a geodesic all of whose points are at distance at most $r$ from $x$.
    Then
    \begin{equation}\label{eq:chain}
        d_{\Tcone x}\br{\Log_x(\sigma(0)),\Log_x(\sigma(1))}\ \leq\ \Th(r)\,\ol{\sigma(0)}{\sigma(1)}
        \eqfs
    \end{equation}
\end{lemma}

\begin{proof}
    Put $\ell:=\ol{\sigma(0)}{\sigma(1)}$, the length of $\sigma$, and note $\ell\leq2r$.
    Since $r<\Dk$ there is $n\in\mathbb N$ with $2r+\ell/n<2\Dk$; put $y_i:=\sigma(i/n)$ for $0\leq i\leq n$.
    The triangle $(x,y_i,y_{i+1})$ has $\ol{y_i}x,\ol{y_{i+1}}x\leq r$ and perimeter at most $2r+\ell/n<2\Dk$, so \cref{lem:logcomparison} applies to it and gives $d_{\Tcone x}\br{\Log_x(y_i),\Log_x(y_{i+1})}\leq\Th(r)\,\ell/n$.
    Summing over $i$ and using the triangle inequality in the metric space $\Tcone x$ gives \eqref{eq:chain}.
\end{proof}

\begin{lemma}\label{lem:conelip}
    Let $x\in\mc Q$ and $\xi\in\Scone x$, and define $\beta_\xi\colon\Tcone x\to\R$ by $\beta_\xi(o_x):=0$ and $\beta_\xi\br{(\varrho,\zeta)}:=\varrho\cos\angle(\zeta,\xi)$.
    Then $\beta_\xi$ is $1$-Lipschitz and $\abs{\beta_\xi(\upsilon)}\leq d_{\Tcone x}(\upsilon,o_x)$ for every $\upsilon\in\Tcone x$.
\end{lemma}

\begin{proof}
    Let $\eta(\lambda):=(\lambda,\xi)$ for $\lambda>0$.
    By \eqref{eq:conelaw}, for $\upsilon=(\varrho,\zeta)$,
    \begin{equation*}
        \lambda-d_{\Tcone x}\br{\upsilon,\eta(\lambda)}
        =\frac{\lambda^2-d_{\Tcone x}\br{\upsilon,\eta(\lambda)}^2}{\lambda+d_{\Tcone x}\br{\upsilon,\eta(\lambda)}}
        =\frac{2\varrho\lambda\cos\angle(\zeta,\xi)-\varrho^2}{\lambda+d_{\Tcone x}\br{\upsilon,\eta(\lambda)}}
        \xrightarrow{\ \lambda\to\infty\ }\beta_\xi(\upsilon)
        \eqcm
    \end{equation*}
    and the same limit is $0=\beta_\xi(o_x)$ for $\upsilon=o_x$.
    Each $\upsilon\mapsto\lambda-d_{\Tcone x}(\upsilon,\eta(\lambda))$ is $1$-Lipschitz, hence so is the pointwise limit $\beta_\xi$; and $\abs{\beta_\xi\br{(\varrho,\zeta)}}\leq \varrho$.
    (Equivalently, $\beta_\xi$ is the negative of the Busemann function of the ray $\eta$.)
\end{proof}

The next proposition proves \eqref{eq:RC} by a first variation along $\seg pq$; this subsection is the only place where the tangent cone is used.

\begin{proposition}[Reshetnyak's comparison with a constant]\label{prop:firstvar}
    Let $y,z,q,p$ satisfy \eqref{eq:rdiam}.
    Then \eqref{eq:RC} holds, that is,
    \begin{equation*}
        \Quad_{\sqmap}(y,z;q,p)\ \leq\ 2\,\Th(r)\,\ol yz\,\ol pq
        \eqfs
    \end{equation*}
\end{proposition}

\begin{proof}
    Put $\ell:=\ol pq$; we may assume $\ell>0$.
    Let $\gamma\colon\ab{0,1}\to\seg pq$ be the geodesic from $p$ to $q$, of speed $\ell$, and set $h(\lambda):=\olt y{\gamma(\lambda)}-\olt z{\gamma(\lambda)}$, so that $h(1)-h(0)=\Quad_{\sqmap}(y,z;q,p)$.
    All the distances occurring here are at most $r$ by \eqref{eq:rdiam}, and each $\lambda\mapsto\ol w{\gamma(\lambda)}$ is $\ell$-Lipschitz with values in $\ab{0,r}$; hence $h$ is Lipschitz, therefore absolutely continuous.

    Fix $\lambda\in\ab{0,1}$ with $\lambda<1$, let $\xi_\lambda\in\Scone{\gamma(\lambda)}$ be the direction of $\gamma$ at $\gamma(\lambda)$ towards $\gamma(1)$, and write $d_w:=\ol w{\gamma(\lambda)}$ and $B_w:=\beta_{\xi_\lambda}\br{\Log_{\gamma(\lambda)}(w)}$ for $w\in\cb{y,z}$.
    If $w\neq\gamma(\lambda)$, then $0<d_w\leq r<\Dk$ and the first variation formula \cite[Cor.~II.3.6]{bridson99} gives
    $\lim_{\epsilon\downarrow0}\epsilon^{-1}\br{\ol w{\gamma(\lambda+\epsilon)}-d_w}=-\ell\cos\ang{\gamma(\lambda)}(w,\xi_\lambda)=-\ell B_w/d_w$, and therefore
    \begin{equation}\label{eq:fv}
        \lim_{\epsilon\downarrow0}\frac{\olt w{\gamma(\lambda+\epsilon)}-d_w^2}{\epsilon}
        = 2d_w\cdot\brOf{-\ell\,\frac{B_w}{d_w}} = -2\ell\,B_w
        \eqcm
    \end{equation}
    the factor $d_w$ cancelling.
    If $w=\gamma(\lambda)$ then $\ol w{\gamma(\lambda+\epsilon)}=\epsilon\ell$ and $B_w=0$, so both sides of \eqref{eq:fv} vanish and the identity persists.
    Hence the right derivative $h^+(\lambda)$ exists for every $\lambda<1$ and equals $2\ell\br{B_z-B_y}$.
    Now $\abs{B_z-B_y}\leq\Th(r)\,\ol yz$ by \cref{lem:conelip} combined with \cref{lem:chain}, the latter applied at the point $\gamma(\lambda)$ to the geodesic $\seg yz$: every point of $\seg yz$ is at distance at most $r$ from $\gamma(\lambda)\in\seg pq$, by \eqref{eq:rdiam}.
    So $h^+(\lambda)\leq2\ell\,\Th(r)\,\ol yz$, and as $h$ is absolutely continuous with $h\pr=h^+$ almost everywhere, integrating over $\ab{0,1}$ gives the claim.
\end{proof}

\subsection{\texorpdfstring{Ptolemy's inequality with the constant $\Th(r)$}{Ptolemy's inequality with the constant Theta}}\label{app:catk:ptolemy}

This subsection proves \eqref{eq:PC}.
The argument is multiplicative.
By \eqref{eq:chord} and \eqref{eq:Theta} every distance is its chord times an explicit factor,
\begin{equation}\label{eq:chord-factor}
    x = \Th\brOf{\tfrac x2}\,\chmap(x)
    \qquad\br{0\leq x<\Dk}
    \eqcm
\end{equation}
and by \cref{prop:chordal} the chords satisfy Ptolemy's inequality \eqref{eq:chordP} \emph{exactly}.
Writing $\ol yq\,\ol zp-\ol yp\,\ol zq$ in terms of chords therefore leaves a single defect, the ratio of the diagonal factor to the leg factor, and the whole task is to bound that ratio.
Its logarithm is the quadruple functional of
\begin{equation}\label{eq:Lam}
    \Lambda(x) := \log\Th\brOf{\tfrac x2}
    = \log\frac{\sqrt\kappa\,x/2}{\sin\br{\sqrt\kappa\,x/2}}
    \qquad\br{0\leq x<\Dk}
    \eqcm
\end{equation}
which vanishes at $0$ and increases to $+\infty$ as $x\uparrow\Dk$, and the bound we prove for it is \cref{prop:Lambda-bound}.

For a differentiable $f$ we write $\chi_f(x):=f\pr(x)/x$ and, as in \cref{def:Q},
$\Quad_f(y,z;q,p):=f(\ol yq)+f(\ol zp)-f(\ol yp)-f(\ol zq)$.
Alongside $\chi_\Lambda(x)=\Lambda\pr(x)/x$ we write
\begin{equation}\label{eq:Kap}
    \Kap(x) := \Th(x)\,\chi_\Lambda(x) = \frac{\sqrt\kappa\,\Lambda\pr(x)}{\sin\br{\sqrt\kappa\,x}}
    \quad\text{for }0<x<\Dk,
    \qquad \Kap(0):=\tfrac{\kappa}{12}
    \eqfs
\end{equation}

\subsubsection*{The function \texorpdfstring{$\Lambda$}{Lambda}}

The series we need come from the Euler product for the sine.

\begin{lemma}\label{lem:Lambda}
    For $0\leq x<\Dk$,
    \begin{equation}\label{eq:Lambda-series}
    \begin{gathered}
        \Lambda(x) = -\sum_{n\geq1}\log\brOf{1-\frac{\kappa x^2}{4n^2\pi^2}}
        \eqcm\qquad
        \Lambda\pr(x) = \sum_{n\geq1}\frac{2\kappa x}{4n^2\pi^2-\kappa x^2}
        \eqcm\\
        \chi_\Lambda(x) = \sum_{n\geq1}\frac{2\kappa}{4n^2\pi^2-\kappa x^2}
        \eqcm\qquad
        \Lambda\prr(x) = \sum_{n\geq1}\frac{2\kappa\br{4n^2\pi^2+\kappa x^2}}{\br{4n^2\pi^2-\kappa x^2}^2}
        \eqfs
    \end{gathered}
    \end{equation}
    In particular $\Lambda$, $\Lambda\pr$, $\Lambda\prr$ and $\chi_\Lambda$ are nonnegative and increasing on $[0,\Dk)$, with $\Lambda(0)=\Lambda\pr(0)=0$ and
    \begin{equation}\label{eq:Lambda-at-zero}
        \chi_\Lambda(0)=\Lambda\prr(0)=\frac{\kappa}{12}
        \eqcm
    \end{equation}
    and $\Kap$ is increasing and continuous on $[0,\Dk)$, its value $\Kap(0)=\kappa/12$ prescribed in \eqref{eq:Kap} being the continuous extension.
\end{lemma}

\begin{proof}
    Put $w:=\sqrt\kappa\,x/2\in[0,\pi/2)$, so that $\Lambda(x)=\log(w/\sin w)$ and $w^2/(n^2\pi^2)=\kappa x^2/(4n^2\pi^2)$.
    Euler's product $\sin w/w=\prod_{n\geq1}\br{1-w^2/(n^2\pi^2)}$ converges locally uniformly on $\abs{w}<\pi$, and none of its factors vanishes there, so taking logarithms gives the first formula, valid for $\abs{x}<2\Dk$; the series of logarithms converges locally uniformly there, hence may be differentiated term by term, which gives the other three.
    In each of the last three series every summand is nonnegative on $[0,\Dk)$ and increasing there---the numerator does not decrease and the positive denominator does not increase---so the same holds for the sums, and $\Lambda(x)=\int_0^x\Lambda\pr\geq0$ is increasing as well.
    Evaluating the third series at $x=0$ gives $\sum_{n\geq1}\kappa/(2n^2\pi^2)=\kappa/12$, and the fourth gives the same value; that these are the limits of $\chi_\Lambda$ and $\Lambda\prr$ at $0$ is again locally uniform convergence.
    Finally $\Kap=\Th\cdot\chi_\Lambda$ is a product of two positive increasing continuous functions, $\Th$ by \eqref{eq:Theta} and $\chi_\Lambda$ by the above; since $\Th(0)=1$, its limit at $0$ is $\kappa/12$.
\end{proof}

\begin{lemma}[Tightness identity]\label{lem:tightness}
    For $0\leq x<\Dk$,
    \begin{equation}\label{eq:tightness}
        \Th\brOf{\tfrac x2}^2+\Kap(x)\,x^2 = \Th(x)
        \eqfs
    \end{equation}
    In particular $\Th(x/2)^2\leq\Th(x)$, with equality only at $x=0$.
\end{lemma}

\begin{proof}
    Put $w:=\sqrt\kappa\,x/2$, so that
    \begin{equation*}
        \Th\brOf{\tfrac x2}=\frac{w}{\sin w}
        \eqcm\qquad
        \Th(x)=\frac{2w}{\sin 2w}=\frac{w}{\sin w\cos w}
        \eqfs
    \end{equation*}
    Differentiating $\Lambda(x)=\log w-\log\sin w$ gives $\Lambda\pr(x)=\tfrac{\sqrt\kappa}{2}\br{1/w-\cot w}$, hence
    \begin{equation*}
        x\,\Lambda\pr(x) = w\brOf{\frac1w-\cot w} = 1-w\cot w
        \eqfs
    \end{equation*}
    Since $\Kap(x)x^2=\Th(x)\,\chi_\Lambda(x)\,x^2=\Th(x)\,x\,\Lambda\pr(x)$, this gives $\Kap(x)x^2=\Th(x)\br{1-w\cot w}$, while
    \begin{equation*}
        \Th(x)\,w\cot w = \frac{w}{\sin w\cos w}\cdot\frac{w\cos w}{\sin w}
        = \frac{w^2}{\sin^2w} = \Th\brOf{\tfrac x2}^2
        \eqcm
    \end{equation*}
    and adding the two displays gives \eqref{eq:tightness}.
    The last assertion follows because $\Kap>0$ and $x^2>0$ for $x>0$.
\end{proof}

\begin{lemma}\label{lem:Lambda-second}
    For $0<x<\Dk$ one has $0<\Lambda\prr(x)\leq\Kap(x)$.
    Consequently, for $0\leq\rho<\Dk$,
    \begin{equation}\label{eq:sup-K}
        \sup_{0\leq x\leq\rho}\ \max\brOf{\Lambda\prr(x),\ \Kap(x)} = \Kap(\rho)
        \eqfs
    \end{equation}
\end{lemma}

\begin{proof}
    Positivity is \cref{lem:Lambda}.
    For the upper bound, note that $\Lambda$ depends on $\kappa$ only through the scaling $\Lambda(x)=\Lambda_1\br{\sqrt\kappa\,x}$, where $\Lambda_1$ is the function \eqref{eq:Lam} for $\kappa=1$; hence $\Lambda\prr(x)=\kappa\Lambda_1\prr\br{\sqrt\kappa\,x}$ and, by \eqref{eq:Kap}, $\Kap(x)=\kappa\,\mathrm{K}_1\br{\sqrt\kappa\,x}$.
    So we may and do assume $\kappa=1$, and the claim is $\Lambda\prr(x)\sin x\leq\Lambda\pr(x)$ for $0<x<\pi$.
    Here $\Lambda\pr(x)=\tfrac1x-\tfrac12\cot\tfrac x2$ and $\Lambda\prr(x)=-\tfrac1{x^2}+\tfrac1{4\sin^2(x/2)}$.
    Substituting $w=x/2\in(0,\pi/2)$ and $\sin x=2\sin w\cos w$, the claim reads
    \begin{equation*}
        \frac{\cos w}{2\sin w}-\frac{\sin w\cos w}{2w^2}
        \ \leq\ \frac{1}{2w}-\frac{\cos w}{2\sin w}
        \eqcm
    \end{equation*}
    Multiplying by $2w^2\sin w>0$ and rearranging, this is equivalent to
    \begin{equation*}
        2w^2\cos w\ \leq\ w\sin w+\sin^2w\cos w
        \eqcm
    \end{equation*}
    that is, after division by $\cos w>0$, to
    \begin{equation}\label{eq:amgm-target}
        2w^2\ \leq\ \sin^2w+w\tan w
        \qquad\br{0<w<\tfrac\pi2}
        \eqfs
    \end{equation}
    The arithmetic--geometric mean inequality gives
    \begin{equation*}
        \sin^2w+w\tan w\ \geq\ 2\sqrt{w\sin^3w/\cos w}
        \eqcm
    \end{equation*}
    so \eqref{eq:amgm-target} follows from
    \begin{equation}\label{eq:sinw-cube}
        \brOf{\frac{\sin w}{w}}^{3}\ \geq\ \cos w
        \qquad\br{0<w<\tfrac\pi2}
        \eqcm
    \end{equation}
    since \eqref{eq:sinw-cube} gives $\sin^3w/\cos w\geq w^3$ and hence $2\sqrt{w\cdot w^3}=2w^2$.

    To prove \eqref{eq:sinw-cube} put $h(w):=3\log\sin w-3\log w-\log\cos w$, so that $h(0^+)=0$ and $h\pr(w)=3\cot w+\tan w-3/w$.
    Multiplying by $w\sin w\cos w>0$, the inequality $h\pr>0$ is equivalent to $w\br{3\cos^2w+\sin^2w}>3\sin w\cos w$, that is, using $3\cos^2w+\sin^2w=\cos2w+2$ and $2\sin w\cos w=\sin2w$, to $\psi(s)>0$ for $s=2w\in(0,\pi)$, where
    \begin{equation*}
        \psi(s) := s\br{\cos s+2}-3\sin s
        \eqfs
    \end{equation*}
    Now $\psi(0)=\psi\pr(0)=\psi\prr(0)=0$ and $\psi^{\prime\prime\prime}(s)=s\sin s>0$ on $(0,\pi)$, so successively $\psi\prr>0$, $\psi\pr>0$ and $\psi>0$ there.
    Hence $h\pr>0$, so $h>0$, which is \eqref{eq:sinw-cube}.

    For \eqref{eq:sup-K}: on $\ab{0,\rho}$ both $\Lambda\prr$ and $\Kap$ are at most $\Kap(\rho)$, the first by what was just proved together with the monotonicity of $\Kap$ and the second by that monotonicity alone; and the value $\Kap(\rho)$ is attained at $x=\rho$.
\end{proof}

\subsubsection*{The mixed second difference}

Let $\gamma\colon\ab{0,u}\to\mc Q$ be the unit-speed geodesic from $y$ to $z$ and $\eta\colon\ab{0,v}\to\mc Q$ the unit-speed geodesic from $q$ to $p$, and put
\begin{equation}\label{eq:Xi-def}
    \Xi(\sigma,\varphi) := \ol{\gamma(\sigma)}{\eta(\varphi)}
    \eqfs
\end{equation}
The four corners of the rectangle $\ab{0,u}\times\ab{0,v}$ are
\begin{equation*}
    \Xi(0,0)=a,\qquad \Xi(u,v)=e,\qquad \Xi(0,v)=b,\qquad \Xi(u,0)=c
    \eqcm
\end{equation*}
so that $\Quad_\Lambda(y,z;q,p)=\Lambda(a)+\Lambda(e)-\Lambda(b)-\Lambda(c)$ is the \emph{mixed second difference} of $\Lambda\circ\Xi$ over that rectangle.
Two consequences follow, one combinatorial and one differential.

\begin{lemma}[Exact additivity]\label{lem:additivity}
    Let $0=\sigma_0<\dots<\sigma_N=u$ and $0=\varphi_0<\dots<\varphi_N=v$, and put $y_i:=\gamma(\sigma_i)$ and $q_j:=\eta(\varphi_j)$.
    Then
    \begin{equation}\label{eq:additivity}
        \Quad_\Lambda(y,z;q,p)
        = \sum_{i,j=0}^{N-1}\Quad_\Lambda\br{y_i,y_{i+1};q_j,q_{j+1}}
        \eqfs
    \end{equation}
    Moreover the two bases of the quadruple $\br{y_i,y_{i+1};q_j,q_{j+1}}$ are $\sigma_{i+1}-\sigma_i$ and $\varphi_{j+1}-\varphi_j$.
\end{lemma}

\begin{proof}
    Write $F(\sigma,\varphi):=\Lambda\br{\Xi(\sigma,\varphi)}$.
    By \cref{def:Q} the summand indexed by $(i,j)$ is
    \begin{equation*}
        F(\sigma_i,\varphi_j)+F(\sigma_{i+1},\varphi_{j+1})-F(\sigma_i,\varphi_{j+1})-F(\sigma_{i+1},\varphi_j)
        \eqcm
    \end{equation*}
    the mixed second difference of $F$ over the cell $\ab{\sigma_i,\sigma_{i+1}}\times\ab{\varphi_j,\varphi_{j+1}}$.
    Summing mixed second differences over a grid telescopes, first in $j$ and then in $i$, to the mixed second difference over the whole rectangle, which is $\Quad_\Lambda(y,z;q,p)$.
    The last assertion holds because a subsegment of a geodesic is a geodesic, so $\ol{y_i}{y_{i+1}}=\sigma_{i+1}-\sigma_i$ and $\ol{q_j}{q_{j+1}}=\varphi_{j+1}-\varphi_j$.
\end{proof}

\begin{lemma}[The mixed derivative in the model plane]\label{lem:mixed}
    Let $\bar\gamma,\bar\eta$ be unit-speed geodesics of $\Mk$, let $\bar\Xi$ be given by \eqref{eq:Xi-def} for them, and let $f$ be twice continuously differentiable with $f\pr\geq0$.
    At every parameter pair with $0<\bar\Xi<\Dk$,
    \begin{equation}\label{eq:mixed}
        \partial_\sigma\partial_\varphi f\br{\bar\Xi}
        = -\alpha\mu\,f\prr\br{\bar\Xi}-\beta\nu\,\Th\br{\bar\Xi}\chi_f\br{\bar\Xi}
        \eqcm
    \end{equation}
    where $(\alpha,\beta)$ and $(\mu,\nu)$ are unit vectors of $\R^2$ depending on the pair.
    Consequently
    \begin{equation}\label{eq:mixed-bound}
        \abs{\partial_\sigma\partial_\varphi f\br{\bar\Xi}}
        \ \leq\ \max\brOf{\abs{f\prr\br{\bar\Xi}},\ \Th\br{\bar\Xi}\chi_f\br{\bar\Xi}}
        \eqfs
    \end{equation}
\end{lemma}

\begin{proof}
    Both sides of \eqref{eq:mixed} are multiplied by $\lambda^{-2}$ when the metric is scaled by $\lambda>0$, that is, when $\kappa$ is replaced by $\lambda^{-2}\kappa$ and $f$ by $f(\cdot/\lambda)$: the left-hand side because unit-speed geodesics are reparametrized by $\lambda$ and $\bar\Xi$ is multiplied by it, the right-hand side because $f\prr$ and $\chi_f$ both acquire the factor $\lambda^{-2}$ while $\Th\br{\bar\Xi}$ is unchanged.
    So we may and do assume $\kappa=1$, with $\Mk=S^2$ the unit sphere of $\R^3$, $\Th(x)=x/\sin x$ and $\bar\Xi=\arccos\ip GE$ for $G:=\bar\gamma(\sigma)$ and $E:=\bar\eta(\varphi)$.

    Fix a parameter pair with $0<\bar\Xi<\pi$ and choose an orthonormal frame $\cb{G,w,n}$ of $\R^3$ with $E=\cos\bar\Xi\,G+\sin\bar\Xi\,w$.
    The tangent space at $G$ is spanned by $w$ and $n$, so $\bar\gamma\pr=\alpha w+\beta n$ with $\alpha^2+\beta^2=1$; the tangent space at $E$ is spanned by $n$ and $w_E:=-\sin\bar\Xi\,G+\cos\bar\Xi\,w$, so $\bar\eta\pr=\mu w_E+\nu n$ with $\mu^2+\nu^2=1$.
    Differentiating $\cos\bar\Xi=\ip GE$ in $\sigma$ and in $\varphi$ gives
    \begin{equation*}
        -\sin\bar\Xi\ \partial_\sigma\bar\Xi=\ip{\bar\gamma\pr}{E}=\alpha\sin\bar\Xi
        \eqcm\qquad
        -\sin\bar\Xi\ \partial_\varphi\bar\Xi=\ip{G}{\bar\eta\pr}=-\mu\sin\bar\Xi
        \eqcm
    \end{equation*}
    that is $\partial_\sigma\bar\Xi=-\alpha$ and $\partial_\varphi\bar\Xi=\mu$.
    Differentiating the first of these in $\varphi$,
    \begin{equation*}
        -\cos\bar\Xi\,\br{\partial_\varphi\bar\Xi}\br{\partial_\sigma\bar\Xi}-\sin\bar\Xi\ \partial_\sigma\partial_\varphi\bar\Xi
        = \ip{\bar\gamma\pr}{\bar\eta\pr} = \alpha\mu\cos\bar\Xi+\beta\nu
        \eqcm
    \end{equation*}
    whose left-hand side is $\alpha\mu\cos\bar\Xi-\sin\bar\Xi\,\partial_\sigma\partial_\varphi\bar\Xi$; the terms $\alpha\mu\cos\bar\Xi$ cancel and there remains
    $\partial_\sigma\partial_\varphi\bar\Xi=-\beta\nu/\sin\bar\Xi$.
    Now
    \begin{equation*}
        \partial_\sigma\partial_\varphi f\br{\bar\Xi}
        = f\prr\br{\bar\Xi}\,\partial_\sigma\bar\Xi\,\partial_\varphi\bar\Xi
            + f\pr\br{\bar\Xi}\,\partial_\sigma\partial_\varphi\bar\Xi
        = -\alpha\mu f\prr\br{\bar\Xi}-\frac{\beta\nu\,f\pr\br{\bar\Xi}}{\sin\bar\Xi}
        \eqcm
    \end{equation*}
    which is \eqref{eq:mixed}, because $f\pr(x)/\sin x=\Th(x)\chi_f(x)$ at $\kappa=1$.
    For \eqref{eq:mixed-bound}, the Cauchy--Schwarz inequality gives
    $\abs{\alpha\mu}+\abs{\beta\nu}\leq\sqrt{\alpha^2+\beta^2}\sqrt{\mu^2+\nu^2}=1$, so \eqref{eq:mixed} is at most $\max\br{\abs{f\prr},\Th\chi_f}\br{\abs{\alpha\mu}+\abs{\beta\nu}}$ in absolute value.
\end{proof}

\begin{proposition}[Model-plane bound]\label{prop:model-mixed}
    Let $\bar\gamma\colon\ab{0,u\pr}\to\Mk$ and $\bar\eta\colon\ab{0,v\pr}\to\Mk$ be unit-speed geodesics with $\bar\Xi\leq\rho<\Dk$ throughout, and write $\bar y:=\bar\gamma(0)$, $\bar z:=\bar\gamma(u\pr)$, $\bar q:=\bar\eta(0)$, $\bar p:=\bar\eta(v\pr)$.
    Then
    \begin{equation}\label{eq:model-mixed}
        \Quad_\Lambda\br{\bar y,\bar z;\bar q,\bar p}\ \leq\ \Kap(\rho)\,u\pr v\pr
        \eqfs
    \end{equation}
\end{proposition}

\begin{proof}
    The first series in \eqref{eq:Lambda-series} exhibits $\Lambda$ as a real-analytic function of $x^2$ on $\abs{x}<2\Dk$.
    Writing $G:=\bar\gamma(\sigma)$ and $E:=\bar\eta(\varphi)$ as points of $\Mk\subset\R^3$, so that $\ip GE=\kappa^{-1}\cos\br{\sqrt\kappa\,\bar\Xi}$, we have
    $\bar\Xi^2=\kappa^{-1}\brOf{\arccos\br{\kappa\ip GE}}^2$, and $t\mapsto\br{\arccos t}^2$ is real-analytic near $t=1$.
    Since $\bar\Xi\leq\rho<\Dk$, the composite $\Lambda\circ\bar\Xi$ is therefore smooth on all of $\ab{0,u\pr}\times\ab{0,v\pr}$, \emph{including at parameter pairs where $\bar\Xi=0$}, and
    \begin{equation*}
        \Quad_\Lambda\br{\bar y,\bar z;\bar q,\bar p}
        = \int_0^{u\pr}\!\!\int_0^{v\pr}\partial_\sigma\partial_\varphi\,\Lambda\br{\bar\Xi}\dl\varphi\dl\sigma
        \eqcm
    \end{equation*}
    the left-hand side being the mixed second difference of $\Lambda\circ\bar\Xi$ over the rectangle.
    On the open set where $\bar\Xi>0$, \cref{lem:mixed} with $f=\Lambda$ and \cref{lem:Lambda-second} via \eqref{eq:sup-K} bound the integrand by $\Kap(\rho)$; the integrand is continuous, and so is $\Kap$ at $0$ by \cref{lem:Lambda}, so the same bound holds where $\bar\Xi=0$.
    Integrating gives \eqref{eq:model-mixed}.
\end{proof}

\subsubsection*{From the model plane to \texorpdfstring{$\mc Q$}{Q}}

The passage to a general \catk{} space is by subdivision and the four-point subembedding of Bridson and Haefliger, which we recall in the labeling \eqref{eq:six}.

\begin{proposition}[{Four-point subembedding \cite[Def.~II.1.10 and Prop.~II.1.11]{bridson99}}]\label{prop:subembed}
    Let $y\pr,z\pr,q\pr,p\pr$ be four points of a \catk{} space whose $4$-cycle perimeter satisfies
    $\ol{y\pr}{z\pr}+\ol{z\pr}{q\pr}+\ol{q\pr}{p\pr}+\ol{p\pr}{y\pr}<2\Dk$.
    Then there are $\bar y\pr,\bar z\pr,\bar q\pr,\bar p\pr\in\Mk$ with the four sides equal to the corresponding sides of the original,
    \begin{equation*}
        \ol{\bar y\pr}{\bar z\pr}=\ol{y\pr}{z\pr},\quad
        \ol{\bar z\pr}{\bar q\pr}=\ol{z\pr}{q\pr},\quad
        \ol{\bar q\pr}{\bar p\pr}=\ol{q\pr}{p\pr},\quad
        \ol{\bar p\pr}{\bar y\pr}=\ol{p\pr}{y\pr}
        \eqcm
    \end{equation*}
    and with both diagonals no smaller,
    $\ol{\bar y\pr}{\bar q\pr}\geq\ol{y\pr}{q\pr}$ and $\ol{\bar z\pr}{\bar p\pr}\geq\ol{z\pr}{p\pr}$.
\end{proposition}

\begin{proposition}[The $\Lambda$-bound]\label{prop:Lambda-bound}
    Let $y,z,q,p$ satisfy \eqref{eq:rdiam}.
    Then
    \begin{equation}\label{eq:Lambda-bound}
        \Quad_\Lambda(y,z;q,p)\ \leq\ \Kap(r)\,\ol yz\,\ol pq
        \eqfs
    \end{equation}
\end{proposition}

\begin{proof}
    If $u=0$ then $y=z$, so $a=c$ and $e=b$ and both sides of \eqref{eq:Lambda-bound} vanish; likewise if $v=0$.
    So assume $u,v>0$.
    Fix $N\in\mathbb N$ and take the uniform partitions $\sigma_i:=iu/N$ and $\varphi_j:=jv/N$ of \cref{lem:additivity}, so that every cell has $u\pr=u/N$ and $v\pr=v/N$.
    Put $\varepsilon_N:=(u+v)/N$ and $\rho_N:=r+\varepsilon_N$, and let $N$ be large enough that $\rho_N<\Dk$.

    Fix a cell and consider the quadruple $\br{y_i,y_{i+1};q_j,q_{j+1}}$, whose six distances we write $a\pr,e\pr,b\pr,c\pr,u\pr,v\pr$ as in \eqref{eq:six}.
    Each of $a\pr,b\pr,c\pr,e\pr$ is the distance between a point of $\seg yz$ and a point of $\seg pq$, hence at most $r$ by \eqref{eq:rdiam}, so the $4$-cycle perimeter of the quadruple is
    \begin{equation*}
        u\pr+c\pr+v\pr+b\pr\ \leq\ \varepsilon_N+2r\ <\ 2\Dk
    \end{equation*}
    for $N$ large.
    \Cref{prop:subembed} therefore applies and yields $\bar y\pr,\bar z\pr,\bar q\pr,\bar p\pr\in\Mk$ with $\bar u\pr=u\pr$, $\bar v\pr=v\pr$, $\bar b\pr=b\pr$, $\bar c\pr=c\pr$, $\bar a\pr\geq a\pr$ and $\bar e\pr\geq e\pr$.
    As $\Lambda$ is nondecreasing (\cref{lem:Lambda}),
    \begin{equation*}
        \Quad_\Lambda\br{y_i,y_{i+1};q_j,q_{j+1}}\ \leq\ \Quad_\Lambda\br{\bar y\pr,\bar z\pr;\bar q\pr,\bar p\pr}
        \eqfs
    \end{equation*}
    Parametrize $\seg{\bar y\pr}{\bar z\pr}$ by $\bar\gamma$ and $\seg{\bar q\pr}{\bar p\pr}$ by $\bar\eta$, both at unit speed and in that direction.
    Going from $\bar\gamma(\sigma)$ back to $\bar y\pr$, across the side $\ol{\bar y\pr}{\bar p\pr}=b\pr$ and forward to $\bar\eta(\varphi)$, the triangle inequality gives
    \begin{equation*}
        \bar\Xi(\sigma,\varphi)\ \leq\ \sigma+\ol{\bar y\pr}{\bar p\pr}+\br{v\pr-\varphi}
        \ \leq\ u\pr+b\pr+v\pr\ \leq\ r+\varepsilon_N=\rho_N
        \eqcm
    \end{equation*}
    so \cref{prop:model-mixed} applies with $\rho=\rho_N$ and bounds the right-hand side by $\Kap(\rho_N)u\pr v\pr$.
    Summing over the $N^2$ cells and using \cref{lem:additivity},
    \begin{equation*}
        \Quad_\Lambda(y,z;q,p)\ \leq\ N^2\cdot\frac uN\cdot\frac vN\cdot\Kap(\rho_N)
        = \Kap(\rho_N)\,uv
        \eqfs
    \end{equation*}
    Letting $N\to\infty$ gives $\rho_N\downarrow r$, and $\Kap$ is continuous at $r<\Dk$ by \cref{lem:Lambda}, so the right-hand side tends to $\Kap(r)uv$.
\end{proof}

\subsubsection*{Proof of the inequality}

One elementary observation remains.
It is what allows the two bounds available for the defect---one from the size of the factors, one from \cref{prop:Lambda-bound}---to be traded against each other at the optimal rate.

\begin{lemma}\label{lem:minsplit}
    Let $A>0$ and $E\geq0$. Then $\min\br{A-\Psi,\ \Psi E}\leq AE/(1+E)$ for every $\Psi\in(0,A]$.
\end{lemma}

\begin{proof}
    If $\Psi\leq A/(1+E)$ then $\Psi E\leq AE/(1+E)$; otherwise $A-\Psi<A-A/(1+E)=AE/(1+E)$.
\end{proof}

\begin{theorem}[Ptolemy's inequality in \catk{} spaces]\label{thm:ptolemy}
    Let $y,z,q,p$ satisfy \eqref{eq:rdiam}.
    Then
    \begin{equation}\label{eq:ptolemy}
        \ol yq\,\ol zp\ \leq\ \ol yp\,\ol zq+\Th(r)\,\ol yz\,\ol pq
        \eqfs
    \end{equation}
\end{theorem}

\begin{proof}
    Use the abbreviations \eqref{eq:six} and write $\ch x:=\chmap(x)$ for the chord of a distance $x$, so that $x=e^{\Lambda(x)}\ch x$ by \eqref{eq:chord-factor} and \eqref{eq:Lam}.
    Put
    \begin{equation*}
        \Pidia := \Th\brOf{\tfrac a2}\Th\brOf{\tfrac e2}
        \eqcm\qquad
        \Pileg := \Th\brOf{\tfrac b2}\Th\brOf{\tfrac c2}
        \eqcm
    \end{equation*}
    so that $ae=\Pidia\,\ch a\ch e$ and $bc=\Pileg\,\ch b\ch c$, and
    \begin{equation}\label{eq:ratio-is-exp}
        \frac{\Pidia}{\Pileg} = \exp\brOf{\Quad_\Lambda(y,z;q,p)}
        \eqfs
    \end{equation}
    All six distances are at most $r$ and $\Th$ is increasing, so
    \begin{equation}\label{eq:PiPi-bounds}
        1\leq\Pileg\leq\Th\brOf{\tfrac r2}^2
        \eqcm\qquad
        1\leq\Pidia\leq\Th\brOf{\tfrac r2}^2
        \eqfs
    \end{equation}
    We abbreviate $\Pi_r:=\Th(r/2)^2$, the common upper bound in \eqref{eq:PiPi-bounds}, and $E:=e^{\Kap(r)uv}-1\geq0$.

    \emph{Step 1: split off the chordal inequality.}
    All six distances are at most $r<\Dk$, so the truncation in \cref{def:chord} is inactive and the chordal distances of the quadruple are $\ch a,\ch e,\ch b,\ch c,\ch u,\ch v$; hence \eqref{eq:chordP} gives $\ch a\ch e\leq\ch b\ch c+\ch u\ch v$.
    Hence
    \begin{equation}\label{eq:ptolemy-split}
        ae-bc = \Pidia\,\ch a\ch e-\Pileg\,\ch b\ch c
        \ \leq\ \Pidia\,\ch u\ch v+\posp{\Pidia-\Pileg}\ch b\ch c
        \eqfs
    \end{equation}

    \emph{Step 2: the main term.}
    By \eqref{eq:PiPi-bounds} and $\chmap\leq\mathrm{id}$,
    \begin{equation}\label{eq:ptolemy-main}
        \Pidia\,\ch u\ch v\ \leq\ \Pi_r\,uv
        \eqfs
    \end{equation}

    \emph{Step 3: two bounds on the defect.}
    First, $\Pidia\leq\Pi_r$ gives $\posp{\Pidia-\Pileg}\leq\Pi_r-\Pileg$, the right-hand side being nonnegative by \eqref{eq:PiPi-bounds}.
    Second, by \eqref{eq:ratio-is-exp} and \cref{prop:Lambda-bound},
    \begin{equation*}
        \posp{\Pidia-\Pileg} = \Pileg\posp{e^{\Quad_\Lambda(y,z;q,p)}-1}
        \ \leq\ \Pileg\br{e^{\Kap(r)uv}-1} = \Pileg E
        \eqfs
    \end{equation*}
    So $\posp{\Pidia-\Pileg}\leq\min\br{\Pi_r-\Pileg,\ \Pileg E}$, and \cref{lem:minsplit} with $A=\Pi_r$ and $\Psi=\Pileg\in(0,\Pi_r]$ turns this into
    \begin{equation}\label{eq:defect-bound}
        \posp{\Pidia-\Pileg}
        \ \leq\ \frac{\Pi_r E}{1+E}
        = \Pi_r\brOf{1-\frac1{1+E}}
        = \Pi_r\br{1-e^{-\Kap(r)uv}}
        \eqfs
    \end{equation}

    \emph{Step 4: the defect term.}
    Since $b,c\leq r$ and $\chmap$ is increasing, $\ch b\ch c\leq\chmap(r)^2$, and $\chmap(r)^2\,\Pi_r=r^2$ by \eqref{eq:chord-factor}.
    So \eqref{eq:defect-bound} gives, using $1-e^{-t}\leq t$ for $t\geq0$,
    \begin{equation}\label{eq:ptolemy-defect}
        \posp{\Pidia-\Pileg}\ch b\ch c
        \ \leq\ r^2\br{1-e^{-\Kap(r)uv}}
        \ \leq\ \Kap(r)\,r^2\,uv
        \eqfs
    \end{equation}

    Combining \eqref{eq:ptolemy-split}, \eqref{eq:ptolemy-main} and \eqref{eq:ptolemy-defect} and then applying the tightness identity \eqref{eq:tightness} at $x=r$,
    \begin{equation*}
        ae-bc\ \leq\ \Pi_r\,uv+\Kap(r)r^2\,uv
        = \brOf{\Th\brOf{\tfrac r2}^2+\Kap(r)r^2}uv
        = \Th(r)\,uv
        \eqcm
    \end{equation*}
    which is \eqref{eq:ptolemy}.
\end{proof}

\subsection{The inputs assembled}\label{app:catk:inputs}

\begin{remark}[$\Th(r)$ is sharp]\label{rem:theta-sharp}
    The constant of \eqref{eq:RC} cannot be improved for any $r\in(0,\Dk)$, even on $\Mk$.
    Let $\mc B:=\bar B(N,r/2)\subset\Mk$ be the closed ball of radius $r/2$ about a point $N$; it is convex, being of radius less than $\tfrac12\Dk$, and of diameter $r$.
    For $0<\phi<\pi$ let $y,z,q,p\in\partial\mc B$ be the points of longitude $0$, $\phi$, $\pi$ and $\pi+\phi$ respectively, $N$ being the pole.
    Two points of colatitude $\hat\varrho:=\sqrt\kappa\,r/2$ whose longitudes differ by $\alpha$ are at distance $d$ with $\cos(\sqrt\kappa\,d)=\cos^2\hat\varrho+\sin^2\hat\varrho\cos\alpha$.
    A longitude difference of $\pi$ gives $\ol yq=\ol zp=r$, and differences $\phi$ and $\pi\pm\phi$ give
    \begin{equation}\label{eq:sharpconf}
        \cos\br{\sqrt\kappa\,\ol yz}=\cos^2\hat\varrho+\sin^2\hat\varrho\cos\phi
        \eqcm\qquad
        \cos\br{\sqrt\kappa\,\ol yp}=\cos^2\hat\varrho-\sin^2\hat\varrho\cos\phi
        \eqcm
    \end{equation}
    with $\ol yp=\ol zq$ and $\ol yz=\ol pq$.
    Put $\varepsilon:=\sin^2\hat\varrho\,(1-\cos\phi)$, so that $\cos\br{\sqrt\kappa\,\ol yz}=1-\varepsilon$ and $\cos\br{\sqrt\kappa\,\ol yp}=\cos(\sqrt\kappa r)+\varepsilon$, and let $\phi\downarrow0$, that is $\varepsilon\downarrow0$.
    Then $\ol yp=r-\varepsilon/\br{\sqrt\kappa\sin(\sqrt\kappa r)}+O(\varepsilon^2)$ and $\olt yz=2\varepsilon/\kappa+O(\varepsilon^2)$, hence
    \begin{equation*}
    \begin{gathered}
        \Quad_{\sqmap}(y,z;q,p)=2\br{r-\ol yp}\br{r+\ol yp}=\frac{4r\varepsilon}{\sqrt\kappa\sin(\sqrt\kappa r)}+O(\varepsilon^2)
        \eqcm\\
        2\,\ol yz\,\ol pq=2\,\olt yz=\frac{4\varepsilon}{\kappa}+O(\varepsilon^2)
        \eqcm
    \end{gathered}
    \end{equation*}
    whose ratio tends to $\sqrt\kappa\,r/\sin(\sqrt\kappa r)=\Th(r)$.
    The quadruple lies in the convex set $\mc B$, so $\diam\br{\seg yz\cup\seg pq}\leq r$, with equality because $\ol yq=r$.
    The same expansions give, with $s$ and $\Am$ as in \eqref{eq:sA},
    \begin{equation}\label{eq:sharplim}
        \frac{\Quad_{\mathrm{id}}(y,z;q,p)}{2\br{\Am-s}}\ \longrightarrow\ \Th(r)
        \eqcm\qquad
        \frac{\ol yq\,\ol zp-\ol yp\,\ol zq}{\ol yz\,\ol pq}\ \longrightarrow\ \Th(r)
        \eqcm
    \end{equation}
    so $\Th(r)$ is a lower bound for any constant admissible in \eqref{eq:PC} as well; and, since \eqref{eq:TC} at $\tran=\sqmap$ reduces to \eqref{eq:RC}, for \eqref{eq:TC} too.
    The first limit shows in addition that \eqref{eq:TC} at $\tran=\mathrm{id}$, the linear bound entering the proof of \cref{thm:intro:positive}, carries the least possible constant.
    This one family therefore settles every claim about the constant at once.
\end{remark}

We can now assemble the two theorems of \cref{ssec:intro:positive}.

\begin{proof}[Proof of \cref{thm:intro:RC,thm:intro:PC}]
    By \eqref{eq:rdiam} and \eqref{eq:Theta-CR} we have $\Th(r)=C_R=R/\sin R$, with $\Th(0)=1$.
    Inequality \eqref{eq:RC} is \cref{prop:firstvar} and \eqref{eq:PC} is \cref{thm:ptolemy}.
    That neither constant can be lowered is \cref{rem:theta-sharp}: the family constructed there has $\diam\br{\seg yz\cup\seg pq}=r$ and realizes $\Th(r)$ in the limit for both ratios.
\end{proof}

\begin{remark}[The range $R<\pi$ is optimal]\label{rem:range}
    Let $\mc Q$ be a circle of circumference $2\Dk$ with its arclength metric, a \catk{} space of diameter $\Dk$, and let $\theta\in\R/2\Dk\mathbb Z$ be the arclength coordinate.
    For $\epsilon,\ell>0$ with $2\epsilon+\ell<\Dk$ take $y=0$, $z=\ell$, $q=\Dk-\epsilon$ and $p=\Dk+\epsilon+\ell$.
    Then $\ol yq=\ol zp=\Dk-\epsilon$, $\ol yp=\ol zq=\Dk-\epsilon-\ell$, $\ol yz=\ell$ and $\ol pq=2\epsilon+\ell$, so
    \begin{equation}\label{eq:range-optimal}
        \frac{\Quad_{\sqmap}(y,z;q,p)}{2\,\ol yz\,\ol pq}=\frac{2(\Dk-\epsilon)-\ell}{2\epsilon+\ell}
        \eqcm
    \end{equation}
    which is unbounded as $\epsilon,\ell\downarrow0$.
    Here $\seg yz\cup\seg pq$ contains a pair of antipodal points, so $R=\pi$: no constant works at the right endpoint of the range, in \cref{thm:intro:RC,thm:intro:PC} or in \cref{thm:intro:positive}.
\end{remark}

\section{Deferred proofs}\label{app:deferred}

\subsection{When both hypotheses are equalities}\label{app:tight}

On the sextuples for which \eqref{eq:r} and \eqref{eq:p} are both equalities, \eqref{eq:t} collapses to a statement about one function of one variable.

\begin{lemma}[Both hypotheses tight]\label{lem:rp-tight}
Let $a,e,b,c,u,v\geq0$ satisfy \eqref{eq:r} and \eqref{eq:p} with equality, and let $s$ and $\Am$ be as in \cref{thm:algebraic}.
Then $\abs{a-e}=\abs{b-c}$ and $a+e=2\Am$; hence, after relabeling if necessary such that $a\leq e$ and $b\leq c$,
\begin{equation}\label{eq:eq-shift}
  a = b+w \eqcm \qquad e = c+w \eqcm \qquad\text{where } w := \Am-s \geq 0
  \eqfs
\end{equation}
Writing $g(x) := \tran(x+w)-\tran(x)$, inequality \eqref{eq:t} reads
\begin{equation}\label{eq:eq-gform}
  g(b) + g(c) \leq 2g\brOf{\tfrac{b+c}{2}}
  \eqcm
\end{equation}
and, unless $b=c$ or $w=0$, it is an equality if and only if $\dtran$ is affine on $\ab{\min(b,c),\max(b,c)+w}$; if $b=c$ or $w=0$ it is an equality for every $\tran\in\setcc$.
\end{lemma}

\begin{proof}
Subtracting twice \eqref{eq:p} from \eqref{eq:r}, both equalities by assumption, gives $(a-e)^2=(b-c)^2$, and then
\begin{equation*}
  (a+e)^2 = (a-e)^2+4ae = (b-c)^2+4\br{bc+uv} = (b+c)^2+4uv = 4\Am^2
  \eqfs
\end{equation*}
So $\cb{b,c}=\cb{s-\delta,s+\delta}$ and $\cb{a,e}=\cb{\Am-\delta,\Am+\delta}$ with $\delta:=\tfrac12\abs{b-c}$, which is \eqref{eq:eq-shift}; both sides of \eqref{eq:t} are symmetric in $a\leftrightarrow e$ and in $b\leftrightarrow c$, so the relabeling is free.
With \eqref{eq:eq-shift} the left-hand side of \eqref{eq:t} is $g(b)+g(c)$, and the right-hand side is $2(\tran(s+w)-\tran(s))=2g(s)$ because $\Am=s+w$; this is \eqref{eq:eq-gform}, that is, midpoint concavity of $g$.
Now $g\pr(x)=\dtran(x+w)-\dtran(x)$, and this is nonincreasing in $x$ for every $w\geq0$ precisely when $\dtran$ is concave; so $g$ is concave and \eqref{eq:eq-gform} holds.
If $b=c$ or $w=0$ both sides agree for every $\tran$.
Otherwise equality forces $g$ to be affine on $\ab{b,c}$, that is $\ddrtran(x+w)=\ddrtran(x)$ for almost every $x\in\ab{b,c}$, where $\ddrtran$ is the right derivative of the concave function $\dtran$; as $\ddrtran$ is nonincreasing, that says that $\dtran$ is affine on $\ab{b,c+w}$, and conversely.
\end{proof}

\begin{proof}[Proof of \cref{prop:equality}]
\ref{it:eq:trap}$\Rightarrow$\ref{it:eq:all}: if $b=c$ and $a=e=\Am$ then $s=b=c$ and both sides of \eqref{eq:t} equal $2(\tran(\Am)-\tran(s))$, for every $\tran$.

\ref{it:eq:all}$\Rightarrow$\ref{it:eq:one}: $\tran=x^{3/2}$ lies in $\setcc$, and its derivative $\tfrac32x^{1/2}$ is strictly concave, hence affine on no nondegenerate interval.

\ref{it:eq:one}$\Rightarrow$\ref{it:eq:trap}: let $\tran$ be as in \ref{it:eq:one}.
Both sides of \eqref{eq:t} are unchanged when a constant is added to $\tran$, so we may assume $\tran\in\setcco$.
Let $\eta_0,\eta_1,\nu$ be as in \cref{lem:representation} applied to $f=\dtran$, and let $\Gfun,\Phifun$ be as in \cref{lem:choquet}.
By \eqref{eq:decomposition}, the difference of the two sides of \eqref{eq:t} is $\eta_0\Gfun(0)+\eta_1\Phifun(\infty)+\int\Phifun\dl\nu$, where $\Gfun(0)\leq0$ and $\Phifun(\infty)\leq0$ by \cref{rem:endpoints} and $\Phifun\leq0$ on $\Rpp$ by \cref{lem:tent,lem:reduced}.
All three terms are nonpositive, so equality forces $\int\Phifun\dl\nu=0$, that is, $\Phifun=0$ $\nu$-almost everywhere.
Now $\nu((\sigma,\theta))=f\pr(\sigma)-f\pr(\theta-)$ for $0<\sigma<\theta$, and this vanishes only if the nonincreasing function $f\pr$ is constant on $[\sigma,\theta)$, which would make $\dtran$ affine there; so $\nu$ charges every nondegenerate interval.
As $\Phifun$ is continuous and nonpositive, $\Phifun(\theta_0)<0$ at some $\theta_0$ would give $\Phifun<0$ on a nondegenerate interval of positive $\nu$-measure, so $\Phifun\equiv0$ on $\Rpp$, and then $\Gfun\equiv0$ on $\Rp$, because $\Gfun$ is continuous and $\Phifun(\theta)=\int_0^\theta\Gfun$.

With $m,h,\delta$ as in \cref{lem:tent}, $\Gfun(0)=2m-2\Am=0$ gives $m=\Am$, and then $\Phifun(\infty)=h^2-\delta^2-(\Am^2-m^2)=0$ gives $h=\delta$.
By \eqref{eq:tent-G}, the identity $\Gfun\equiv0$ now reads $\tent mh\equiv\tent sh$.
If $h>0$, two tents of the same height $h$ agree only if their centers agree, so $s=m=\Am$, which contradicts $\Am^2=s^2+uv$ with $uv>0$.
Hence $h=\delta=0$, that is, $a=e=m=\Am$ and $b=c$.
\end{proof}

\subsection{The four-point conditions}\label{app:fourpoint}

\begin{lemma}[Symmetric forms]\label{lem:fourpoint-symmetric}
The three ways of splitting $\cb{y,z,q,p}$ into two pairs single out the three \emph{opposite pairs} $\cb{a,e}$, $\cb{b,c}$, $\cb{u,v}$; put
\begin{equation}\label{eq:Spi}
\begin{aligned}
  \sigma_1 &:= a+e, & \sigma_2 &:= b+c, & \sigma_3 &:= u+v, \\
  S_1 &:= a^2+e^2, & S_2 &:= b^2+c^2, & S_3 &:= u^2+v^2, \\
  \pi_1 &:= ae, & \pi_2 &:= bc, & \pi_3 &:= uv \eqfs
\end{aligned}
\end{equation}
Then, whenever $\cb{i,j,k}=\cb{1,2,3}$,
\begin{align}
  \textnormal{(R)} &\iff \abs{S_i-S_j}\leq2\pi_k \eqcm \label{eq:Rall}\\
  \textnormal{(P)} &\iff \pi_1,\pi_2,\pi_3 \text{ satisfy the triangle
    inequality} \eqcm \label{eq:Pall}\\
  \textnormal{(P}^{*}\textnormal{)} &\iff \abs{\sigma_i^2-\sigma_j^2}\leq4\pi_k
    \eqfs \label{eq:Psall}
\end{align}
\end{lemma}

\begin{proof}
A relabeling permutes $(\sigma_1,\sigma_2,\sigma_3)$, $(S_1,S_2,S_3)$ and $(\pi_1,\pi_2,\pi_3)$ by one and the same permutation of $\cb{1,2,3}$, and the $24$ labelings realize all six permutations, four labelings each.
For the labeling in which $\cb{u,v}$ is the pair of bases, \eqref{eq:r} reads $S_1-S_2\leq2\pi_3$, and interchanging $y$ with $z$ turns it into $S_2-S_1\leq2\pi_3$; likewise \eqref{eq:p} reads $\pi_1\leq\pi_2+\pi_3$, and \eqref{eq:pstar} squared reads $\sigma_1^2\leq\sigma_2^2+4\pi_3$, because $4\Am^2=\sigma_2^2+4\pi_3$.
\end{proof}

\begin{proof}[Proof of \cref{prop:fourpoint}]
\ref{it:fp:H}\ A subset of a metric space is a metric space, and \eqref{eq:R} and \eqref{eq:P} apply to every labeling of the four points.
(There is a converse: a four-point metric space embeds in a \cato{} space if and only if it satisfies the weighted quadruple inequalities of \cite{gromov01} and \cite[Theorem~4.9(iv)]{sturm03}, by \cite[Theorem~1.3]{toyoda20}; we do not use it below.)

\ref{it:fp:RP}\ In each labeling, adding \eqref{eq:r} to twice \eqref{eq:p} gives the square of \eqref{eq:pstar}.
Symmetrically: $\sigma_i^2-\sigma_j^2=(S_i-S_j)+2(\pi_i-\pi_j)$, where \eqref{eq:Rall} bounds the first bracket by $2\pi_k$ in absolute value and \eqref{eq:Pall} bounds the second by $2\pi_k$.

\ref{it:fp:T}\ \cref{thm:algebraic}, applied in each labeling separately.

\ref{it:fp:complete}\ Call a condition \emph{derived} from a set $\mc X$ if \ref{it:fp:H}--\ref{it:fp:T} produce it from $\mc X$.
Reading those three items off, $\textnormal{(H)}$ is derived only from $\textnormal{(H)}$; $\textnormal{(M)}$ and $\textnormal{(P)}$ only from themselves and $\textnormal{(H)}$; $\textnormal{(R)}$ from itself, $\textnormal{(T)}$ and $\textnormal{(H)}$; $\textnormal{(P}^{*}\textnormal{)}$ from itself, $\textnormal{(T)}$, $\textnormal{(H)}$ and $\textnormal{(R)}\wedge\textnormal{(P)}$; and $\textnormal{(T)}$ exactly when both $\textnormal{(R)}$ and $\textnormal{(P}^{*}\textnormal{)}$ are derived.
A row that satisfies $\mc X$ and fails $\textnormal{(Y)}$ also settles every subset of $\mc X$, so only the $\mc X$ maximal with $\textnormal{(Y)}$ not derived have to be treated; since $\textnormal{(H)}$ derives everything and $\textnormal{(T)}$ may be traded for $\textnormal{(R)}\wedge\textnormal{(P}^{*}\textnormal{)}$, these maximal sets lie in $\cb{\textnormal{(M)},\textnormal{(R)},\textnormal{(P)},\textnormal{(P}^{*}\textnormal{)}}$.
There are eight of them, and \cref{tab:fourpoint} supplies a row for each:
\begin{center}
\small
\begin{tabular}{@{}ll@{\qquad}l@{}}
\toprule
$\textnormal{(Y)}$ & maximal $\mc X$ not deriving $\textnormal{(Y)}$ & row \\
\midrule
$\textnormal{(H)}$ & $\textnormal{(M)},\textnormal{(R)},\textnormal{(P)},\textnormal{(P}^{*}\textnormal{)}$ & \ref{it:w:notH} \\
$\textnormal{(M)}$ & $\textnormal{(R)},\textnormal{(P)},\textnormal{(P}^{*}\textnormal{)}$ & \ref{it:w:notM} \\
$\textnormal{(P)}$ & $\textnormal{(M)},\textnormal{(R)},\textnormal{(P}^{*}\textnormal{)}$ & \ref{it:w:notP} \\
$\textnormal{(R)}$ & $\textnormal{(M)},\textnormal{(P)},\textnormal{(P}^{*}\textnormal{)}$ & \ref{it:w:notR} \\
$\textnormal{(P}^{*}\textnormal{)}$ & $\textnormal{(M)},\textnormal{(R)}$ \quad and \quad $\textnormal{(M)},\textnormal{(P)}$ & \ref{it:w:notPs}, \ref{it:w:notRPs} \\
$\textnormal{(T)}$ & $\textnormal{(M)},\textnormal{(R)}$ \quad and \quad $\textnormal{(M)},\textnormal{(P)},\textnormal{(P}^{*}\textnormal{)}$ & \ref{it:w:notPs}, \ref{it:w:notR} \\
\bottomrule
\end{tabular}
\end{center}

It remains to verify the rows.
In each of them $\textnormal{(R)}$, $\textnormal{(P)}$ and $\textnormal{(P}^{*}\textnormal{)}$ are read off \cref{lem:fourpoint-symmetric}, $\textnormal{(T)}$ is $\textnormal{(R)}\wedge\textnormal{(P}^{*}\textnormal{)}$ by \ref{it:fp:T}, and $\textnormal{(M)}$ is the twelve triangle inequalities.
\begin{enumerate}[label=\textup{(\alph*)}]
  \item\label{it:w:simplex} $(1,1,1,1,1,1)$ is the regular simplex of $\R^3$,
    so $\textnormal{(H)}$ holds and with it everything else.
  \item\label{it:w:notH} $(1,2,1,2,2,2)$: put $q,y,p$ on a line at $-1,0,1$ and
    let $z$ be at distance $2$ from each of them; all twelve triangle
    inequalities hold, one of them with equality.
    Here $\sigma=(3,3,4)$, $S=(5,5,8)$ and $\pi=(2,2,4)$, so \eqref{eq:Rall}
    reads $0\leq8$, $3\leq4$, $3\leq4$ and \eqref{eq:Pall} reads $4\leq2+2$:
    both hold, the latter with equality, and
    $\textnormal{(P}^{*}\textnormal{)}$ follows by \ref{it:fp:RP}.
    But if the four points lay in a \cato{} space then $y$ would be a midpoint
    of $q$ and $p$, and the \cato{} inequality would give
    \begin{equation*}
      \ol zy^{2}\leq\tfrac12\ol zq^{2}+\tfrac12\ol zp^{2}
        -\tfrac14\ol pq^{2}=2+2-1=3
      \eqcm
    \end{equation*}
    whereas $\ol zy^{2}=4$.
  \item\label{it:w:notM} $(1,2,1,3,1,3)$: $\sigma=(3,4,4)$, $S=(5,10,10)$ and
    $\pi=(2,3,3)$, so \eqref{eq:Rall} reads $5\leq6$, $5\leq6$, $0\leq4$ and
    $(2,3,3)$ is a triangle; $\textnormal{(P}^{*}\textnormal{)}$ follows by
    \ref{it:fp:RP}.
    But $\ol zq=3>2=\ol yz+\ol yq$.
  \item\label{it:w:notP} $(1,4,2,5,5,3)$: $\sigma=(5,7,8)$, $S=(17,29,34)$ and
    $\pi=(4,10,15)$.
    Now \eqref{eq:Rall} reads $12\leq30$, $17\leq20$, $5\leq8$ and
    \eqref{eq:Psall} reads $24\leq60$, $39\leq40$, $15\leq16$, so
    $\textnormal{(T)}$ holds; but $15>4+10$, so $\textnormal{(P)}$ fails.
    All twelve triangle inequalities hold, one of them with equality.
  \item\label{it:w:notR} $(1,2,2,1,1,3)$: $\sigma=(3,3,4)$, $S=(5,5,10)$ and
    $\pi=(2,2,3)$.
    Here $(2,2,3)$ is a triangle and \eqref{eq:Psall} reads $0\leq12$,
    $7\leq8$, $7\leq8$; but \eqref{eq:Rall} asks for $5\leq4$, so
    $\textnormal{(R)}$ fails and with it $\textnormal{(T)}$.
    All twelve triangle inequalities hold, two of them with equality.
  \item\label{it:w:notPs} $(3,4,2,4,3,1)$: $\sigma=(7,6,4)$, $S=(25,20,10)$ and
    $\pi=(12,8,3)$.
    Now \eqref{eq:Rall} reads $5\leq6$, $15\leq16$, $10\leq24$, so
    $\textnormal{(R)}$ holds; but $12>8+3$, so $\textnormal{(P)}$ fails, and
    \eqref{eq:Psall} asks for $\abs{49-36}=13\leq12$, so
    $\textnormal{(P}^{*}\textnormal{)}$ and with it $\textnormal{(T)}$ fails.
    All twelve triangle inequalities hold, one of them with equality.
  \item\label{it:w:notRPs} $(1,1,1,1,1,2)$ is the member of the family of
    \ref{it:w:notH} with $z$ at distance $1$ from $q,y,p$; here
    $\sigma=(2,2,3)$, $S=(2,2,5)$ and $\pi=(1,1,2)$.
    Then $(1,1,2)$ is a triangle, so $\textnormal{(P)}$ holds; but
    \eqref{eq:Rall} asks for $3\leq2$ and \eqref{eq:Psall} for $5\leq4$, so
    neither $\textnormal{(R)}$ nor $\textnormal{(P}^{*}\textnormal{)}$ holds.
\end{enumerate}
In every row but \ref{it:w:simplex} and \ref{it:w:notH} one of $\textnormal{(M)}$, $\textnormal{(R)}$, $\textnormal{(P)}$ already fails, so $\textnormal{(H)}$ fails by \ref{it:fp:H}.
\end{proof}

\subsection{\texorpdfstring{The witnesses of \cref{prop:necessity}}{The witnesses}}\label{app:witnesses}

\begin{proof}[Proof of \cref{prop:necessity}]
Both sides of \eqref{eq:T} are unchanged when a constant is added to $\tran$, so we may take $\tran\in\setcco$; and by \cref{prop:choquet-structure} such a $\tran$ is $\lambda\mathrm{id}+\mu x^2+\int\psi_\theta\,\nu(\dl\theta)$ with $\lambda,\mu\geq0$ and $\nu$ a nonnegative measure on $\Rpp$, where $\psi_\theta(x)=x^2-\posp{x-\theta}^2$; this is \eqref{eq:choquet-tau} renormalized, with $\lambda=\eta_0$, $\mu=\tfrac12\eta_1$ and $\nu$ halved.
Both sides are linear in $\tran$, so it is enough to evaluate their difference, in the reading \eqref{eq:six},
\begin{equation}\label{eq:necessity-gap}
  \Delta(\tran) := \tran(a)+\tran(e)-\tran(b)-\tran(c)-2\br{\tran(\Am)-\tran(s)}
\end{equation}
on these generators; \eqref{eq:T} fails exactly when $\Delta(\tran)>0$.

\ref{it:nec:R}\ Put $y,z,q$ on a line at $0,1,4$ and let $p$ be a fourth point at distance $3$ from each of them, so that
\begin{equation}\label{eq:witness-P}
  a=\ol yq=4,\ \ e=\ol zp=3,\ \  b=\ol yp=3,\ \  c=\ol zq=3,\ \ 
  u=\ol yz=1,\ \  v=\ol pq=3
  \eqfs
\end{equation}
All twelve triangle inequalities hold, one of them with equality.
The three products of opposite pairs are $\ol yq\,\ol zp=12$, $\ol yp\,\ol zq=9$ and $\ol yz\,\ol pq=3$; they satisfy the triangle inequality, which by \cref{lem:fourpoint-symmetric} is exactly \eqref{eq:P} for every labeling---with equality for the one written---while \eqref{eq:R} fails, $a^2+e^2-b^2-c^2=7>6=2uv$.
Here $s=3$ and $\Am=2\sqrt3$, so $\Delta(\tran)=\tran(4)+\tran(3)-2\tran(2\sqrt3)$.
On the generators,
\begin{equation}\label{eq:witness-P-gaps}
    \begin{gathered}
  \Delta(\mathrm{id}) = 7-4\sqrt3,\qquad \Delta(x^2)=1,
  \\[6pt]
  \Delta(\psi_\theta) =
  \begin{cases}
    \br{14-8\sqrt3}\theta, & 0\leq\theta\leq3,\\
    \theta^2-\br{8\sqrt3-8}\theta+9, & 3\leq\theta\leq2\sqrt3,\\
    -(\theta-3)(\theta-5), & 2\sqrt3\leq\theta\leq4,\\
    1, & \theta\geq4,
  \end{cases}
  \end{gathered}
\end{equation}
and all of these are positive for $\theta>0$: the first because $7^2=49>48=(4\sqrt3)^2$ and $14^2=196>192=(8\sqrt3)^2$; the second because its discriminant $(8\sqrt3-8)^2-36=220-128\sqrt3$ is negative, $220^2=48400<49152=128^2\cdot3$; the third because $3<2\sqrt3<4<5$.
So $\Delta(\tran)>0$ unless $\lambda=\mu=0$ and $\nu=0$, that is, unless $\tran$ is constant.

\ref{it:nec:P}\ Take
\begin{equation}\label{eq:witness-R}
  a=e=13,\qquad b=14,\qquad c=10,\qquad u=7,\qquad v=3
  \eqfs
\end{equation}
All twelve triangle inequalities hold, one of them with equality, so these are the six distances of a four-point metric space.
The three sums of squares over opposite pairs are $338$, $296$ and $58$, and the three products are $169$, $140$ and $21$, so the three differences $42$, $280$, $238$ are at most the three doubled products $42$, $280$, $338$: by \cref{lem:fourpoint-symmetric}, \eqref{eq:R} holds for every labeling, with equality in the first two cases, the first of which is the labeling written, while \eqref{eq:P} fails, $ae=169>161=bc+uv$.
Here $s=12$, $\Am=\sqrt{165}$ and $\diam=14$, so $\Delta(\tran)=2\tran(13)+2\tran(12)-\tran(14)-\tran(10)-2\tran(\sqrt{165})$.
On the generators, $\Delta(\mathrm{id})=26-2\sqrt{165}$ and $\Delta(x^2)=0$, and
\begin{equation}\label{eq:witness-R-gaps}
  \Delta(\psi_\theta) =
  \begin{cases}
    \br{52-4\sqrt{165}}\theta, & 0\leq\theta\leq10,\\
    -\theta^2+\br{72-4\sqrt{165}}\theta-100, & 10\leq\theta\leq12,\\
    \theta^2+\br{24-4\sqrt{165}}\theta+188, & 12\leq\theta\leq\sqrt{165},\\
    -\theta^2+24\theta-142, & \sqrt{165}\leq\theta\leq13,\\
    (\theta-14)^2, & 13\leq\theta\leq14,\\
    0, & \theta\geq14.
  \end{cases}
\end{equation}
Every case is positive on $(0,14)$.
The first because $52^2=2704>2640=(4\sqrt{165})^2$, which also gives $\Delta(\mathrm{id})>0$ since $26^2=676>660$.
The second is concave and positive at both endpoints, its values there being $520-40\sqrt{165}$ and $620-48\sqrt{165}$ with $520^2=270400>264000$ and $620^2=384400>380160$.
The third has negative discriminant, $(24-4\sqrt{165})^2-752=2464-192\sqrt{165}$ with $2464^2=6071296<6082560=192^2\cdot165$.
The fourth has roots $12\pm\sqrt2$, and $12-\sqrt2<\sqrt{165}\leq\theta\leq13<12+\sqrt2$.
So $\Delta(\tran)>0$ unless $\lambda=0$ and $\nu((0,14))=0$, that is, unless $\tran(x)=(\mu+\nu(\Rpp))x^2$ on $\ab{0,14}=\ab{0,\diam}$, and then $\Delta(\tran)=0$.
\end{proof}

\subsection{The converse result}\label{app:proof:converse}

\begin{proof}[Proof of \cref{thm:converse}]
    \emph{The increments are concave.} Let $b,c,v\geq0$ and take the four collinear points $y,p,q,z$ on a line of the Euclidean plane, at coordinates $0$, $b$, $b+v$, $b+v+c$ along that line, so that $a=b+v$, $e=c+v$, $u=b+c+v$ and the sixth distance is $v$.
    Both \eqref{eq:r} and \eqref{eq:p} are equalities here, and $\Am-s=v$, so by \cref{lem:rp-tight} the hypothesis reduces on them to $g_v(b)+g_v(c)\leq2g_v(\tfrac{b+c}{2})$, where
    \begin{equation}\label{eq:converse-g}
        g_v := \tran(\cdot+v)-\tran(\cdot)
        \eqcm
    \end{equation}
    and $b,c,v\geq0$ are arbitrary.
    Thus $g_v$ is midpoint concave on $\Rp$, and it is measurable, being a difference of measurable functions.
    A measurable midpoint-concave function on an open interval is concave, by Sierpi\'nski's theorem \cite{sierpinski20}; so $g_v$ is concave on $\Rpp$, hence continuous there.
    Midpoint concavity at the endpoint $0$ then extends concavity to $\Rp$: iterating it gives $g_v((1-\lambda)x)\geq\lambda g_v(0)+(1-\lambda)g_v(x)$ for dyadic $\lambda\in(0,1)$ and $x>0$, and continuity on $\Rpp$ gives it for all $\lambda\in(0,1)$.
    So $g_v$ is concave on $\Rp$ for every $v\geq0$.
    
    \emph{$\tran$ nondecreasing.} Let $\rho,\eta\geq0$ and take the rectangle
    \begin{equation}\label{eq:rectangle}
        y=(0,0),\qquad z=(\rho,0),\qquad q=(0,\eta),\qquad p=(\rho,\eta)
    \end{equation}
    in the Euclidean plane.
    Its six distances are $\ol yq=\ol zp=\eta$, $\ol yp=\ol zq=\sqrt{\rho^2+\eta^2}$ and $\ol yz=\ol pq=\rho$, so that $s=\sqrt{\rho^2+\eta^2}$ and $\Am=\sqrt{s^2+\rho^2}=\sqrt{\eta^2+2\rho^2}$, and \eqref{eq:T} reads
    \begin{equation}\label{eq:rectangle-conclusion}
        2\tran(\eta)-2\tran\brOf{\sqrt{\rho^2+\eta^2}}
        \leq
        2\tran\brOf{\sqrt{\eta^2+2\rho^2}}-2\tran\brOf{\sqrt{\rho^2+\eta^2}}
        \eqcm
    \end{equation}
    that is, $\tran(\eta)\leq\tran(\sqrt{\eta^2+2\rho^2})$.
    Given $0\leq\eta\leq x$, choosing $\rho=\sqrt{(x^2-\eta^2)/2}$ gives $\tran(\eta)\leq\tran(x)$.
    
    \emph{$\tran$ convex.} Since $\tran$ is nondecreasing, $g_v\geq0$ on $\Rp$ for every $v\geq0$.
    A concave function that is nonnegative on all of $\Rp$ is nondecreasing: if $\phi(x_1)>\phi(x_2)$ for some $x_1<x_2$, concavity would give $\phi(x)\leq\phi(x_2)-\tfrac{\phi(x_1)-\phi(x_2)}{x_2-x_1}(x-x_2)\to-\infty$ as $x\to\infty$.
    Applied to $\phi=g_v$ this says that $\tran(x+v)-\tran(x)$ is nondecreasing in $x$ for every $v\geq0$, that is, $\tran$ is Wright convex.
    In particular, comparing the increments at $x$ and at $x+v$ gives $2\tran(x+v)\leq\tran(x)+\tran(x+2v)$, so $\tran$ is midpoint convex on $\Rp$; being measurable, it is convex on $\Rpp$ by Sierpi\'nski's theorem again, and on $\Rp$ by the same extension argument as above.
    
    \emph{$\tran$ differentiable, $\dtran$ concave.} Being convex, $\tran$ has a right derivative $\dtran_+$ everywhere on $\Rpp$, nondecreasing, and $(g_v)\pr_+=\dtran_+(\cdot+v)-\dtran_+(\cdot)$; concavity of $g_v$ says this is nonincreasing, for every $v>0$.
    Given $x_1<x_2$, take $\mu=\tfrac12(x_1+x_2)$ and $v=\tfrac12(x_2-x_1)$: the value at $x_1$ is $\dtran_+(\mu)-\dtran_+(x_1)$ and the value at $\mu$ is $\dtran_+(x_2)-\dtran_+(\mu)$, so $2\dtran_+(\mu)\geq\dtran_+(x_1)+\dtran_+(x_2)$.
    Thus $\dtran_+$ is midpoint concave on $\Rpp$ and measurable, being monotone; hence it is concave by Sierpi\'nski's theorem, and therefore continuous on $\Rpp$.
    A convex function with continuous right derivative is differentiable, so $\tran$ is differentiable on $\Rpp$ with $\dtran=\dtran_+$ concave.
\end{proof}

\subsection{The reduced problem}\label{app:reduced}

Throughout this subsection hypothesis \eqref{eq:reduced-hyp} is in force and $P,N,D$ are as in \eqref{eq:closed-forms} and \eqref{eq:closed-forms-D}, so that $\Phifun=P-N-D$ and \cref{lem:reduced} is the statement $P\leq N+D$ on $\Rp$.
We write $\xi:=\Am-m\geq0$ and $\Kk:=\Am^2-m^2=\xi(2m+\xi)$, so that the formula for $D$ in \eqref{eq:closed-forms-D} reads $D(\theta)=2\xi\theta$ for $\theta\leq m$ and $D(\theta)=\Kk-\posp{\Am-\theta}^2$ for $\theta\geq m$.
We split on the sign of $h^2-\Kk$; the two cases are \cref{lem:case1,lem:case2}, which together are \cref{lem:reduced}.

\begin{lemma}[The case $h^2\leq \Kk$; \leanchecked]\label{lem:case1}
If $h^2\leq \Kk$ then $P(\theta)\leq D(\theta)$ for all $\theta\geq0$.
In particular \cref{lem:reduced} holds, since $N\geq0$.
\end{lemma}

\begin{proof}
\emph{(a) $h^2\leq 4m\xi$.} If $m=0$ then $h=0$ and both sides vanish.
Let $m>0$.
If $\xi\leq2m$ then, since $\Kk=\xi(2m+\xi)$, $h^2\leq\Kk=2m\xi+\xi^2\leq2m\xi+2m\xi=4m\xi$.
If $\xi>2m$ then $4m\xi>8m^2\geq m^2\geq h^2$, using $h\leq m$.

\emph{(b) The range $\theta\leq m$.} By \eqref{eq:closed-forms} and \eqref{eq:closed-forms-D}, $P(\theta)-D(\theta)=\tfrac12\posp{\theta-m+h}^2-2\xi\theta$, which is convex on $[0,m]$, vanishes at $\theta=0$, and equals $\tfrac12h^2-2m\xi\leq0$ at $\theta=m$ by (a).
A convex function that is nonpositive at both endpoints of an interval is nonpositive on it.

\emph{(c) The range $\theta\geq m$.} By \eqref{eq:closed-forms} and \eqref{eq:closed-forms-D},
\begin{equation*}
  P(\theta)-D(\theta) = \br{h^2-\Kk} - \tfrac12\posp{m+h-\theta}^2 + \posp{\Am-\theta}^2
  \eqcm
\end{equation*}
so, $h^2\leq\Kk$ being assumed, it suffices to prove
\begin{equation}\label{eq:phi-bound}
  \Upsilon(\theta) := \posp{\Am-\theta}^2-\tfrac12\posp{m+h-\theta}^2 \leq \Kk-h^2
  \qquad(\theta\geq m)
  \eqfs
\end{equation}
At $\theta=m$ we have $\Upsilon(m)=\xi^2-\tfrac12h^2$, and $\Upsilon(m)\leq\Kk-h^2 \iff \tfrac12 h^2\leq\Kk-\xi^2=2m\xi$, which is (a).

If $m+h\leq\Am$: on $[m,m+h]$, $\Upsilon\pr(\theta)=\theta+m+h-2\Am\leq2(m+h)-2\Am\leq0$; on $[m+h,\Am]$, $\Upsilon\pr(\theta)=-2(\Am-\theta)\leq0$; and on $[\Am,\infty)$, $\Upsilon\leq0\leq\Kk-h^2$.
Hence $\Upsilon\leq\Upsilon(m)$ throughout and \eqref{eq:phi-bound} follows.

If $m+h>\Am$: on $[m,\Am]$ we have $\Upsilon\prr=2-1=1>0$, so $\Upsilon$ is convex there and bounded by its endpoint values $\Upsilon(m)\leq\Kk-h^2$ and $\Upsilon(\Am)=-\tfrac12(m+h-\Am)^2\leq0\leq\Kk-h^2$; and on $[\Am,\infty)$, $\Upsilon\leq0$.
\end{proof}

\begin{lemma}[The case $h^2>\Kk$; \leanchecked]\label{lem:case2}
If $h^2>\Kk$ then $P(\theta)\leq N(\theta)+D(\theta)$ for all $\theta\geq0$.
\end{lemma}

\begin{proof}
Put $\dz:=\sqrt{h^2-\Kk}$, so that $0\leq\dz\leq h$ and, by the hypothesis $h^2-\delta^2\leq\Kk$, $\dz^2 = h^2-\Kk\leq\delta^2$, that is $\dz\leq\delta$.

\emph{(d) Two free consequences.} The case is vacuous when $m=0$, for then $h=0$ and $h^2=0\leq\Kk$.
For $m>0$, $\Kk<h^2\leq m^2$ gives $\Am^2<2m^2$, so
\begin{equation}\label{eq:xi-lt-m}
  \Am<\sqrt2m, \qquad\text{in particular}\qquad \xi = \Am-m<m
  \eqfs
\end{equation}

\emph{(e) Worst case.} The tent $\tent s\delta$ is nondecreasing in $\delta$, both arguments of the minimum in \eqref{eq:tent} being so; and for $\delta\leq s$ the primitive $N$ is nonincreasing in $s$, because then the left foot $s-\delta$ is nonnegative, so $N(\theta)=\widehat N(\theta-s)$ for a fixed nondecreasing $\widehat N$.
(The translation argument needs $\delta\leq s$, and so does the integral of \eqref{eq:tent}: for $\delta=1$ and $\theta=\tfrac12$ one gets $\int_0^\theta\tent s\delta=\tfrac38$ at $s=0$ but $\tfrac7{16}$ at $s=\tfrac14$. The closed form \eqref{eq:closed-forms}, which is the definition in force here and in the Lean development, is nonincreasing in $s$ for every $\delta$: its $s$-derivative is $-\posp{\theta-s+\delta}$ for $\theta\leq s$ and $-\posp{s+\delta-\theta}$ for $\theta\geq s$.) Since $\delta$ and $s$ occur nowhere else in $\Phifun=P-N-D$, and the constraints permit any $\delta\geq\dz$ and any $s\leq\Am$, it suffices to treat
\begin{equation*}
  \delta=\dz, \qquad s=\Am
  \eqcm
\end{equation*}
which is admissible since $\dz\leq h\leq m\leq\Am$, and along the path $(s,\dz)\to(\Am,\dz)$ the constraint $\delta\leq s$ is preserved.
With this choice $\Phifun(\infty)=h^2-\dz^2-\Kk=0$.

\emph{(f) The key inequality.}
\begin{equation}\label{eq:key}
  h-\dz \geq \xi \eqfs
\end{equation}
Indeed $h>0$, since $h^2>\Kk\geq0$; and $h^2-\dz^2=\Kk=\xi(2m+\xi)$ with $h+\dz\leq2m$ (as $\dz\leq h\leq m$), so
\begin{equation*}
  h-\dz = \frac{\xi(2m+\xi)}{h+\dz} \geq \frac{\xi(2m+\xi)}{2m} \geq \xi
  \eqfs
\end{equation*}
In particular $\Am = m+\xi \leq m+h-\dz$.

\emph{(g) The range $\theta\leq m$.} Put $\theta_1:=m-h$ and $\theta_2:=\Am-\dz$.
Here $N(\theta)=\tfrac12\posp{\theta-\theta_2}^2$, so
\begin{equation*}
  \Phifun(\theta) = \tfrac12\posp{\theta-\theta_1}^2-\tfrac12\posp{\theta-\theta_2}^2-2\xi\theta
  \eqfs
\end{equation*}
By \eqref{eq:key}, $\theta_2-\theta_1=\xi+h-\dz\geq2\xi\geq0$, so $\theta_1\leq\theta_2$ and $\Phifun\prr=\ind_{\theta>\theta_1}-\ind_{\theta>\theta_2}\geq0$: $\Phifun$ is convex on $[0,m]$.
Moreover $\Phifun(0)=0$ and
\begin{equation}\label{eq:Phi-at-m}
  \Phifun(m) = \tfrac12h^2-\tfrac12\posp{\dz-\xi}^2-2m\xi \leq 0
  \eqfs
\end{equation}
To see \eqref{eq:Phi-at-m}, substitute $h^2=\dz^2+\xi(2m+\xi)$.
If $\dz\geq \xi$ this gives the identity $\Phifun(m)=\xi(\dz-m)$, which is nonpositive because $\dz\leq h\leq m$ and $\xi\geq0$---no division is needed, so $\xi=0$ is covered.
If $\dz<\xi$, then \eqref{eq:Phi-at-m} reads $\dz^2+\xi^2\leq2m\xi$, and indeed $\dz^2+\xi^2<2\xi^2\leq2m\xi$ by \eqref{eq:xi-lt-m}.
Convexity now gives $\Phifun(\theta)\leq\max(\Phifun(0),\Phifun(m))=0$ on $[0,m]$.

\emph{(h) The range $\theta\geq m$.} Since $\Phifun(\infty)=0$, the claim $\Phifun(\theta)\leq0$ is equivalent to $\int_\theta^\infty\Gfun\geq0$.
Writing $\bar P=h^2-P$, $\bar N=\dz^2-N$ and $\bar D=\Kk-D$, this is $F(\theta)\geq0$ where, for $m\leq \theta\leq\Am$,
\begin{equation*}
  F(\theta)=\tfrac12\posp{m+h-\theta}^2+\tfrac12\posp{\theta-\Am+\dz}^2-\dz^2-\posp{\Am-\theta}^2
  \eqfs
\end{equation*}
Beyond $\Am$ one has instead $\bar N(\theta)=\tfrac12\posp{\Am+\dz-\theta}^2$ and $\bar D=0$, so there $F(\theta)=\tfrac12\posp{m+h-\theta}^2-\tfrac12\posp{\Am+\dz-\theta}^2\geq0$, because $m+h\geq\Am+\dz$ by \eqref{eq:key}.

On $[m,\Am]$ we have $\theta\leq\Am\leq m+h$, so $F\prr=\ind_{\theta<m+h}+\ind_{\theta>\Am-\dz}-2=\ind_{\theta>\Am-\dz}-1 \in\cb{-1,0}$: $F$ is concave on $[m,\Am-\dz]$ and affine on $[\Am-\dz,\Am]$.
Such a function is bounded below by the minimum of its values at $m$, at $\Am-\dz$ when that point lies in $[m,\Am]$, and at $\Am$.
Now
\begin{align*}
  F(m) &= \tfrac12h^2+\tfrac12\posp{\dz-\xi}^2-\dz^2-\xi^2 = -\Phifun(m)\geq0
    &&\text{by \eqref{eq:Phi-at-m}},\\
  F(\Am-\dz) &= \tfrac12\br{h-\xi+\dz}^2-2\dz^2\geq0
    &&\iff h-\xi\geq\dz,\\
  F(\Am) &= \tfrac12\br{h-\xi}^2-\tfrac12\dz^2\geq0
    &&\iff h-\xi\geq\dz,
\end{align*}
and $h-\xi\geq\dz$ is \eqref{eq:key}.
(The last two lines drop a $\posp{\cdot}$, legitimately, using $h-\xi\geq\dz\geq0$; and when $\Am-\dz<m$ the middle point lies outside $[m,\Am]$, $F$ is affine on all of $[m,\Am]$, and the two endpoints suffice.)
\end{proof}

\subsection{Proofs for the consequences}\label{app:geometry}

\begin{lemma}[The diameter is dominated; triangle inequality only]\label{lem:diam-bound}
For any four points of a metric space,
\begin{equation}\label{eq:diam-dominated}
\begin{gathered}
  \tfrac12\br{\ol yq+\ol zp+\ol yz+\ol pq} \geq \diam
  \eqcm\\
  \tfrac12\br{\ol yp+\ol zq+\ol yz+\ol pq} \geq \diam
  \eqfs
\end{gathered}
\end{equation}
Consequently $\sqrt{\smin^2+\ol yz\,\ol pq}+\smin\geq\diam$.
\end{lemma}

\begin{proof}
Use the abbreviations \eqref{eq:six}.
For the second inequality in \eqref{eq:diam-dominated}, the triangle inequality along the eight paths
\begin{equation*}
  \begin{array}{ll@{\qquad}ll}
    a\leq b+v, & a\leq u+c, & e\leq c+v, & e\leq u+b,\\
    b\leq u+c+v, & c\leq u+b+v, & u\leq b+v+c, & v\leq b+u+c
  \end{array}
\end{equation*}
gives $2x\leq b+c+u+v$ for each of the six distances $x$---for $a$ and $e$ by adding the two bounds, for the others directly.
For the first inequality, apply this to the relabeled quadruple $(y,z;p,q)$.
Relabeling permutes the same four points, so the diameter is unchanged, while the six distances become $(b,c,a,e,u,v)$: the pairs $\cb{a,e}$ and $\cb{b,c}$ are exchanged and $u,v$ are fixed.

For the consequence: by \cref{lem:cross-dominates}, $\smin\geq\tfrac12\abs{u-v}$, so $\sqrt{\smin^2+uv}\geq\tfrac12(u+v)$ and therefore
\begin{align*}
  \sqrt{\smin^2+uv}+\smin
  &\geq \tfrac12(u+v)+\min\brOf{\tfrac{a+e}{2},\tfrac{b+c}{2}}
  \\&= \min\brOf{\tfrac{a+e+u+v}{2},\tfrac{b+c+u+v}{2}}
  \\&\geq \diam
  \eqfs
\end{align*}
\end{proof}

\begin{proof}[Proof of \cref{cor:twosided,cor:twosided-catk,rem:bound-order}]
We give the \cato{} case; \cref{cor:twosided-catk} is the same argument with \eqref{eq:TC} in place of \eqref{eq:T}, and we point out at the end where the constant enters and where it does not.
Write $u=\ol yz$, $v=\ol pq$ and $B(x):=2\tran(\sqrt{x^2+uv})-2\tran(x)$, which is nonincreasing by \cref{lem:gap-monotone}, so that \eqref{eq:ts:trapezoid} is $B(\smin)$ and, by \cref{lem:cross-dominates} and $\tfrac14(u-v)^2+uv=\tfrac14(u+v)^2$, \eqref{eq:ts:bases} is $B(\tfrac12\abs{u-v})\geq B(\smin)$.

\eqref{eq:T} applied to $y,z,q,p$ gives $\Quadt\leq B(\sleg)$ with $\sleg:=\tfrac12(\ol yp+\ol zq)$, and applied to $z,y,q,p$, whose six distances are $(c,b,e,a,u,v)$ in the notation \eqref{eq:six}, it gives $-\Quadt\leq B(\sdia)$ with $\sdia:=\tfrac12(\ol yq+\ol zp)$, by \cref{lem:symmetry}.
As $B$ is nonincreasing, $\abs{\Quadt}\leq\max(B(\sleg),B(\sdia))=B(\smin)$, which is \eqref{eq:ts:trapezoid}; and $\eqref{eq:ts:trapezoid}\leq\eqref{eq:ts:bases}$ was just noted.

For \eqref{eq:ts:diam}, write $t_{\ms{min}}:=\sqrt{\smin^2+uv}$ and $M:=\tfrac12(t_{\ms{min}}+\smin)$.
The upper bound in \eqref{eq:hh} gives $\tran(t_{\ms{min}})-\tran(\smin)\leq(t_{\ms{min}}-\smin)\dtran(M)$, and $t_{\ms{min}}-\smin=(t_{\ms{min}}^2-\smin^2)/(t_{\ms{min}}+\smin)=uv/(2M)$, so $B(\smin)\leq uv\dtran(M)/M$.
By \cref{lem:diam-bound}, $2M\geq\diam$, and $x\mapsto\dtran(x)/x$ is nonincreasing (\cref{lem:ratio-monotone}), so $\dtran(M)/M\leq2\dtran(\tfrac12\diam)/\diam$: this is \eqref{eq:ts:diam} together with $\eqref{eq:ts:trapezoid}\leq\eqref{eq:ts:diam}$.

Applying \eqref{eq:ccdiff-abs} with $x_1=\tfrac u2$ and $x_2=\tfrac v2$ turns \eqref{eq:ts:bases} into \eqref{eq:ts:hhminmax}, and $\max(u,v)\leq\diam$ with \cref{lem:ratio-monotone} turns \eqref{eq:ts:diam} into \eqref{eq:ts:hhminmax} as well.
The same monotonicity gives $v\dtran(\tfrac u2)\leq u\dtran(\tfrac v2)$ for $v\leq u$, so \eqref{eq:ts:hhminmax} is the smaller of the two bounds in \eqref{eq:ts:hh}, and likewise for \eqref{eq:ts:Wminmax} in \eqref{eq:ts:W}.
By the definition \eqref{eq:Wtau} of $\ctranobest{\tran}$ we have $2\dtran(\tfrac12 x)\leq\ctranobest{\tran}\dtran(x)$ for every $x>0$, which turns \eqref{eq:ts:hh} into \eqref{eq:ts:W} and \eqref{eq:ts:hhminmax} into \eqref{eq:ts:Wminmax}.
Finally $\eqref{eq:ts:bases}\leq\eqref{eq:ts:sqrt}\leq\eqref{eq:ts:sum}$ is \cref{lem:KJ-chain}.

That no further relation holds is seen on three configurations, with $\tran=x^{\alpha}$.
On the four collinear points $y,p,q,z$ of $\R$ at $-\tfrac u2,-\tfrac v2,\tfrac v2,\tfrac u2$, with $v\leq u$, one has $\smin=\tfrac12(u-v)$ and $\diam=u$, so \eqref{eq:ts:diam}, \eqref{eq:ts:hhminmax} and \eqref{eq:ts:Wminmax} all equal $2v\dtran(\tfrac u2)$.
For $u=v=1$ and $1<\alpha<2$ that is $\alpha2^{2-\alpha}$, strictly larger than $2=\eqref{eq:ts:bases}=\eqref{eq:ts:sqrt}=\eqref{eq:ts:sum}$; for $\alpha=1$ and $v<u$ it is $2v<2\sqrt{uv}<u+v$, so it is strictly smaller than \eqref{eq:ts:sqrt} and \eqref{eq:ts:sum}.
On two parallel bases of equal length $u$ at distance $H$, with $\alpha=1$, \eqref{eq:ts:diam} is $2u^2/\sqrt{u^2+H^2}\to0$ as $H\to\infty$ while \eqref{eq:ts:bases} stays $2u$.
Together these separate the seven pairs left open by \eqref{eq:bound-order}.

For \cref{cor:twosided-catk}, only the first paragraph after the definition of $B$ uses the trapezoid comparison, and it uses it exactly twice: for the labelings $(y,z;q,p)$ and $(z,y;q,p)$.
These have the same pair of bases, namely $\cb{y,z}$ and $\cb{p,q}$, hence the same two geodesic segments and the same $r$.
There \eqref{eq:TC} gives $\Quadt\leq2\tran\brOf{\sqrt{\sleg^2+C_R\,uv}}-2\tran(\sleg)\leq C_R\,B(\sleg)$, the second step by \eqref{eq:gap-scaling} with $x=\sleg$, $\lambda=uv$ and $C=C_R\geq1$; likewise $-\Quadt\leq C_R\,B(\sdia)$, and so $\abs{\Quadt}\leq C_R\,B(\smin)$.
Stopping before the second step gives \eqref{eq:ts:trapezoid-catk}: writing $B_C(x):=2\tran\brOf{\sqrt{x^2+C\,uv}}-2\tran(x)$, which is nonincreasing in $x$ by \cref{lem:gap-monotone}, the same two applications of \eqref{eq:TC} give $\Quadt\leq B_{C_R}(\sleg)$ and $-\Quadt\leq B_{C_R}(\sdia)$, hence $\abs{\Quadt}\leq B_{C_R}(\smin)$.
Everything after that paragraph is a chain of inequalities between the right-hand sides alone---\cref{lem:diam-bound,lem:ratio-monotone,lem:ccdiff,lem:KJ-chain} and \eqref{eq:Wtau}---which uses no curvature bound at all, only the triangle inequality among the four points and properties of $\setcc$.
Multiplying that chain by $C_R\geq1$ preserves it, which is \cref{cor:twosided-catk}.
In particular the diameter appearing in \eqref{eq:ts:diam} remains $\diam\cb{y,z,q,p}$ of \eqref{eq:diam}: it enters only through \cref{lem:diam-bound}, whose conclusion $2M\geq\diam$ need not survive the replacement of $\diam\cb{y,z,q,p}$ by the larger $\diam\br{\seg yz\cup\seg pq}$.
\end{proof}

\begin{proof}[Proof of \cref{prop:Wtau}]
Write $u=\ol yz$ and $v=\ol pq$, and call a constant $L$ \emph{admissible} if $\QuadOf yzqp\leq L\,\ol pq\,\dtran\brOf{\ol yz}$ holds for every quadruple of every \cato{} space.
For $y=p$ and $z=q$ one has $\ol yp=\ol zq=0$ and $\ol yq=\ol zp=u=v$, so $\abs{\Quadt}$ and the right-hand sides of \eqref{eq:ts:sqrt} and \eqref{eq:ts:sum} all equal $2\tran(u)$.

\emph{Upper bound.} By \eqref{eq:ts:Wminmax},
\begin{equation*}
  \Quadt\leq\ctranobest{\tran}\min(u,v)\dtran(\max(u,v))\leq\ctranobest{\tran}\,v\,\dtran(u)
  \eqcm
\end{equation*}
where the second inequality is an equality for $v\leq u$, and for $v>u$ it reads $u\dtran(v)\leq v\dtran(u)$, which holds because $x\mapsto\dtran(x)/x$ is nonincreasing (\cref{lem:ratio-monotone}).
So $\ctranobest{\tran}$ is admissible.

\emph{Lower bound.} On the collinear quadruple of the previous proof, with $v\leq u$, $s_0:=\tfrac12(u-v)$ and $t_0:=\tfrac12(u+v)$, one has $\Quadt=2(\tran(t_0)-\tran(s_0))=2\int_{(u-v)/2}^{(u+v)/2}\dtran$, so
\begin{equation*}
  \frac{\Quadt}{v\dtran(u)}
  = \frac{2}{v\dtran(u)}\int_{(u-v)/2}^{(u+v)/2}\dtran
  \xrightarrow[v\searrow0]{}
  \frac{2\dtran(u/2)}{\dtran(u)}
  \eqcm
\end{equation*}
since $\dtran$ is concave, hence continuous, on $\Rpp$.
Taking the supremum over $u>0$ shows that no $L<\ctranobest{\tran}$ is admissible.

The two suprema in \eqref{eq:Wtau} agree because $\dtran(x_1)+\dtran(x_2)\leq2\dtran(\tfrac{x_1+x_2}2)$ (\cref{lem:dtran-subadd}), so the second supremum is attained when $x_1=x_2$.
Finally $\ctranobest{\tran}\geq1$ by the first inequality of \cref{lem:dtran-subadd} and $\ctranobest{\tran}\leq2$ because $\dtran$ is nondecreasing.
\end{proof}

\section{\texorpdfstring{The class $\setcco$}{The class S0}}\label{app:class}

\subsection{Members}\label{ssec:app:members}

\Cref{tab:members} lists some members of $\setcco$.

\begin{table}[ht]
\centering
\small
\begin{tabular}{@{}lllcl@{}}
\toprule
$\tran(x)$ & $\dtran(x)$ & $\ddtran(x)$ & $\ctranobest{\tran}$ & name/reference\\
\midrule
$x^{\alpha}$, $\alpha\in[1,2]$
  & $\alpha x^{\alpha-1}$
  & $\alpha(\alpha-1)x^{\alpha-2}$
  & $2^{2-\alpha}$
  & \\
  $\psi_\theta$
  & $2\min(x,\theta)$
  & $2\ind_{x<\theta}$
  & $2$
  & Huber \cite{huber64} \\

$\theta^2(\sqrt{1+x^2/\theta^2}-1)$
  & $x/\sqrt{1+x^2/\theta^2}$
  & $(1+x^2/\theta^2)^{-3/2}$
  & $2$
  & pseudo-Huber \cite{charbonnier94} \\
$\log\cosh x$
  & $\tanh x$
  & $\operatorname{sech}^2 x$
  & $2$
  & \cite{green90} \\
$x\log(1+x)$
  & $\log(1+x)+\tfrac{x}{1+x}$
  & $\tfrac{1}{1+x}+\tfrac{1}{(1+x)^2}$
  & $2$
  & \\
$\int_0^x\log(1+\sigma+\tfrac12\sigma^2)\dl\sigma$
  & $\log(1+x+\tfrac12x^2)$
  & $\tfrac{1+x}{1+x+x^2/2}$
  & $2$
  & Catoni \cite{catoni12} \\
$cx-c^2\log(1+x/c)$
  & $cx/(c+x)$
  & $c^2/(c+x)^2$
  & $2$
  & Fair \cite[Section 6.4.5]{Rey1983} \\

\bottomrule
\end{tabular}
\caption{Some members of $\setcco$, for $\alpha\in[1,2]$ and $\theta,c\in\Rpp$; here $\psi_\theta(x)=x^2-\posp{x-\theta}^2$.
The column $\ctranobest{\tran}$ gives the sharp constant \eqref{eq:Wtau} of \cref{prop:Wtau}; the supremum defining it is attained only for $x^{\alpha}$ and $\psi_\theta$.}
\label{tab:members}
\end{table}

\subsection{Elementary properties}\label{ssec:app:properties}

Throughout this subsection $\tran\in\setcco$.
\begin{lemma}[{\cite[Lemma~3]{quadruple}}]\label{lem:dtran-subadd}
For $x_1,x_2\in\Rp$,
\begin{equation*}
  \dtran(x_1+x_2)\leq\dtran(x_1)+\dtran(x_2)\leq2\dtran\brOf{\frac{x_1+x_2}{2}}
  \eqfs
\end{equation*}
\end{lemma}

\begin{lemma}[{\cite[Lemma~3]{quadruple}}]\label{lem:dtran-factor}
For $x,a\in\Rp$,
\begin{equation*}
  \dtran(ax)\geq a\dtran(x)\text{ if } a\leq1,
  \qquad
  \dtran(ax)\leq a\dtran(x)\text{ if } a\geq1
  \eqfs
\end{equation*}
\end{lemma}

\begin{lemma}[{\cite[Lemma~2]{quadruple}}; the Hermite--Hadamard inequality for the concave function $\dtran$]\label{lem:ccdiff}
For $x_1,x_2\in\Rp$,
\begin{equation}\label{eq:hh}
  \abs{x_1-x_2}\frac{\dtran(x_1)+\dtran(x_2)}{2}
  \leq\abs{\tran(x_1)-\tran(x_2)}
  \leq\abs{x_1-x_2}\dtran\brOf{\frac{x_1+x_2}{2}}
  \eqfs
\end{equation}
In particular, for $x_1,x_2\geq0$,
\begin{equation}\label{eq:ccdiff-abs}
  \tran(x_1+x_2)-\tran\brOf{\abs{x_1-x_2}}
  \leq2\min(x_1,x_2)\dtran\brOf{\max(x_1,x_2)}
  \eqfs
\end{equation}
\end{lemma}

\begin{lemma}\label{lem:ratio-monotone}
The function $x\mapsto\dtran(x)/x$ is nonincreasing on $\Rpp$.
\end{lemma}

\begin{proof}
This is \cref{lem:dtran-factor} with $a=x_1/x_2\leq1$ and $x=x_2$.
\end{proof}

\begin{lemma}\label{lem:gap-monotone}
Let $\lambda\geq0$.
Then
\begin{equation}\label{eq:gap-monotone}
  x \longmapsto 2\br{\tran\brOf{\sqrt{x^2+\lambda}}-\tran(x)}
\end{equation}
is nonincreasing on $\Rp$.
\end{lemma}

\begin{proof}
For $\lambda=0$ the map vanishes identically, so assume $\lambda>0$ and write $W(x)=\sqrt{x^2+\lambda}>x$.
Since $\dtran$ is nondecreasing and hence bounded on compacts, the map is locally Lipschitz on $\Rp$, so it suffices that its derivative be nonpositive on $\Rpp$, where it exists by \cref{def:class}.
That derivative is
\begin{equation*}
  2\br{\dtran\brOf{W}\tfrac{x}{W}-\dtran(x)}
  = 2x\br{\frac{\dtran(W)}{W}-\frac{\dtran(x)}{x}}
  \leq 0
  \eqcm
\end{equation*}
because $W\geq x$ and $x\mapsto\dtran(x)/x$ is nonincreasing by \cref{lem:ratio-monotone}.
\end{proof}

\begin{lemma}\label{lem:gap-scaling}
Let $\tran\in\setcc$, $\lambda\geq0$, and $C\geq1$.
Then, for every $x\geq0$,
\begin{equation}\label{eq:gap-scaling}
  \tran\brOf{\sqrt{x^2+C\lambda}}-\tran(x)
  \ \leq\ C\br{\tran\brOf{\sqrt{x^2+\lambda}}-\tran(x)}
  \eqfs
\end{equation}
\end{lemma}

\begin{proof}
Put $\phi(\lambda):=\tran\brOf{\sqrt{x^2+\lambda}}-\tran(x)$, so that $\phi(0)=0$.
Being convex and nondecreasing on $\Rp$, $\tran$ is continuous there, hence so is $\phi$; and for $\lambda>0$,
\begin{equation*}
  \phi\pr(\lambda) = \frac{\dtran\brOf{\sqrt{x^2+\lambda}}}{2\sqrt{x^2+\lambda}}
  \eqcm
\end{equation*}
which is nonincreasing in $\lambda$, because $y\mapsto\dtran(y)/y$ is nonincreasing (\cref{lem:ratio-monotone}) and $y=\sqrt{x^2+\lambda}$ increases with $\lambda$.
So $\phi$ is concave on $\Rp$ with $\phi(0)=0$, hence $\lambda\mapsto\phi(\lambda)/\lambda$ is nonincreasing on $\Rpp$ and $\phi(C\lambda)\leq C\phi(\lambda)$ for $C\geq1$; at $\lambda=0$ both sides vanish.
\end{proof}

\begin{lemma}\label{lem:KJ-chain}
For all $u,v\geq0$,
\begin{equation}\label{eq:KJ-chain}
  2\br{\tran\brOf{\tfrac{u+v}{2}}-\tran\brOf{\tfrac{\abs{u-v}}{2}}}
  \leq 2\tran\brOf{\sqrt{uv}}
  \leq \tran(u)+\tran(v)
  \eqfs
\end{equation}
\end{lemma}

\begin{proof}
Since $\tfrac14(u-v)^2+uv=\tfrac14(u+v)^2$, the left-hand side of \eqref{eq:KJ-chain} is the map \eqref{eq:gap-monotone} with $\lambda=uv$ evaluated at $\tfrac12\abs{u-v}$, and $2\tran(\sqrt{uv})$ is its value at $0$, because $\tran(0)=0$.
So the first inequality is \cref{lem:gap-monotone}.
For the second, $\sqrt{uv}\leq\tfrac12(u+v)$ and $\tran$ is nondecreasing and convex, so $2\tran(\sqrt{uv})\leq2\tran(\tfrac{u+v}2)\leq\tran(u)+\tran(v)$.
\end{proof}

\subsection{Structure}\label{ssec:app:stability}

Up to the linear term $\eta_0x$, the following proposition is \cite[Proposition~1.2]{Pinelis2015} and the remark following it, formulated there for the even extensions of the functions in $\setcco$ with $\dtran(0)=0$.

\begin{proposition}[Choquet structure; cf.\ {\cite[Proposition~1.2]{Pinelis2015}}]\label{prop:choquet-structure}
The extreme rays of the cone $\setcco$ are $\R_+\mathrm{id}$, $\R_+x^2$ and $\R_+\psi_\theta$ for $\theta\in\Rpp$, where $\psi_\theta(x)=x^2-\posp{x-\theta}^2$ is the Huber loss with threshold $\theta$.
The ray $\R_+\mathrm{id}$ must be listed separately: it is extreme, since $\mathrm{id}=\tran_1+\tran_2$ with $\tran_i\in\setcco$ forces $\dtran_1+\dtran_2\equiv1$ with both summands nondecreasing, hence both constant; and it does not lie in the cone generated by the others, whose derivatives all tend to $0$ at $0+$, although it lies in its closure, as $\psi_\theta/(2\theta)\to\mathrm{id}$ uniformly on $\Rp$ for $\theta\searrow0$ \cite{Pinelis2015}.
Equivalently, by \cref{lem:representation}, every $\tran\in\setcco$ is
\begin{equation}\label{eq:choquet-tau}
  \tran(x) = \eta_0x + \tfrac{\eta_1}{2}x^2
    + \tfrac12\int_{\Rpp}\psi_\theta(x)\nu(\dl\theta)
\end{equation}
for some $\eta_0,\eta_1\geq0$ and a nonnegative measure $\nu$ on $\Rpp$.
\end{proposition}

\begin{proof}
Apply \cref{lem:representation} to $f=\dtran$ as in the proof of \cref{lem:choquet}, and integrate from $0$ to $x$.
Since $\tran(0)=0$ we get $\tran(x)=\int_0^x\dtran$, and term by term $\int_0^x \eta_0\dl\sigma = \eta_0x$, $\int_0^x \eta_1\sigma\dl\sigma =\tfrac{\eta_1}{2}x^2$, and, by Tonelli,
\begin{equation*}
  \int_0^x\!\int_{\Rpp}(\sigma\wedge\theta)\nu(\dl\theta)\dl\sigma
  = \int_{\Rpp}\!\int_0^x(\sigma\wedge\theta)\dl\sigma\nu(\dl\theta)
  = \tfrac12\int_{\Rpp}\psi_\theta(x)\nu(\dl\theta)
  \eqcm
\end{equation*}
because $\int_0^x(\sigma\wedge\theta)\dl\sigma =\tfrac12x^2-\tfrac12\posp{x-\theta}^2=\tfrac12\psi_\theta(x)$.
This is \eqref{eq:choquet-tau}.
Conversely each of $\mathrm{id}$, $x^2$ and $\psi_\theta$ lies in $\setcco$, so \eqref{eq:choquet-tau} exhibits $\setcco$ as the closed convex cone they generate; that they are the extreme rays is the corresponding statement for the cone of nonnegative nondecreasing concave functions, whose extreme rays are the constants, the identity and the ramps $\sigma\mapsto\sigma\wedge\theta$, together with the argument given above for $\R_+\mathrm{id}$.
\end{proof}

\section{Formal verification}\label{app:lean}
\Cref{lem:reduced}, with the case analysis of \cref{lem:case1,lem:case2}, has been formalized in Lean~4 as the theorem \texttt{TrapezoidComparison.reduced} of the supplementary file \texttt{TrapezoidComparison.lean}; nothing else in the paper has been.
The statement of \texttt{reduced} is
\begin{quote}\small\ttfamily
    theorem reduced \{m h s $\delta$ t $\theta$ : $\mathbb{R}$\}\\
    \hspace*{1em}(hh : 0 $\le$ h) (hhm : h $\le$ m) (hmt : m $\le$ t)\\
    \hspace*{1em}(h$\delta$ : 0 $\le$ $\delta$) (h$\delta$s : $\delta$ $\le$ s) (hst : s $\le$ t)\\
    \hspace*{1em}(hcon : h\textasciicircum{}2 - $\delta$\textasciicircum{}2 $\le$ t\textasciicircum{}2 - m\textasciicircum{}2)\\
    \hspace*{1em}(ht0 : 0 $\le$ $\theta$) :\\
    \hspace*{1em}tentPrim m h $\theta$ $\le$ tentPrim s $\delta$ $\theta$ + dropPrim m t $\theta$
\end{quote}
with \texttt{tentPrim} the closed form of $P$ and of $N$, and \texttt{dropPrim} that of $D$; the hypotheses are \eqref{eq:reduced-hyp}, the comparison diagonal $\Am$ being written \texttt{t}.
The file contains no \texttt{sorry}, and \texttt{\#print axioms} reports only \texttt{propext}, \texttt{Classical.choice} and \texttt{Quot.sound}; checked against Lean \texttt{v4.33.0} and Mathlib tag \texttt{v4.33.0}, that is commit \texttt{db584cd6d46c92f2\allowbreak 09a44c0f1c829460\allowbreak d327499d}.
Its sole dependency is \texttt{import Mathlib}.
It is distributed as an ancillary file of the arXiv version of this paper, together with a \texttt{lean-toolchain} and a \texttt{lakefile.toml} recording that pin, so that the development can be rebuilt exactly.
The closed forms \eqref{eq:closed-forms} and \eqref{eq:closed-forms-D} of $P$, $N$ and $D$ are taken as definitions in the Lean file; that they agree with the integrals \eqref{eq:PND}, and everything else in the paper, is checked only on paper.

\section*{Statements and Declarations}

\subsection*{Competing interests}
The author has no competing interests to declare.

\subsection*{Funding}
No specific funding was received for this work.

\subsection*{Data availability}
No datasets were generated or analyzed. The Lean~4 development of \cref{app:lean} is provided as an ancillary file of the arXiv version.

\subsection*{AI use}
The results in this paper were obtained with the assistance of the large language model Claude Opus 5 (Anthropic), which was used interactively to generate and refine candidate arguments. All statements and proofs were independently verified by the author, who is solely responsible for the content of the paper.

\bibliographystyle{amsplain-doi}
\bibliography{references}

\end{document}